\documentclass[10pt]{article}

\usepackage{etex}
\usepackage[margin={1.2in}]{geometry}
\usepackage[titletoc]{appendix}

\usepackage{amsmath,blkarray,booktabs,bigstrut}

\usepackage{enumerate}

\usepackage{tikz,eso-pic}
\usepackage{tikz-3dplot}
\usepackage{tikz}
\usetikzlibrary{matrix}

\usepackage{eucal}
\usepackage{mathrsfs}
\usepackage{stmaryrd}

\usepackage{hyperref} %% pdflatex

\usepackage{rotating}

\usepackage{longtable}

\usepackage[sc]{mathpazo}
\usepackage{courier}

\usepackage[utf8]{inputenc}
\usepackage[T1]{fontenc}

\usepackage[english]{babel}
\usepackage{blindtext}

\usepackage{amsmath}

\usepackage{microtype}

\usepackage{amsmath}
\usepackage{amssymb}
\usepackage{amsthm}
\usepackage{color}
\usepackage{latexsym,extarrows}

\usepackage[all]{xy}
\usepackage{graphicx}
\usepackage{subcaption}
\usepackage[stable,multiple]{footmisc}

\RequirePackage[font=small,format=plain,labelfont=bf,textfont=it]{caption}
\makeatletter
\newsavebox{\@brx}
\newcommand{\llangle}[1][]{\savebox{\@brx}{\(\m@th{#1\langle}\)}%
  \mathopen{\copy\@brx\kern-0.5\wd\@brx\usebox{\@brx}}}
\newcommand{\rrangle}[1][]{\savebox{\@brx}{\(\m@th{#1\rangle}\)}%
  \mathclose{\copy\@brx\kern-0.5\wd\@brx\usebox{\@brx}}}
\makeatother

\newcommand\xqed[1]{%
	\leavevmode\unskip\penalty9999 \hbox{}\nobreak\hfill
	\quad\hbox{#1}}
\newcommand\xxqed{\xqed{$\triangle$}}
\def\e#1\e{\begin{equation}#1\end{equation}}
\def\ea#1\ea{\begin{align}#1\end{align}}

\theoremstyle{plain}% default
\newtheorem{thm}{Theorem}[section]
\newtheorem{lem}[thm]{Lemma}
\newtheorem{prop}[thm]{Proposition}
\newtheorem{cor}[thm]{Corollary}
\theoremstyle{definition}
\newtheorem{dfn}[thm]{Definition}
\newtheorem{ex}[thm]{Example}
\newtheorem{rem}[thm]{Remark}
\newtheorem{conv}[thm]{Convention}

\newcommand{\LL}{\llangle[\Big]}
\newcommand{\RR}{\rrangle[\Big]}

\newcommand{\E}{{\mathcal E}}

\newcommand{\pii}{2\pi \mathbf{i}}

\def\Xint#1{\mathchoice
{\XXint\displaystyle\textstyle{#1}}%
{\XXint\textstyle\scriptstyle{#1}}%
{\XXint\scriptstyle\scriptscriptstyle{#1}}%
{\XXint\scriptscriptstyle\scriptscriptstyle{#1}}%
\!\int}
\def\XXint#1#2#3{{\setbox0=\hbox{$#1{#2#3}{\int}$}
\vcenter{\hbox{$#2#3$}}\kern-.5\wd0}}

\def\dashint{\Xint-}

\numberwithin{equation}{section}

\makeatletter
\newcommand{\subjclass}[2][2010]{%
  \let\@oldtitle\@title%
  \gdef\@title{\@oldtitle\footnotetext{#1 \emph{Mathematics Subject Classification.} #2}}%
}
\newcommand{\keywords}[1]{%
  \let\@@oldtitle\@title%
  \gdef\@title{\@@oldtitle\footnotetext{\emph{Key words and phrases.} #1.}}%
}
\makeatother
\makeatletter
\let\orig@afterheading\@afterheading
\def\@afterheading{%
   \@afterindenttrue
  \orig@afterheading}
\makeatother

\begin{document}
%%%%%%%%%%%%%%%%%%%%%%%%%%%%%%%%%%%%%%%%%%%%%%%%%%%%%%%%%%%%%%%%%%%%%%%%
%%%%%%%%%%%%%%%%%%%%%%%    Text of paper    %%%%%%%%%%%%%%%%%%%%%%%%%%%%
%%%%%%%%%%%%%%%%%%%%%%%%%%%%%%%%%%%%%%%%%%%%%%%%%%%%%%%%%%%%%%%%%%%%%%%%
%The following title can be changed according to needs.%
\title{\bf 
Gromov-Witten Generating Series of Elliptic Curves 
as
 Configuration Space Integrals 
}
\author{Jie Zhou}
\date{}
%\subjclass{14N35, 11Fxx}
\maketitle

\begin{abstract}

The generating series of Gromov-Witten invariants of elliptic curves
can be expressed in terms of multi-variable elliptic functions
 by works of Bloch-Okounkov and Okounkov-Pandharipande. In this work we give new sum-over-partitions formulas for these generating series
and show that
they are  configuration space integrals  of cohomology classes constructed from sections of  Poincar\'e bundles.
We also discuss their  quasi-elliptic and quasi-modular properties.
\end{abstract}
\setcounter{tocdepth}{2} \tableofcontents

\section{Introduction}

Elliptic curves play an important role in understanding  the modern aspects of geometry.
Due to their rich facets, they  lie at the intersection of various branches and 
provide the simplest examples for the development of many theories and tools.
In this work we focus on the Gromov-Witten (GW) theory of elliptic curves.

\subsection{Review on GW generating series of elliptic curves}

Let $\mathcal{E}$ be an elliptic curve equipped with a K\"ahler class $\omega$. Let
\[
\LL \prod_{i=1}^n \omega\psi_{i}^{\ell_i}\RR_{g,\,n,d}^{ \bullet \E}
\] be the disconnected, {\em stationary}, descendent GW invariant of genus $g$ and degree $d$ with $n$ marked points (satisfying the dimension axiom $2g-2+n=\sum_{i=1}^{n}( \ell_i+1)$),
where $\psi_{i}$ is the descendent cotangent line class attached to the $i$th marking, and the symbol $\bullet$ stands for disconnected counting.
The invariant is called  stationary as the insertions only involve the descendents of $\omega$.
Following \cite{Okounkov:2006}, we denote
\[
\LL \prod_{i=1}^n \omega\psi_{i}^{\ell_i}\RR_{g,n}^{ \bullet \E}
=\sum_{d\geq 0}\LL \prod_{i=1}^n \omega\psi_{i}^{\ell_i}\RR_{g,\,n,d}^{ \bullet \E}\,q^{d}\,,\quad
\LL\omega\psi^{-2}\RR_{0,1}^{\bullet\E}:=1
\]
and define
 the {\em $n$-point generating function}
\begin{equation}
\label{n-point-function}
F_n(z_{1},\cdots, z_{n}; q):=\sum_{\ell_1, \cdots, \ell_n \geq -2}\LL \prod_{i=1}^n \omega\psi_{i}^{\ell_i}\RR_{g,n}^{\bullet\E} \,\prod_{i=1}^{n} z_{i}^{\ell_i+1}\,.
\end{equation}

The celebrated  GW/Hurwitz correspondence \cite[Theorem 5]{Okounkov:2006} 
relates $F_{n}$ to a character as follows.
% constructions on the semi-infinite wedge which is the mathematical formulation for free fermions.
Following \cite{Okounkov:2001, Okounkov:2006, Okounkov:2006b, Eskin:2006},
consider the Clifford algebra generated by the operators  $\psi_{r},\psi^{*}_{s},r,s\in \mathbb{Z}+{1\over 2}$
that satisfy
\[
\psi_{r}\psi^{*}_{s}+\psi^{*}_{s}\psi_{r}=\delta_{r+s}\,.
\]
The semi-infinite wedge $\bigwedge^{\infty\over 2} V=\bigoplus_{m}\bigwedge^{\infty\over 2}_m V$, where $m$ stands for the charge,
is constructed as a representation of the above algebra of operators.
One organizes the above operators into generating series (called fields)
\[
\psi(u)=\sum_{r} \psi_{r}u^{-r}\,,\quad
\psi^{*}(v)=\sum_{s} \psi^{*}_{s}v^{-s}\,.
\]
Define 
\begin{eqnarray*}
\mathbold{T}(u)&=&\sum_{r} u^{r} \psi_{-r}\psi^{*}_{r}=[v^{0}] (\psi(uv)\psi^{*}(v))\,,\\
%\,,\quad |u|>1\,.
\mathbold{T}_{n}(u)&=& \mathbold{T}(u_1)\cdots \mathbold{T}(u_{n}) 
=[v_{1}^{0} v_{2}^0 \cdots v_{n}^0] \left(\psi(u_{1}v_{1})\psi^{*}(v_{1})\cdots \psi(u_{n}v_{n})\psi^{*}(v_{n}) \right)\,,
\end{eqnarray*}
where the operators are chronologically ordered:
\begin{equation}\label{eqnGWordering}
|u_{i}v_{i}|>|v_{i}|>|u_{j}v_{j}|>|v_{j}|\,~\mathrm{for}~ i<j\,.
\end{equation}
Hereafter the notation $[t^k]f$ represents the degree $t^{k}$ coefficient of the formal Laurent series $f$ in $t$.
Taking the corresponding
$q$-trace over $\bigwedge^{\infty\over 2} V$, denoted by $\langle -\rangle$ below, leads to
\begin{equation}\label{eqnFnaszeromode}
\langle \mathbold{T}_{n} \rangle
=[v_{1}^{0} v_{2}^0 \cdots v_{n}^0]\, \big\langle \psi(u_{1}v_{1})\psi^{*}(v_{1})\cdots \psi(u_{n}v_{n})\psi^{*}(v_{n}) \big\rangle
\,.
\end{equation}
Up to normalization which will be detailed below, this is 
the $n$-point function
$F_{n}$ 
\begin{equation}\label{eqndefinfinitewedge}
	F_n(z_{1},\cdots, z_{n};q)=\mathrm{Tr}_{\bigwedge^{\infty\over 2}_0 V}\left(q^{H} \prod_{i=1}^{n}\mathcal{E}_{0}(z_i)\right)\,,\quad
	\mathcal{E}_{0}(z_i)%=[v_i^0]\psi(u_i v_i)\psi^{*}(v_i)
	=\sum_{k=0}^{+\infty}e^{z_i (\lambda_k-k+{1\over 2})}\,,
\end{equation}
where $H$ is the Hamiltonian operator and the states in the semi-infinite wedge $\bigwedge^{\infty\over 2}_0 V$ are labeled by partitions $\lambda=\sum_{k}\lambda_{k}$. The operator
$\mathcal{E}_{0}(z)$ is nothing but the operator $\mathbold{T}(u)$ under $u=e^{z}$.

The correspondence \eqref{eqndefinfinitewedge}
allows one to rewrite the  $n$-point functions $F_n(z_{1},\cdots, z_{n}, q)$ 
in terms of character formulas from \cite{Bloch:2000}
%The GW generating series $F_{n}$ admit
%a beautiful formula from \cite{Bloch:2000}
\begin{equation}\label{eqngeneratingseriesformulaintro}
	F_{n}(z_1,z_2,\cdots ,z_n;q)
	=(q)_{\infty}^{-1}\cdot\sum_{\text{all permutations of} ~z_1,\cdots, z_{n}}  {\det M_{n}(z_{1},z_{2},\cdots, z_{n})\over   \theta(z_{1}+z_{2}+\cdots+ z_{n})}\,,
\end{equation}
where $M_{n}(z_{1},z_{2},\cdots, z_{n})$ is the matrix whose $(i,j), j\neq n$ entries are zero for $i>j+1$ and otherwise are given by
\begin{equation*}
	{\theta^{(j-i+1)}(z_{1}+\cdots+z_{n-j})  \over (j-i+1)! \theta(z_{1}+\cdots+z_{n-j}) }\,,\,
	j\neq n\,,
	\quad\quad
	{\theta^{(n-i+1)}(0)   \over (n-i+1)!}\,,\,
	j=n\,.
\end{equation*}
Here $\theta(z)$ is the unique odd Jacobi theta function satisfying $(\partial_z \theta)(0)=1$
and $(q)_{\infty}:=\prod_{k=1}^{+\infty}(1-q^{k})$ is the Euler function.
These formulas, albeit beautiful, are of considerable combinatorial complexity, 
as they are expressed 
in terms of  determinants of matrices
whose entries are derivatives of the Jacobi theta function with different sets of arguments.

\subsection{Configuration space integrals and sum over partitions}
\label{subsecGWconfigurationspaceintegrals}

The  present work attempts to provide a more conceptual 
understanding of $F_{n}$ in terms of configuration space integrals,
as well as deriving new formulas for them.
Instead of using \eqref{eqndefinfinitewedge}, our starting point is \eqref{eqnFnaszeromode}.

We shall focus on a scaled version of $F_{n}$ given by
\begin{equation}\label{eqndefTn}
	T_{n}(z_{1},\cdots, z_{n};q):=\theta(\sum_{i=1}^{n} z_{i})\cdot (q)_{\infty}F_{n}(z_{1},\cdots, z_{n};q)\,.
\end{equation}
% where $\theta(z)$ is a normalized odd Jacobi theta function satisfying $(\partial_z \theta)(0)=1$
% and $(q)_{\infty}:=\prod_{k=1}^{+\infty}(1-q^{k})$ is the Euler function.
% This is \eqref{eqnFnaszeromode} up to normalization that will be detailed below.
Note that the factor $(q)_{\infty}^{-1}$ is nothing but the contribution of the 
unstable connected component with genus one domain and no marked points.
Thus by multiplying by $(q)_{\infty}$, one is
cancelling out  this contribution.
See \cite{Zagier:2016partitions} (also \cite{Bloch:2000, Okounkov:2001, Milas:2003}) 
for more about the enumerative content of the function $T_{n}$.\\

Let $E=\mathbb{C}/\Lambda_{\tau}, \Lambda_{\tau}=\mathbb{Z}\pii\oplus \mathbb{Z}\pii\tau$ 
be an elliptic curve.
We fix the additive coordinate $z$ on the universal cover $\widetilde{E}=\mathbb{C}\rightarrow E$ and
multiplicative coordinate $u=e^{z}$ on the cover $\mathbb{C}^{*}=\mathbb{C}/\pii \mathbb{Z}\rightarrow E$.
The origin on $E$ is taken to be the image of $0\in \mathbb{C}$ under the covering map $\mathbb{C}\rightarrow E$.
%We choose the fundamental domain fore the latter cover to be $|q|\leq |u|<1$.
We also fix a weak Torelli marking $\{A,B\}$, namely a symplectic frame for $H_{1}(E,\mathbb{Z})$.
The class $A$ is the image of the segment connecting $\pii \tau$ and $\pii \tau+\pii$ on $\mathbb{C}$.

For any $n\geq 1$, let $\mathsf{Conf}_{n}(E)$ be the configuration space of $n$ ordered points on $E$.
We shall construct a relative cohomology class $\varphi_{n}^{\mathrm{GW}}$ %from Poincar\'e bundles 
and a relative homology class $\gamma_{n}^{\mathrm{GW}}$
on 
a relative
configuration space $\widetilde{E}^{n}\times \mathsf{Conf}_{n}(E)\rightarrow \widetilde{E}^{n}$, and then
 consider their pairing
 \begin{equation}\label{eqnparingofGWclasses}
\Big \langle ~\gamma_{n}^{\mathrm{GW}}\,, ~ \varphi_{n}^{\mathrm{GW}} ~\Big\rangle
 \end{equation}
 which  concretely  is the fiberwise integration.
 See Section \ref{secconstructionofGWclasses} for their definitions.

Our first main result is the following sum-over-partitions formula.

\begin{thm}\label{thmTngwintermsofBintro}(=Proposition \ref{propGWpairing=bindpairing}+Theorem \ref{thmTngwintermsofB})
	Let $\mathbold{B}_{n}(z)$ be the 
	complete Bell polynomial $\mathbold{B}_{n}(z)$ in the variables $\mathcal{E}_{m}^{*}(z),m\geq 1$ which are the Eisenstein-Kronecker series\footnote{A quick definition of the complete Bell polynomial in the variables $x_1,\cdots, x_{m},\cdots$ is given by $\mathbold{B}_{m}(x_1,\cdots,x_{m})=m!\cdot [t^{m}]\exp(\sum_{k=1}^{\infty}{t^{k}\over k!}x_{k} )$. The quantity $\mathcal{E}_{m}^{*}$ can be alternatively written as $(\ln\theta)^{(m)}+2\mathbb{G}_{m}$, where $2\mathbb{G}_{m}$ is the ordinary Eisenstein series defined as the summation of $(m-1)!\cdot\lambda^{-m}$ over $\lambda\in \Lambda_{\tau}-\{(0,0)\}$.}
	\[
	\mathcal{E}_{m}^{*}(z):=
	(-1)^{m-1}(m-1)!
	\left({1\over z^m}+\sum_{\lambda\in \Lambda_{\tau}-\{(0,0)\}}( {1\over (z+\lambda)^{m}}+{(-1)^{m-1}}{1\over \lambda^{m}})
	\right)
	\,,\quad m\geq 1\,.
	\]
	Let the linear coordinates on the $n$ components of $\widetilde{E}^{n}$ in $\widetilde{E}^{n}\times \mathsf{Conf}_{n}(E)$ be $z_1,\cdots, z_{n}$ respectively.
		Denote $\Pi_{[n]\setminus \{j\}}$ to be the set of partitions of $[n]\setminus \{j\}$ consisting of 
	elements of the form $\pi=\{\pi_1,\cdots, \pi_{\ell}\}$.
	 Then one has
	\begin{equation}\label{eqnTngwpureweightstructureintro}
\Big \langle ~\gamma_{n}^{\mathrm{GW}}\,, ~ \varphi_{n}^{\mathrm{GW}} ~\Big\rangle=
		\sum_{j\in [n]}\sum_{\pi\in \Pi_{[n]\setminus \{j\}}}
		{\mathbold{B}_{{\pi}}\over |\pi|}\,,
		\quad
		{\mathbold{B}_{{\pi}}\over |\pi|}:=\prod_{k=1}^{\ell} {\mathbold{B}_{|\pi_{k}|}(\sum_{i\in \pi_{k}}z_{i})\over |\pi_{k}|}\,.
	\end{equation}
\end{thm}
The proof
is obtained by first realizing $\langle \gamma_{n}^{\mathrm{GW}}\,, ~ \varphi_{n}^{\mathrm{GW}} \rangle$ as concrete integrations (see \eqref{eqnpairingisintegration}).
Then we 
compute it based on the explicit  expression of $\varphi_{n}^{\mathrm{GW}}$
in terms of multi-variable elliptic functions.
In deriving Theorem \ref{thmTngwintermsofBintro}, we also
obtain some purely combinatorial
results such as  Theorem \ref{thmquasiellipticitygba} and Theorem \ref{thmsumoverpartitions} 
which seem to be of independent interest. These are in parallel with the sum-over-partial-flags formulas in
\cite{Bloch:2000}.

Based on Theorem \ref{thmTngwintermsofBintro}, we then relate the GW generating series $T_{n}$ 
in \eqref{eqndefTn}
with the paring \eqref{eqnparingofGWclasses} between algebro-geometric quantities.

\begin{thm}\label{thmTn=Acycleintegralintro}(=Theorem \ref{thmTn=Acycleintegral})
	Following the notation above,
	let the linear coordinates on the $n$ components of $\widetilde{E}^{n}$ in $\widetilde{E}^{n}\times \mathrm{Conf}_{n}(E)$ be $z_1,\cdots, z_{n}$ respectively, and $q=e^{\pii\tau}$.
	Then
	\begin{equation}\label{eqnTnintermsofintergralintro}
	T_{n}(z_1,\cdots, z_{n};q)
	=\Big \langle ~\gamma_{n}^{\mathrm{GW}}\,, ~ \varphi_{n}^{\mathrm{GW}} ~\Big\rangle \,.
	\end{equation}
\end{thm}

This result is a reformulation of \eqref{eqnFnaszeromode} proved in \cite{Okounkov:2006, Okounkov:2006b} in its equivalent form \eqref{eqndefinfinitewedge}.
%, see also \cite{Okounkov:2001, Eskin:2006}.
A practical way of evaluating \eqref{eqnFnaszeromode}
is also provided in \cite[Section 3.3]{Eskin:2006} (see also \cite{Milas:2003}).
However, we found  in the literature
no systematical studies on their evaluations 
besides reducing them to \eqref{eqndefinfinitewedge}.
We therefore decide to provide a second proof of this result by taking advantage of the fact that 
the quantity
$\langle \psi(u_{1}v_{1})\psi^{*}(v_{1})\cdots \psi(u_{n}v_{n})\psi^{*}(v_{n}) \rangle$ in \eqref{eqnFnaszeromode}
is purely algebro-geometric and explicit.
 Our proof of Theorem \ref{thmTn=Acycleintegralintro} proceeds by showing that the right hand side of  \eqref{eqnTnintermsofintergralintro}, which by
 Theorem \ref{thmTngwintermsofBintro} is given by the sum-over-paritions expression in  	 
 \eqref{eqnTngwpureweightstructureintro}, satisfy
 the same difference equations and singular behavior conditions for $T_{n}$ given in \cite{Bloch:2000} (see also \cite{Okounkov:2001}). The latter conditions  are enough to determine them uniquely.

The above configuration space integral result
fits in the general mirror symmetry philosophy that ``GW generating series = period integrals''.
However, it 
 is different from the original and more familiar formulation \cite{Dijkgraaf:1995} (see also \cite{Roth:2009mirror, Boehm:2014tropical, Goujard:2016counting, Li:2011BCOV, Li:2011mi, Li:2016vertex, Li:2020regularized}) for mirror symmetry of elliptic curves, which relates an individual Laurent coefficient
in \eqref{n-point-function} to a sum of period integrals of cohomology classes labeled by graphs.

The integration formula in Theorem \ref{thmTn=Acycleintegralintro}
then provides a new perspective in
 investigating properties of $F_{n}$
 using the tools of ordered $A$-cycle integrals and regularized integrals
developed in \cite{Li:2020regularized, Li:2022regularized}.
We demonstrate this by studying the elliptic and modular completion of $F_{n}$.\\

The structural results in Theorem \ref{thmTngwintermsofBintro} and Theorem \ref{thmTn=Acycleintegralintro} can
make it convenient to study further properties of GW
generating series.
We hope these particularly organized forms can
% shed some light 
provide some insights in exploring
the relation between the rich structures (such as quasi-modularity) of GW generating series and stratification of configuration spaces.

\subsection*{Structure of the paper}
In Section \ref{secconstructionofGWclasses}
we construct the relative cohomology  $\varphi_{n}^{\mathrm{GW}}$ and homology class $\gamma_{n}^{\mathrm{GW}}$.
In Section \ref{secorderedAcycleintegralofC2n} we derive concrete formulas for \eqref{eqnparingofGWclasses} not only for the class
 $\gamma_{n}^{\mathrm{GW}}$ but also 
 for other natural cycle classes $\gamma_{n}$.
In Section \ref{secquasielliptic} we specialize to the class $\gamma_{n}^{\mathrm{GW}}$
and prove the main results Theorem \ref{thmTngwintermsofBintro} and Theorem \ref{thmTn=Acycleintegralintro}
by assuming Theorem \ref{thmquasiellipticitygba}, as well as prove some other results regarding quasi-ellipticity and quasi-modularity of $T_{n}$.
We also study properties of $T_{n}$ using the tool of 
regularized integrals.
In Appendix \ref{appendixellipticfunctions} we review a few basic properties of the Jacobi theta functions and prove
Theorem \ref{thmquasiellipticitygba}.
We relegate some involved computations on $n$-point functions to Appendix \ref{appendixevaluations}.

\subsection*{Acknowledgement}

 The author would like to thank Si Li and Yefeng Shen for fruitful collaborations on related topics, thank Dingxin Zhang and Zijun Zhou for helpful communications and discussions.
The author is supported by the
national key research and development
 program of China (No. 2022YFA1007100) and the Young overseas high-level talents introduction plan of China.

\subsubsection*{Notation and conventions}

\begin{itemize}
	\item 
	For notational simplicity, we shall often 
	suppress the arguments of a function whenever they are clear from the surrounding texts.
	\item
	Throughout this work, when we need to keep track of the variables in a construction such as a function $f_{n}$ that are taken from some finite set $J$ or sequence $\mathbold{K}$
	of variables  with cardinality $n$, we use the notation
	$f_{J}$ or $f_{\mathbold{K}}$.
	Correspondingly, we shall
	use the notation $S$ abusively for both the set $S$ and its cardinality $|S|$.	
	\item \label{conventionuS}
	With the map $\exp: \mathbb{C}\rightarrow \mathbb{C}^{*}$ understood, 
	we use the additive coordinate $z$ on $\mathbb{C}$ and the multiplicative coordinate $u=e^{z}$ on $\mathbb{C}^{*}$ interchangably.
	
	For a nonempty set $S\subseteq [n]:=(1,2,\cdots,n)$, we use the notation
	\[
	z_{S}=\sum_{i\in S} z_{i}\,,\quad u_{S}=\prod_{i\in S}u_{i}\,.
	\]
	\item
	Fix a space $X$ and a nonempty finite set $S$, let $X_{a}$ be a copy of $X$ labeled by $a\in S$ and	
	\[
	X_{S}=\prod_{a\in S}X_{a}\,,\quad \mathsf{Conf}_{S}(X)=
	\big\{(p_{a})_{a\in S}\in \prod_{a\in S}X_{a}~|~ p_{a}\neq p_{b}~\text{for}~ a\neq b\in S
	\big\}\,.
	\]
	The complement of $\mathrm{Conf}_{S}(X)$ in $X_{S}$ is called the big diagonal in $X_{S}$.
	\item
	We denote collectively the points $P_{i},1\leq i\leq n$ and $Q_{j},1\leq j\leq n$ by $P,Q$, respectively.
	Similarly  we apply the same convention to denote a sequence of coordinates, e.g., $z=(z_1,\cdots, z_n)$.
\item 
The notation $[t^k]f$ represents the degree $t^{k}$ coefficient of the formal Laurent series $f$ in $t$.	
	
\end{itemize}

\section{Construction of the pairing $
	\langle \gamma_{n}^{\mathrm{GW}}\,, ~ \varphi_{n}^{\mathrm{GW}} \rangle$ }
\label{secconstructionofGWclasses}

In this section, we shall construct a relative cohomology class $ \varphi_{n}^{\mathrm{GW}}$
and a relative homology class $\gamma_{n}^{\mathrm{GW}}$, then use the relative pairing between singular cohomology and singular homology to define the desired pairing  $
\langle \gamma_{n}^{\mathrm{GW}}\,, ~ \varphi_{n}^{\mathrm{GW}} \rangle$ in \eqref{eqnparingofGWclasses}. These constructions are designed to ``geometrize" \eqref{eqnFnaszeromode}.
\\

Let  $\kappa \in \mathsf{Pic}^{(g-1)}(E)$ be the unique odd theta characteristic $\kappa^{\otimes 2}=K_{E}$ which is also the unique one
satisfying $h^{0}(E, \kappa)\neq 0$. It is the Riemann constant for the corresponding odd theta divisor $\Theta$ on the Jacobian variety $\mathsf{Jac}(E)\cong 
 \mathsf{Pic}^{(0)}(E)$.
The line bundle $\mathcal{O}_{\mathsf{Jac}(E)}([\Theta])$
has $\theta$ as the unique holomorphic section that vanishes along $\Theta$, normalized such that $(\partial_z \theta)(0)=1$.

Let
$\eta\in \mathsf{Pic}^{(g-1)}(E)$ and $c$  be the corresponding shift
$\eta\otimes \kappa^{-1}$ on $\mathrm{Jac}(E)$. 
Consider the map
\[
\delta\times \mathrm{id}: E\times E\rightarrow J\times \Delta\,,\quad (P,Q)\mapsto (P-Q,Q)\,,
\]
where $J$ is a copy of $E$ that is isomorphic to $\mathsf{Jac}(E)$ via the Abel-Jacobi map with respect to the holomorphic volume form $dz \in H^{0}(E,K_{E})$.
Under the inclusion map  
\[\Delta\rightarrow J\times \Delta\xrightarrow{(\delta\times \mathrm{id})^{-1}} E\times E\,,\]
the space $\Delta$ can be regarded as the diagonal in $E\times E$. 
Its $A$-cycle class $A\in H_{1}(\Delta)$ is then mapped to $A\otimes 1+1\otimes A\in H_{1}(E\times E)$ under the K\"unneth decomposition.
Denote $\check{\eta}=\eta^{-1}\otimes K_{E}$ to be the Serre dual of $\eta$.
It is well known that (see e.g., \cite{Raina:1989, Birkenhake:2013complex})
\begin{equation}\label{eqndeltaLc}
\delta^{*}L_{c}\cong (\check{\eta}\boxtimes \eta)\otimes \mathcal{O}_{E\times E } ([\Delta])\,,\quad
L_{c}:=\mathcal{O}_{J}([\Theta-c])\,,
\end{equation}
here $\Theta-c$ is the divisor obtained by translating $\Theta$ by $-c$.
Furthermore,  if $c\neq 0$, then up to constant 
there is a unique section of
$(\check{\eta} \boxtimes \eta)(\Delta)$. The Szeg\"o kernel $S_{c}(P-Q)$ is such a section which is uniquely
fixed by specifying its residue along the diagonal $\Delta$.
%
%More instrinsically, the prime form is the isomorphism
%\[
% (\check{\eta}\boxtimes \eta)(\Delta)
%	\otimes \mathcal{O}_{\mathsf{Pic}^{(0)}(E)}(\Theta)|_{c}\rightarrow 	\delta^{*}L_{c}
%	\,.
%\]

Now take $2n$ copies of the elliptic curve labeled by $[n]=(1,2,\cdots, n), [\bar{n}]=(\bar{1},\cdots, \bar{n})$.
Denote also 
$E_{[n]}=E_{1}\times E_{2}\times \cdots \times E_{n}$.
Let 
\begin{equation}\label{eqndeltaidmap}
	\mathbold{\delta}\times \mathbold{id}:E_{[n]}\times E_{[\bar{n}]}\rightarrow \mathsf{J}_{[n]}\times \mathsf{\Delta}_{[n]}
\end{equation} be the Cartesian product of the maps
\begin{eqnarray*}\delta_{k}\times \mathrm{id}_{\bar{k}}: E_{k}\times E_{\bar{k}}&\rightarrow& J_{k}\times \Delta_{k}\,,\\
(P_{k},Q_{k})&\mapsto& (P_{k}-Q_{k},Q_{k})\,,\quad  k=1,2,\cdots, n\,.
\end{eqnarray*}
Denote the linear coordinate on the $k$th component of $\mathsf{J}_{[n]}, \mathsf{\Delta}_{[n]} $
by $z_{k},w_{k}$, respectively. Let $u_{k}=e^{z_{k}},v_{k}=e^{w_{k}}$.

Let $J^{\circ}$ be the punctured Jacobian with the origin $0$ removed, and the corresponding Cartesian product be $\mathsf{J}_{[n]}^{\circ}\subseteq \mathsf{J}_{[n]}$.
Let also $\mathsf{Conf}_{[n]}(\Delta)$ be the complement of the big diagonal in $\mathsf{\Delta}_{[n]}\cong \Delta^{n}$.
For later use, we introduce the various divisorson the above two spaces $E_{[n]}\times E_{[\bar{n}]}, \mathsf{J}_{[n]}\times \mathsf{\Delta}_{[n]}$ as shown in Table \ref{table-divisors}. We shall use the notation  $|\mathsf{D}|=\mathsf{D}_{0}+\mathsf{D}_{\infty}$ to denote the support of $\mathsf{D}=-\mathsf{D}_{0}+\mathsf{D}_{\infty}$.
\begin{table}[h]
	\centering
	\caption{Divisors on $E_{[n]}\times E_{[\bar{n}]}$ and $\mathsf{J}_{[n]}\times \mathsf{\Delta}_{[n]}$ that are related by the map $\mathbold{\delta}\times \mathbold{id}$\,.}
	\label{table-divisors}
	\renewcommand{\arraystretch}{1.5} % single-spacing 1.2 is visually more
	% appealing than 1
	%\begin{displaymath}
	\begin{tabular}{c|c}
		\hline
		divisor on $E_{[n]}\times E_{[\bar{n}]}$ &  	divisor on $\mathsf{J}_{[n]}\times \mathsf{\Delta}_{[n]}$ \\
		\hline
		$	\mathbold{\Delta}=-\sum_{i<j}\Delta_{ij}-\sum_{i<j}\Delta_{\bar{i}\bar{j}}
		+\sum_{i,j}\Delta_{i\bar{j}}$  &  	$\mathsf{D}=(\mathbold{\delta}\times \mathbold{id})^{-1}(\mathbold{\Delta})$\\
		$+\sum_{i<j}\Delta_{ij}+\sum_{i<j}\Delta_{\bar{i}\bar{j}}$
		& $\mathsf{D}_{0}$\\	
		$+\sum_{i<j}\Delta_{\bar{i}\bar{j}}$
		& $\mathsf{J}_{[n]}\times (\mathsf{\Delta}_{[n]}-\mathsf{Conf}_{[n]}(\Delta))$
		\\	
		$+\sum_{i,j}\Delta_{i\bar{j}}$
		& $\mathsf{D}_{\infty}$\\
		$+\sum_{i}\Delta_{i\bar{i}}$ &  $(\mathsf{J}_{[n]}-\mathsf{J}_{[n]}^{\circ})\times \mathsf{\Delta}_{[n]}$\\
		$+\sum_{i,j:\,i\neq j}\Delta_{i\bar{j}}$
		& $\mathsf{D}_{\infty}^{\circ}$\\
				\hline
	\end{tabular}
\end{table}

\subsection{Construction of $\Phi_{n}^{\mathrm{GW}}$  }

Using the map \eqref{eqndeltaidmap} we have the following maps in 	Figure \ref{figure:figconstructionofclasses}.
\begin{figure}[h]
	\centering
	\[
	\xymatrix{
				&	E_{[n]}\times E_{[\bar{n}]}\times \mathsf{Pic}^{(0)}(E)\ar[r]^{\mathbold{\delta}\times \mathbold{id}\times \mathrm{id}}   & \quad
		\mathsf{J}_{[n]}\times \mathsf{\Delta}_{[n]}\times \mathsf{Pic}^{(0)}(E)\ar[ddl]_{\pi_{1}}\ar[dd]^{\pi_{2}}\ar[ddr]^{\pi_{3}}&
		\ar[l]
		%_{\mathbold{j}\times \mathbold{id}\times \mathbold{id}}
		\mathsf{J}_{[n]}^{\circ}\times \mathsf{\Delta}_{[n]}\times \mathsf{Pic}^{(0)}(E)\\
		&\mathcal{F}\ar[d]&& \mathcal{K}\ar[d]\\
		%	&\mathcal{L}_{c}\ar[d]&&\mathcal{K}\ar[d]&&&&\\
		&	\mathsf{J}_{[n]}&	\mathsf{Conf}_{[n]}(\Delta)& 	\mathsf{Pic}^{(0)}(E) &
	}
	\]	
	\caption{Maps between spaces.}
	\label{figure:figconstructionofclasses}
\end{figure}
Here $\pi_{k},k=1,2,3$ are the projections from $\mathsf{J}_{[n]}\times\mathsf{\Delta}_{[n]}\times \mathsf{Pic}^{(0)}(E)$ to its three components, respectively.
The sheaves $\mathcal{F},\mathcal{K}$ will be defined later (see \eqref{eqnsheafFasresidueofPoincare}).

Let $\boxtimes$ be the external tensor product of vector bundles.
For example, $\mathcal{F}_1\boxtimes \mathcal{F}_2\boxtimes \mathcal{F}_3:=\pi_{1}^{*}\mathcal{F}_{1}\otimes \pi_{2}^{*}\mathcal{F}_2\otimes \pi_{3}^{*}\mathcal{F}_3$ on the space $\mathsf{J}_{[n]}\times\mathsf{\Delta}_{[n]}\times \mathsf{Pic}^{(0)}(E) $.
\begin{dfn}\label{dfnPincaresheaf}
	Identify $\mathsf{Pic}^{(0)}(E)$ with $\mathsf{Pic}^{(0)}(J)$ via the map $c\mapsto L_{c}L_{0}^{-1}$. let 
	$\mathcal{P}$ be the Poincar\'e sheaf on $J\times \mathsf{Pic}^{(0)}(E) $, normalized 
	such that  $\mathcal{P}|_{\{0\}\times \mathsf{Pic}^{(0)}(E)}$ is trivial.
	Correspondingly let the sheaf 
	$\mathcal{P}_{[n]}$ on $\mathsf{J}_{[n]}\times \mathsf{Pic}^{(0)}(E)$
	be the tensor product of pull backs of the Poincar\'e sheaves on
	$J_{k}\times \mathsf{Pic}^{(0)}(E),k\in[n] $.
	Denote
	\begin{eqnarray*}
		\mathsf{D}_{\Theta}&:=&\pi_{3}^{-1} \Theta
		\subseteq \mathsf{J}_{[n]}\times \mathsf{\Delta}_{[n]}\times \mathsf{Pic}^{(0)}(E)\,,\\
		\mathbold{\Delta}_{\Theta}&:=&(\mathbold{\delta}\times \mathbold{id}\times \mathrm{id} )^{-1}\mathsf{D}_{\Theta} 	\subseteq E_{[n]}\times E_{[\bar{n}]}\times \mathsf{Pic}^{(0)}(E) \,.
	\end{eqnarray*}
Let
$\pi_{ij}$ be the projection to the product of the $i,j$ components.
	Define the following sheaves
	\begin{equation}\label{eqnLpioncare}
		\mathcal{L}=\mathcal{P}_{[n]}\boxtimes \mathcal{O}_{\mathsf{\Delta}_{[n]}}(\pi_{12}^{-1}\mathsf{D}) \,,\quad
		\mathcal{M}
		=(\mathbold{\delta}\times \mathbold{id}\times \mathrm{id} )^{*} 	\mathcal{L} \,.
	\end{equation}
\end{dfn}
\begin{rem}
Let $m: J\times \mathsf{Pic}^{(0)}(E)\rightarrow J$ be the addition map on $J\cong \mathsf{Pic}^{(0)}(E)$.
Then explicitly one has
\[
\mathcal{P}\cong \mathcal{O}_{J\times  \mathsf{Pic}^{(0)}(E)}(-[0\times \mathsf{Pic}^{(0)}(E)]
-[J\times \Theta]+[\mathrm{ker}\,m])\,.
\]
A meromorphic section is given by the Szeg\"o kernel mentioned above.
\end{rem}
%Let the projections from $J^{\circ}_{[n]}\times \mathsf{Pic}^{(0)}(E)$
% to its two components be $p_1,p_2$ respectively.
By the construction in Definition \ref{dfnPincaresheaf}, one has 
$\mathcal{P}|_{J\times \{c\}}=L_{c}L_{0}^{-1}$.
	From \eqref{eqndeltaLc},
it is direct to check that when $c\neq 0$ 
\begin{eqnarray}\label{eqnPoincarebundleisoinducing}
	\mathcal{L}|_{\mathsf{J}_{[n]}\times \mathsf{\Delta}_{[n]}\times \{c\}}&\cong&  \boxtimes_{[n]}L_{c}L_0^{-1}
	\boxtimes
	\mathcal{O}_{\mathsf{\Delta}_{[n]}} 
(\mathsf{D})\,,\nonumber\\
	\mathcal{M}|_{E_{[n]}\times E_{[\bar{n}]}\times \{c\}}&\cong& \boxtimes_{[n]}\check{\eta}  \boxtimes \boxtimes_{[n]}\eta(\mathbold{\Delta})\,.
\end{eqnarray}
The sheaf  $\mathcal{M}$
then gives the data $\{\mathcal{M}_{c}\}_{c\in \mathsf{Pic}^{(0)}(E)} $ on $\mathsf{Pic}^{(0)}(E)$ for the family $\pi_{3}: E_{[n]}\times E_{[\bar{n}]}\times \mathsf{Pic}^{(0)}(E)\rightarrow \mathsf{Pic}^{(0)}(E)$.
It is shown in  \cite{Raina:1989}  that when $c\neq 0$,
up to scalar $\mathcal{M}_{c}$ has a unique meromorphic section:
$
h^{0}(E_{[n]}\times E_{[\bar{n}]},
\mathcal{M}_{c})=1
$.
They glue to the following section  that has a simple pole at the divisor $\mathbold{\Delta}_{\Theta}$
\begin{equation}\label{eqnintrinsicsection}
	C_{2n}	
	\in H^{0}( E_{[n]}\times E_{[\bar{n}]}\times \mathsf{Pic}^{(0)}(E),\mathcal{M}(\mathbold{\Delta}_{\Theta}))\,.
	%H^0 R^{0}\pi_{3,*}\mathcal{M}(\mathbold{\Delta}_{\Theta})\,,
\end{equation}
This particular section is the quantity 
$
\big\langle \psi(u_{1}v_{1})\psi^{*}(v_{1})\cdots \psi(u_{n}v_{n})\psi^{*}(v_{n}) \big\rangle$
in \eqref{eqnFnaszeromode},
%denoted by $\langle \psi(u_{1}v_{1})\psi^{*}(v_{1})\cdots \psi(u_{n}v_{n})\psi^{*}(v_{n}) \rangle$, 
see \eqref{eqndefofC2n}
below for its concrete expression.
We  take 
its residue at $\mathbold{\Delta}_{\Theta}=E_{[n]}\times E_{[\bar{n}]}\times \Theta\cong E_{[n]}\times E_{[\bar{n}]}$
\begin{equation}\label{eqnPoincareresidueofC2n}
	\varpi:=\mathrm{Res}_{\mathbold{\Delta}_{\Theta}} (C_{2n})\in H^{0}(E_{[n]}\times E_{[\bar{n}]},\mathcal{M}(\mathbold{\Delta}_{\Theta})|_{\mathbold{\Delta}_{\Theta}})\,.
\end{equation}
Here we have used the fact that the canonical sheaves of 
$E_{[n]}\times E_{[\bar{n}]}\times \mathsf{Pic}^{(0)}(E),\mathbold{\Delta}_{\Theta}$ are trivial, equipped with natural trivializations induced from $dz \in H^{0}(E,K_{E})$.

We next construct sheaves on the spaces in 	Figure \ref{figure:figconstructionofclasses} and  morphisms between them.

\begin{itemize}
		\item

	Let $\mathcal{F}$ be the following 
  sheaf on $\mathsf{J}_{[n]}$ 
 %and $\mathcal{F}^{\circ}$ the corresponding restriction on $\mathsf{J}_{[n]}^{\circ}$ 
	\begin{equation}\label{eqnsheafFasresidueofPoincare}
	\mathcal{F}:=	\mathcal{P}_{[n]}(\mathsf{J}_{[n]}\times \Theta)|_{\mathsf{J}_{[n]}\times \Theta}\,.
	%\,,\quad	\mathcal{F}^{\circ}:=	\mathcal{P}^{\circ}_{[n]}(\mathsf{J}^{\circ}_{[n]}\times \Theta)|_{\mathsf{J}^{\circ}_{[n]}\times \Theta}
	\end{equation}
	Let also
$\mathcal{K}$ be the canonical sheaf on $\mathsf{\Delta}_{[n]}$.
	From \eqref{eqnPoincarebundleisoinducing} one has
	\begin{equation}\label{eqndeltapullback}
		\left((\mathbold{\delta}\times \mathbold{id})^{-1}\right)^{*}\mathcal{M}(\mathbold{\Delta}_{\Theta})|_{\mathbold{\Delta}_{\Theta}}
		\cong
		\mathcal{L}(\mathsf{D}_{\Theta})|_{\mathsf{D}_{\Theta}}
	%	\cong 	\mathcal{F}		\boxtimes( \boxtimes_{[\bar{n}]}K_{\Delta})(\mathsf{D})
		\cong  (\mathcal{F}\boxtimes \mathcal{K})(\mathsf{D})\,.
	\end{equation}
% Here  we have used the ordering on $[n]$ to fix an orientation class on $\mathsf{\Delta}_{[n]}$ which then leads to the isomorphism $\boxtimes_{[\bar{n}]}K_{\Delta}\cong \mathcal{K}$.

\item 
Let $\mathbold{j}: \mathsf{J}_{[n]}\times \mathsf{\Delta}_{[n]}-\mathsf{D}_{\infty}\rightarrow \mathsf{J}_{[n]}\times \mathsf{\Delta}_{[n]}$ be the open embedding.
Applying adjunction to
\[
(\mathcal{F}\boxtimes \mathcal{K})(\mathsf{D}_{\infty})\rightarrow
  \mathbold{j}_{*} \mathbold{j}^{*}(\mathcal{F}\boxtimes \mathcal{K})(\mathsf{D}_{\infty})\,,
\]
one obtains a map
\begin{equation}\label{eqnj1j2pullback}
	\mathbold{j}^{*}
	: H^{0} (\mathsf{J}_{[n]}\times \mathsf{\Delta}_{[n]}, (\mathcal{F}\boxtimes \mathcal{K})(\mathsf{D}_{\infty}))
	\hookrightarrow 	H^{0}(\mathsf{J}_{[n]}\times \mathsf{\Delta}_{[n]}-\mathsf{D}_{\infty},  
\mathcal{F}\boxtimes \mathcal{K})\,.
\end{equation}

\item 
Let $\pi_{1}^{\circ}: \mathsf{J}_{[n]}\times \mathsf{\Delta}_{[n]}-\mathsf{D}_{\infty}
\rightarrow \mathsf{J}_{[n]}^{\circ}$ be the projection induced by the open embedding $\mathbold{j}$
and $\pi_{1}$.
Similarly use the super- and sub-scripts $\circ$ in the notation to indicate the pull back constructions induced by the open embedding $\mathbold{j}$.
%%%%%

Applying the relative de Rham resolution, we 
obtain
\[
\pi_{1,\circ}^{-1}\mathcal{F}\hookrightarrow \mathcal{G}^{\bullet}\otimes
\pi_{1,\circ}^{*}\mathcal{F}^{\circ}:= (\Omega^{\bullet}_{(\mathsf{J}_{[n]}\times \mathsf{\Delta}_{[n]}-\mathsf{D}_{\infty})/{\mathsf{J}}^{\circ}_{[n]}}\otimes_{\mathcal{O}_{\mathsf{J}_{[n]}\times \mathsf{\Delta}_{[n]}-\mathsf{D}_{\infty}}}
\pi_{1,\circ}^{*}\mathcal{F}^{\circ}, d_{\mathrm{dR}}\otimes 
\mathrm{id})\,.
\] 
Applying the functor $R^{n}\pi_{1,*}^{\circ}(-)$ and adjunction, one obtains the isomorphism
\begin{equation*}
 R^{n}\pi_{1,*}^{\circ}(\mathcal{G}^{\bullet}\otimes
	\pi_{1,\circ}^{*}\mathcal{F}^{\circ} )\xrightarrow{\cong }
	R^{n}\pi_{1,*}^{\circ}\pi_{\circ}^{-1}\mathcal{F}^{\circ}\xrightarrow{\cong }
	\mathcal{F}^{\circ}\otimes_{\mathbb{C}} R^{n}\pi_{1,*}^{\circ}\mathbb{C}\,.
\end{equation*}
Taking $H^{0}(-)$ gives a map
\begin{equation}\label{eqnpushforwardcohomoloy}
	\mathfrak{r}: H^{0}(\mathsf{J}^{\circ}_{[n]}, R^{n}\pi_{1,*}^{\circ}(\mathcal{G}^{\bullet}\otimes
	\pi_{1,\circ}^{*}\mathcal{F}^{\circ} ))\rightarrow 
	H^{0}(\mathsf{J}^{\circ}_{[n]},
	\mathcal{F}^{\circ}\otimes_{\mathbb{C}} R^{n}\pi_{1,*}^{\circ}\mathbb{C})\,.
\end{equation}

\item
The standard spectral sequence machinery gives a map
\begin{equation*}
	\pi_{1,*}^{\circ}(
	\mathcal{F}^{\circ}\boxtimes \mathcal{K})
	\rightarrow 
	R^{n}\pi_{1,*}^{\circ}(\mathcal{G}^{\bullet}\otimes
	\pi_{1,\circ}^{*}\mathcal{F}^{\circ} )
\end{equation*}
which after taking $H^{0}(-)$ gives a map 
\begin{equation}\label{eqnspectralsequence}
	\mathfrak{s}: H^{0}(\mathsf{J}_{[n]}\times \mathsf{\Delta}_{[n]}-\mathsf{D}_{\infty},  
	\mathcal{F}^{\circ}\boxtimes \mathcal{K})
	\rightarrow 
H^{0}(\mathsf{J}^{\circ}_{[n]}, R^{n}\pi_{1,*}^{\circ}(\mathcal{G}^{\bullet}\otimes
	\pi_{1,\circ}^{*}\mathcal{F}^{\circ} ))\,.
\end{equation}

\end{itemize}

Combining \eqref{eqndeltapullback}, 
\eqref{eqnj1j2pullback},
 \eqref{eqnspectralsequence} and \eqref{eqnpushforwardcohomoloy},
one obtains a map
	\begin{equation}\label{eqncompositionmapgivingGWcohomologyclass}
 \mathfrak{r}\circ\mathfrak{s}\left(	\mathbold{j}^{*}
((\mathbold{\delta}\times \mathbold{id})^{-1})^{*}- \right):
 H^{0}(E_{[n]}\times E_{[\bar{n}]},\mathcal{M}(\mathbold{\Delta}_{\Theta})|_{\mathbold{\Delta}_{\Theta}})\rightarrow 
 H^{0}(\mathsf{J}^{\circ}_{[n]},
 \mathcal{F}^{\circ}\otimes_{\mathbb{C}} R^{n}\pi_{1,*}^{\circ}\mathbb{C})\,.
\end{equation}

We now obtain a cohomology class $\Phi_{n}^{\mathrm{GW}}$
as follows. 
\begin{dfn}\label{dfnrelativeGWcohomologyclass}
	The class $\Phi_{n}^{\mathrm{GW}}$ is 
	defined as
	\[
	\Phi_{n}^{\mathrm{GW}}= \mathfrak{r}\circ\mathfrak{s}
	\left(
	\mathbold{j}^{*}
	((\mathbold{\delta}\times \mathbold{id})^{-1})^{*}\varpi 
	\right)
	\in 
	 H^{0}\left(\mathsf{J}^{\circ}_{[n]},
	\left(\mathcal{P}^{\circ}_{[n]}(\mathsf{J}^{\circ}_{[n]}\times \Theta)|_{\mathsf{J}^{\circ}_{[n]}\times \Theta}\right)\otimes_{\mathbb{C}} R^{n}\pi^{\circ}_{1,*}\mathbb{C}\right)
	\,.
	\]
\end{dfn}

\subsection{Construction of $\gamma_{n}^{\mathrm{GW}}$  }
\label{secconstructionofhomology}

Let $A\in H_{1}(\Delta)$ be the $A$-cycle class on $\Delta$.
Let $A^{\mathrm{GW}}\in H^{0}(\mathsf{J}^{\circ}_{[n]}, (R^{n}\pi_{1,*}^{\circ}\mathbb{Z})^{*})$
be the relative homology class
given by the trivial product 
\[
\mathsf{J}^{\circ}_{[n]}\times  (A\otimes \cdots\otimes A)\subseteq \mathsf{J}^{\circ}_{[n]}\times H_{n}(\Delta_{[n]})\,,
\]
 where 
$ (A\otimes \cdots\otimes A)$ has been regarded as an element in $   H_{n}(\Delta_{[n]})$ under the K\"unneth decomposition.

 \begin{dfn}\label{dfnrelativeGWhomologyclass}
 Let as before $\mathbb{C}^{*}\rightarrow J\cong E$ be the covering map.
Similarly
 let $p:\widetilde{\mathsf{J}^{\circ}_{[n]}}\rightarrow \mathsf{J}^{\circ}_{[n]}$
 and $\widetilde{\mathsf{D}_{\infty}}\rightarrow {\mathsf{D}_{\infty}}$ be the corresponding (fiberwise) covering maps; and $\widetilde{\pi_{1}^{\circ}}:\widetilde{\mathsf{J}^{\circ}_{[n]}}\times  \mathsf{\Delta}_{[n]} -\widetilde{\mathsf{D}_{\infty}}\rightarrow \widetilde{\mathsf{J}^{\circ}_{[n]}}$ be the lift  $p^{*}\pi_{1}^{\circ}$ (see  
 Figure \ref{figure:figliftingtocovering}).
   	\begin{figure}[h]
 	\centering\[
 	 	\xymatrix{
 &&	\widetilde{\mathsf{J}^{\circ}_{[n]}}\times  \mathsf{\Delta}_{[n]} -\widetilde{\mathsf{D}_{\infty}}
 	\ar[dd]^{\widetilde{\pi_{1}^{\circ}}}
  && {\mathsf{J}^{\circ}_{[n]}}\times  \mathsf{\Delta}_{[n]} -{\mathsf{D}_{\infty}}\ar[dd]^{\pi_{1}^{\circ}} &&\\
 	 &&	&& &&\\
 	 &&	 \widetilde{\mathsf{J}_{[n]}^{\circ}}\ar[rr]^{p}	&&\mathsf{J}_{[n]}^{\circ}&&
 	}\]
 	\caption{Lifting to the covering of $\mathsf{J}_{[n]}^{\circ}$.}
 	\label{figure:figliftingtocovering}
 \end{figure} 
	For each $(u_1,\cdots, u_{n})\in \widetilde{\mathsf{J}^{\circ}_{[n]}}$, define $\gamma_{n}^{\mathrm{GW}}$ to be the cycle
	given by the image of a cycle of the form
			\[
			\Gamma_{n}^{\mathrm{GW}}:=\Big\{ v\in (\mathbb{C}^{*})^{n}~|~
		|v_{j}|=\mathrm{constant}~\text{for}~j\in [n]\,,
		~\text{subject to}~
				|{u_{i}v_{i}\over v_{j}}|,
		|{v_{i}\over u_{j}v_{j}}|>1~\text{for}~i<j\,\Big\}
		\]
		under the quotient map $(\mathbb{C}^{*})^{n}\rightarrow \mathsf{\Delta}_{[n]}$. See 
		Figure \ref{figure:constructionofgammaGW} below for an illustration.
	\end{dfn}
 
 \begin{figure}[h]
 	\centering
 		\centering
 	\includegraphics[scale=1.2]{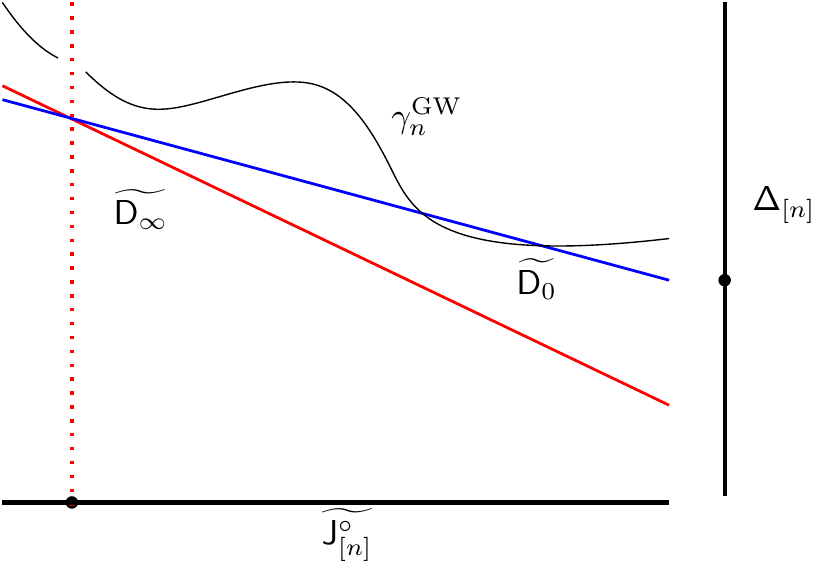}
 	\caption{Construction of the cycle $\gamma_{n}^{\mathrm{GW}}$ that avoids the divisor $\widetilde{\mathsf{D}_{\infty}}$ (shown in red). }
 	\label{figure:constructionofgammaGW}
  \end{figure}
 
 Since $\gamma_{n}^{\mathrm{GW}}$ is closed and avoids $\widetilde{\mathsf{D}_{\infty}}$, it gives a cycle class in 
 	$H^{0}(\widetilde{\mathsf{J}_{[n]}^{\circ}}, (R^{n}\widetilde{\pi_{1,*}^{\circ}}\mathbb{Z})^{*})$ which is still  denoted by $\gamma^{\mathrm{GW}}$ by abuse of notation.
By construction, it is a lift of 
	$A^{\mathrm{GW}}\in  H^{0}(\mathsf{J}^{\circ}_{[n]}, (R^{n}\pi_{1,*}^{\circ}\mathbb{Z})^{*})$ along the  map in homology induced by the covering map $p:\widetilde{\mathsf{J}^{\circ}_{[n]}}\rightarrow \mathsf{J}^{\circ}_{[n]}$.

%\begin{rem}
%The construction in Definition \eqref{dfnrelativeGWhomologyclass} also shows that  $(R^{n}\pi_{1,*}^{\mathrm{conf}}\mathbb{Z})^{*})$
%is not a local system $J_{[n]}$ but rather a constructible sheaf, if the $J_{[n]}^{\circ}\times \mathrm{Conf}_{[n]}\rightarrow J_{[n]}^{\circ}$ fails to be a locally trivial fibration.
%\end{rem}

\subsection{The pairing $
	\langle \gamma_{n}^{\mathrm{GW}}\,, ~ \varphi_{n}^{\mathrm{GW}} \rangle$}
\begin{dfn}\label{dfnrelativeGWcohomologyclasslifted}
	The class $\varphi_{n}^{\mathrm{GW}}$ is 
	defined as
	\[\varphi_{n}^{\mathrm{GW}}:=
	p^{*}\Phi_{n}^{\mathrm{GW}}=p^{*} \circ\mathfrak{r}\circ\mathfrak{s}
	\left(
	\mathbold{j}^{*}
	((\mathbold{\delta}\times \mathbold{id})^{-1})^{*}\varpi 
	\right)
	\in 
	 H^{0}\left(\mathsf{J}^{\circ}_{[n]},
	\left(\widetilde{\mathcal{P}^{\circ}_{[n]}}(\widetilde{\mathsf{J}^{\circ}_{[n]}}\times \Theta)|_{\widetilde{\mathsf{J}^{\circ}_{[n]}}\times \Theta}\right)\otimes_{\mathbb{C}} R^{n}\widetilde{\pi^{\circ}_{1,*}}\mathbb{C}\right) 
	\,.
	\]
\end{dfn}

The paring in singular homology and singular cohomology leads to a map
\begin{eqnarray}\label{eqnchomologicalpairing}
\Big\langle-,- \Big\rangle: \quad
H^{0}\left(\widetilde{\mathsf{J}^{\circ}_{[n]}}, (R^{n}\widetilde{\pi_{1,*}^{\circ}}\mathbb{Z})^{*}\right)
&\times&
H^{0}\left(\widetilde{\mathsf{J}^{\circ}_{[n]}},\left(\widetilde{\mathcal{P}^{\circ}_{[n]}}(\widetilde{\mathsf{J}^{\circ}_{[n]}}\times \Theta)|_{\widetilde{\mathsf{J}^{\circ}_{[n]}}\times \Theta}\right)\otimes_{\mathbb{C}} R^{n}\widetilde{\pi_{1,*}^{\circ}}\mathbb{C}\right)\nonumber\\
&\rightarrow&H^{0}\left(\widetilde{\mathsf{J}^{\circ}_{[n]}},\widetilde{\mathcal{P}^{\circ}_{[n]}}(\widetilde{\mathsf{J}^{\circ}_{[n]}}\times \Theta)|_{\widetilde{\mathsf{J}^{\circ}_{[n]}}\times \Theta}\right)\,.
\end{eqnarray}
This then gives the desired pairing  $
\langle \gamma_{n}^{\mathrm{GW}}\,, ~ \varphi_{n}^{\mathrm{GW}} \rangle
$ in \eqref{eqnparingofGWclasses}.
This pairing 
in fact descends to a smooth section of
	$\mathcal{P}^{\circ}_{[n]}(\mathsf{J}^{\circ}_{[n]}\times \Theta)|_{\mathsf{J}^{\circ}_{[n]}\times \Theta}$
	after a natural procedure called elliptic completion. 
	See Section \ref{secellipticcompletion} and in particular Remark 
\ref{remellipticompletion}
for details.\\

Representing
	the cohomology class $\varphi_{n}^{\mathrm{GW}}$ in Definition \ref{dfnrelativeGWcohomologyclasslifted}
	as a relative holomorphic $(n,0)$-form as in \eqref{eqnspectralsequence},
 the pairing \eqref{eqnchomologicalpairing} then
	becomes the fiberwise integration along $\widetilde{\pi_{1}^{\circ}}: \widetilde{\mathsf{J}_{[n]}}\times \mathsf{\Delta}_{[n]}-\widetilde{\mathsf{D}_{\infty}}
	\rightarrow \widetilde{\mathsf{J}_{[n]}^{\circ}}$. 	
			For practical purpose of evaluation, later on we shall replace 	$\pi_{1}^{\circ}$ by 
	\[
  \mathsf{J}_{[n]}^{\circ}\times \mathsf{Conf}_{[n]}(\Delta)-|\mathsf{D}|
	\rightarrow \mathsf{J}_{[n]}^{\circ}\,,
	\]
	and subsequently choose the lift $\gamma_{n}^{\mathrm{GW}}$ such that it avoids  the full locus $|\widetilde{\mathsf{D}}|=\widetilde{\mathsf{D}}_{\infty}+\widetilde{\mathsf{D}}_{0}$. 
We shall double check in Proposition
\ref{propGWpairing=bindpairing} that this does not change the pairing.
Therefore,
\eqref{eqnchomologicalpairing} correspondingly 
becomes a relative configuration space integral.

\section{Evaluation of ordered $A$-cycle integrals}
\label{secorderedAcycleintegralofC2n}

In this section, we shall
give concrete expressions for the constructions in Section \ref{secconstructionofGWclasses}
and explicitly evaluate
\eqref{eqnparingofGWclasses}, in terms of Jacobi theta functions.

\subsection{$2n$-point functions $C_{2n}$ for free fermions}

We first recall some standard constructions from
 \cite{Birkenhake:2013complex} 
 that are needed for the evaluation of the pairing \eqref{eqnparingofGWclasses}.

In term of the linear coordinate $z$, 
the line bundle $\mathcal{O}_{\mathrm{Jac}(E)}([\Theta])$
has a holomorphic section $\vartheta_{({1\over 2}, {1\over 2})}$ that vanishes at $\Theta$ given by 
\[
\vartheta_{({1\over 2}, {1\over 2} ) }(z,q):=\sum_{n\in \mathbb{Z}}(-1)^{n+{1\over 2}}q^{{1\over 2} (n+{1\over 2} )^2}e^{(n+{1\over 2})z}\,.
\]
It 
satisfies the following automorphy
\begin{equation}\label{eqnautomorphyfactoroftheta}
\vartheta_{({1\over 2}, {1\over 2})}(z+\pii)=\vartheta_{({1\over 2}, {1\over 2})}(z)\,,\quad
\vartheta_{({1\over 2}, {1\over 2})}(z+\pii \tau)=-e^{-\pi \mathbf{i} \tau}e^{-z} \vartheta_{({1\over 2}, {1\over 2})}(z)\,.
\end{equation}
By the Jacobi triple product formula one has
\begin{equation}\label{eqntheta}
\theta(z)= {\vartheta_{({1\over 2}, {1\over 2} ) }(z,q)\over \partial_{z}\vartheta_{({1\over 2}, {1\over 2} ) }(z,q)|_{z=0}}=
\pii{\vartheta_{({1\over 2}, {1\over 2} ) }(z,q)\over -2\pi (q^{1\over 24}(q)_{\infty})^3}
=(q)_{\infty}(e^{z\over 2}-e^{-{z\over 2}})\prod_{k\geq 1}(1-q^{k}e^z)(1-q^{k}e^{-z})\,.
\end{equation}
The prime form on the elliptic curve is the section of $ \mathcal{O}_{E\times E}([\Delta])$  given by 
\begin{equation}\label{eqndefofprimeform}
E(P,Q)={ \vartheta_{({1\over 2}, {1\over 2})} (P-Q)\over  \theta'(0)}
=\theta(P-Q)\,.
\end{equation}

Let
$\eta\in \mathsf{Pic}^{(0)}(E)$ and $c$ be the corresponding shift
$\eta\otimes \kappa^{-1}$ on $\mathrm{Jac}(E)$.
 Denote by $\theta_{c}: z\mapsto \theta(z+c)$
which is a holomorphic section of  $L_{c}=\mathcal{O}_{\mathrm{Jac}}([\Theta-c])$.
Then $\theta_{c}(0)=0$ if and only if $c=0$, namely $\eta=\kappa$.
The
Szeg\"o kernel is the section of $(\check{\eta}\boxtimes \eta)(\Delta)$ given by
\begin{equation}\label{eqndefofSzegokernel}
S_{c}(P,Q)={\theta_{c} (P-Q)\over \theta_{c}(0)}{1\over  E(P,Q)}\,,\quad
c\neq 0\,.
\end{equation}
This is the $2$-point function $\langle \psi(P)\psi^{*}(Q)\rangle$ for free  fermions \cite{Raina:1989} on the elliptic curve $E$,
see also Appendix \ref{appendixorderings}.
In general, the $2n$-point function $ \langle \psi(P_{1})\psi^{*}(Q_{1})\cdots \psi(P_{n})\psi^{*}(Q_{n}) \rangle$ in \eqref{eqnFnaszeromode} for free fermions with $n\geq 2$  is given by \eqref{eqnintrinsicsection}, with
\begin{eqnarray}\label{eqndefofC2n}
C_{2n}(P, Q)=
{\theta_{c} (\sum P_i-\sum Q_j)\over \theta_{c}(0)}
{\prod_{i<j} E(P_i, P_j) E(Q_i, Q_j)\over \prod_i E(P_i, Q_i)\cdot \prod_{i<j}
E(P_i, Q_j) E(Q_i,P_j)}\,,\quad c\neq 0\,.
\end{eqnarray}
The above quantity has a first  order pole at $c=0\in  \mathsf{Pic}^{(0)}(E)$.
The  section $\varpi$ in \eqref{eqnPoincareresidueofC2n} is  its residue at $c=0$, which concretely is
\begin{equation}\label{eqnomegaC2n}
\varpi=( \theta_{c} (0)\cdot C_{2n})|_{c=0}=
\theta (\sum P_i-\sum Q_j))\cdot
{\prod_{i<j} E(P_i, P_j) E(Q_i, Q_j)\over \prod_i E(P_i, Q_i)\cdot \prod_{i<j}
	E(P_i, Q_j) E(Q_i,P_j)}\,.
\end{equation}

Fay's multi-secant identity (a.k.a. boson-fermion correspondence, see e.g., \cite{Raina:1989, Kawamoto:1988}) gives
\begin{equation}\label{eqnFaymulti-secant}
C_{2n}(P, Q)=\det\, (S_{c}(P_i, Q_j))\,.
\end{equation}
Using the additive coordinates $z_i+w_i,w_j$ and multiplicative coordinates $u_i v_i, v_j$ for $P_i, Q_{j}$ interchangebly, then under the map $\mathbold{\delta}\times \mathbold{id}\times \mathrm{id}
: E_{[n]}\times E_{[\bar{n}]}\times \mathsf{Pic}^{(0)}(E)\rightarrow
\mathsf{J}_{[n]}\times \mathsf{ \Delta}_{[n]}\times \mathsf{Pic}^{(0)}(E)$
we have
\[
	\left(
(\mathbold{\delta}\times \mathbold{id}\times \mathrm{id})^{-1}\right)^{*} C_{2n}={\theta_{c} (\prod_{i} u_{i})\over \theta_{c} (0)}
{\prod_{i<j} \theta({u_i v_{i}\over u_j v_j}) \theta({v_i\over v_j})\over
 \prod_i \theta(u_i ) \cdot \prod_{i<j}
\theta({u_i v_{i}\over  v_j})\theta({ v_{i}\over u_j v_j})}\,.
\]
From
\eqref{eqnautomorphyfactoroftheta}, it is direct to check
that
 indeed it has trivial automorphy factors in $v_{i},i=1,2,\cdots,n$, agreeing with \eqref{eqndeltapullback}.\\

In what follows, 
we ignore the cubersome notation such as 	$\left(
(\mathbold{\delta}\times \mathbold{id}\times \mathrm{id})^{-1}\right)^{*}$ when making identifications.

\subsection{Evaluation of ordered $A$-cycle integrals}

For ease of notation, we introduce 
\begin{equation}\label{eqnTheta2n}
\varpi=
\theta (\sum P_i-\sum Q_j)
\cdot \Theta_{2n}(P,Q)\,,\quad 
\Theta_{2n}(P,Q):= {\prod_{i<j} \theta(P_i, P_j) \theta(Q_i, Q_j)\over \prod_i \theta(P_i, Q_i)\cdot \prod_{i<j}
\theta(P_i, Q_j) \theta(Q_i,P_j)}
\,.
\end{equation}

\begin{dfn}
Let $[n]=(1,2,\cdots, n), \overline{[n]}=(\bar{1},\bar{2},\cdots, \bar{n})$.
Consider the region
\begin{equation}\label{eqndomainoforderedAcycleintegral}
\Omega:=\Big\{(u,v)|  \bigcap_{\substack{(a,b)\in ([n]\cup\overline{ [n]})^2\\
a\neq b}} (-1)^{\chi_{ab}}(| x_{a}|-|x_{b}|)>0\Big\}\,,
\quad
x_{i}:=u_{i}v_{i}\,,~ x_{\bar{j}}:=v_{j}\,,
\end{equation}
where $\chi_{ab}$ is the value at $(a,b)$ of a function
\begin{equation}\label{eqndomainoforderedAcycleintegralintermsofchi}
\chi: ([n]\cup\overline{ [n]})^2\setminus \{(a,a)|a\in [n]\cup\overline{ [n]}\}\rightarrow \{0,1\}\,.
\end{equation}
	Let $\gamma_{n}^{\Omega}\subseteq \Omega$
be the cycle given by
	\begin{equation}\label{eqncycleavoidingpole}
	|v_{j}|=\mathrm{constant}~\text{for}~j\in [n]\,,\quad
	\text{subject to}\quad 
	|{x_{a}\over x_{b}}|^{(-1)^{\chi_{ab}}}>1~\text{for}~a,b\in [n]\cup [\bar{n}]\,.
	\end{equation}
Define
\begin{equation}\label{eqnTnOmega}
T_{n}^{\Omega}:=[v_1^{0} \cdots v_n^{0}] \left(\theta(\prod_{k} u_{k})\, \Theta_{2n}\right)
=({1\over \pii})^{n}\int_{\gamma_{n}^{\Omega}} {dv_{1}\over v_1}\cdots {dv_{n}\over v_n}\,
\theta(\prod_{k} u_{k}) \,\Theta_{2n}(P,Q)
\,.
\end{equation}
In what follows we shall denote $[v^{0}]:=[v_1^{0} \cdots v_n^{0}]$ for notational simplicity.
\end{dfn}

Note that extracting the $v^{0}$-term requires developing
the Fourier series and thus the choice of the domain  of convergence $\Omega$.
For example,
from  \eqref{eqnthetaFourier}  in Appendix \ref{appendixellipticfunctions} and \eqref{eqnautomorphyfactoroftheta} one has
\begin{equation} \label{eqnanalyticcontinutationofZ}
  {\theta'(v)\over \theta(v)}=
  \begin{cases}
%-{1\over 2} - g_{1,0}(v)=
-{1\over 2}-\sum_{k\neq 0}{v^{k}\over 1-q^{k}}\,,\quad |q|<|v|<1\,,\\
 % {1\over 2}+g_{1,1}(v)=
 {1\over 2}-\sum_{k\neq 0}{q^{k}v^{k}\over 1-q^{k}}\,,\quad |1|<|v|<|q|^{-1}\,.
  \end{cases}
\end{equation}
  In particular, the $v^{0}$-terms are given by $-{1\over 2},{1\over 2}$, respectively.
  This subtlety results in a technical difficulty in computing
  $[v^{0}] \,(\theta(\prod_{k} u_{k})\Theta_{2n})$ in the region
\eqref{eqndomainoforderedAcycleintegral} which we shall resolve in this section.\\

By definition, the domain $\Omega$ and the characteristic function $\chi$ determine each other. Later we shall often refer to them as orderings.
As explained in Section \ref{secconstructionofGWclasses}, one has
\begin{equation}\label{eqnpairingisintegration}
	\Big \langle ~\gamma_{n}^{\Omega}\,, ~ \varphi_{n}^{\mathrm{GW}} ~\Big\rangle=T_{n}^{\Omega}\,.
\end{equation}
Computationally, this is the fiberwise integration over  
a product of $A$-cycles on $E$ ordered according to \eqref{eqndomainoforderedAcycleintegral}.

	The lift $\gamma_{n}$ of the cycle class $A^{\mathrm{GW}}$ in Section \ref{secconstructionofhomology} is certainly not unique.
%\begin{ex}\label{exOmegawickOmegagw}
Of particular interest are those determined by the region\footnote{The ordering in \eqref{eqndomainoforderedAcycleintegralWick}
is the one in which one applies Wick's theorem to write
the $2n$-point function $C_{2n}$ in terms of the determinant of $2$-point functions as in \eqref{eqnFaymulti-secant}.}
\begin{equation}\label{eqndomainoforderedAcycleintegralWick}
\Omega_{\mathrm{Wick}}:=\Big\{(u,v)~|~
%|q^{-1}v_{1}|>
|u_{1}v_{1}|>|u_{2}v_{2}|>\cdots>|u_{n}v_{n}|>|v_{n}|>\cdots>|v_2|> |v_1|\Big\}\,.
\end{equation}
and the region in \eqref{eqnGWordering}
(see \cite[Equation 32, Equation 34]{Eskin:2006})
\begin{equation}\label{eqndomainoforderedAcycleintegralgw}
\Omega_{\mathrm{GW}}:=
\Big\{(u,v)~|~
%~|q^{-1}v_{n}|>
|u_{1}v_{1}|>|v_1|>|u_{2}v_{2}|>|v_2|>
\cdots>|u_n v_n|>|v_n|\Big\}\,.
\end{equation}
The corresponding $\chi$-functions satisfy
\begin{equation}\label{eqndomainoforderedAcycleWickGWintermsofchi}
\chi^{\mathrm{Wick}}_{i\bar{j}}= 0\,,\quad \chi^{\mathrm{GW}}_{i\bar{j}}=
\begin{cases}
0 \,,\quad i<j \\
1\,,\quad i>j
\end{cases}\,.
\end{equation}
We will also consider orderings $\Omega_{\mathrm{bound}}$ that satisfy
\begin{equation}\label{eqnboundordering}
 \chi_{i\bar{j}}+\chi_{j\bar{i}}=1\,,\quad i\neq j\,.
\end{equation}
This condition implies that
$|u_{i}v_{i}|>|v_{j}|$ if and only if $|v_{i}|>|u_{j}v_{j}|$
for any $i\neq j$, namely $i,\bar{i}$ are bound together.
Later (see Proposition\ref{propGWpairing=bindpairing}) we shall
show that  cycles determined by orderings satisfying \eqref{eqnboundordering} 
all produce the same result for the integral \eqref{eqnpairingisintegration} and concide
with $ \langle \gamma_{n}^{\mathrm{GW}}\,, ~ \varphi_{n}^{\mathrm{GW}}\rangle$.
% as it should be the case.

\subsubsection{Fay's multi-secant identity and Frobenius-Stickelberger formula}

We need some preliminary results before evaluating $T_{n}^{\Omega}$.

\begin{dfn}
For any $n\times n$ matrix $M_{n}$ with $n\geq 2$, let
$M_{n,\check{i}\check{j}}$ be the minor of the entry $(M_{n})_{ij}$, namely the submatrix formed by deleting the
$i$th row and $j$th column of $M_{n}$.
 Let
$
Z_{n}
$ be the $n\times n$ matrix whose $(i,j)$ entry is
$(Z_{n})_{ij}={\theta'\over \theta}({u_{i}v_{i}\over v_{j}})$.
\end{dfn}

 \begin{lem}[Frobenius-Stickelberger formula]
 \label{lemuseFayforAcycleintegral}
Let the notation be as above and $n\geq 2$.
Then in any region $\Omega$
one has
\begin{equation} \label{eqnuseFayforAcycleintegral}
\theta(\prod_{k} u_{k})
\cdot
\Theta_{2n}=\left(\sum_{j=1}^{n}\sum_{i=1}^{n}(-1)^{i+j}
\det Z_{n,\check{i}\check{j}}
\right)\,.
\end{equation}
 \end{lem}

 \begin{proof}

Fay's multi-secant identity \eqref{eqnFaymulti-secant}
gives
\begin{equation}\label{eqnFaymulti-secantsimplified}
\theta(c)^{n-1}\cdot
\theta (c+\sum_i (P_{i}-Q_i) )\cdot
\Theta_{2n}
=\det\left( ( {\theta(c+P_i-Q_j)\over   \theta(P_i-Q_j)} )_{1\leq i,j\leq n}\right)\,.
\end{equation}
 Comparing the degree $(n-1)$ Taylor coefficients in $c$ of the two sides of
 \eqref{eqnFaymulti-secantsimplified}, one is led to
\[
\theta(\prod_{k} u_{k})
\cdot
\Theta_{2n}
=[c^{n-1}]\det \left( (  \sum_{m\geq 0} {\theta^{(m)}(P_i-Q_j)\over m!\, \theta(P_i-Q_j)}c^m)_{1\leq i,j\leq n}\right)\,.
\]
 Subtracting the $j$th column by the $n$th column for $1\leq j\leq n-1$ and extract the
 $c^{n-1}$ coefficient of the right hand side, one obtains
 \[
\theta(\prod_{k} u_{k})
\cdot
\Theta_{2n}
=\det 
%\begin{pmatrix}
\begin{blockarray}{cccc}
		&  & j & n \\
	\begin{block}{c(ccc)}
	&\cdots&\cdots& 1\\
\vdots&&\vdots&\vdots\\
i&\cdots&{\theta'\over \theta}({u_{i}v_{i}\over v_{j}})-{\theta'\over \theta}({u_{i}v_{i}\over v_{n}})& 1\\
\vdots&&\vdots&\vdots\\
	&\cdots&\cdots& 1\\
	\end{block}
\end{blockarray}
%\end{pmatrix}
 \,.
\]
By the linearity of the determinant in the column vectors
 and its  properties under the swapping of columns,
 we see that
\begin{equation}\label{eqnFrobenius}
\theta(\prod_{k} u_{k})
\cdot
\Theta_{2n}
=\sum_{\ell=1}^{n}\,\quad
\det 
\begin{blockarray}{ccccc}
	&  & j & \ell & \\
	\begin{block}{c(cccc)}
		&\cdots&\cdots& 1&\cdots\\
		\vdots&&\vdots&\vdots&\\
		i&\cdots&{\theta'\over \theta}({u_{i}v_{i}\over v_{j}})&1 &\cdots\\
		\vdots&&\vdots&\vdots&\\
		&\cdots&\cdots&1&\cdots\\
	\end{block}
\end{blockarray}
 \,.
\end{equation}
Finally, applying Cramer's rule to the $j$th column to compute each summand above, we obtain \eqref{eqnuseFayforAcycleintegral}.
 \end{proof}

The above result \eqref{eqnFrobenius} is usually written in a more symmetric form
\begin{equation}\label{eqnFrobeniussymmetricform}
\theta(\prod_{k} u_{k})
\cdot
\Theta_{2n}
=\det\,\mathcal{Z}_{n}\,,\quad
\mathcal{Z}_{n}:=
\begin{pmatrix}
0 & 1 & \cdots & 1\\
-1 & {\theta'\over \theta}({u_{1}v_1\over v_1}) &  \cdots & {\theta'\over \theta}({u_{1}v_1\over v_n})\\
 &\vdots& \cdots& \vdots\\
-1 &  {\theta'\over \theta}({u_{n}v_n\over v_1}) &\cdots&  {\theta'\over \theta}({u_{n}v_n\over v_n})\\
\end{pmatrix}\,.
\end{equation}

\begin{dfn}\label{dfncJ}
Let the notation be as above.
For any nonempty subset $J\subseteq [n]$, consider the following matrix obtained from replacing the entries $(\mathcal{Z}_{n})_{ij},i,j\in J$ by their $v^{0}$-coefficients
\begin{equation}\label{eqnv0ZJ}
[v^0]\mathcal{Z}_{J}=
\begin{pmatrix}
0 & 1 & \cdots & 1\\
-1 & &  \cdots & \\
 &\vdots& [v^0] {\theta'\over \theta}({u_{i}v_{i}\over v_{j}})& \vdots\\
-1 & &\cdots& \\
\end{pmatrix}
=^{\eqref{eqnanalyticcontinutationofZ}}
\begin{pmatrix}
0 & 1 & \cdots & 1\\
-1 & &  \cdots & \\
 &\vdots&  {1\over 2}(-1)^{\chi_{i\bar{j}}}& \vdots\\
-1 & &\cdots& \\
\end{pmatrix}\,,\quad i,j\in J\,.
\end{equation}

\end{dfn}

\begin{dfn}\label{dfnchiK}
		Define
	\begin{equation}\label{eqndfncJ}
		c_{J}=\sum_{\sigma\in \mathfrak{S}^{\mathrm{cyl}}(J\cup \{0\})}(-1)^{\mathrm{sgn}(\sigma)}\prod_{a=0}^{|J|}([v^0]\mathcal{Z}_{J})_{a,\sigma(a)}\,,
	\end{equation}
	where $ \mathfrak{S}^{\mathrm{cyl}}(J\cup \{0\})$ stands for the cyclic subgroup of order $|J|+1$
	of the permutation group $\mathfrak{S}(J\cup \{0\})$.
	For any non-recurring sequence
$\mathbold{K}=(k_r, k_{r-1},\cdots ,k_1), 1\leq r\leq n$, define
\begin{equation}\label{eqndfnchiK}
\chi_{\mathbold{K}}=\sum_{a=1}^{r}\chi_{k_{a+1}\overline{k_{a}}}\,,
\end{equation}
where the cyclic ordering on sub-indices is used, for example $k_{r+1}:=k_{1}$.
\end{dfn}
The quantity $c_{J}$ in Definition \ref{dfnchiK} is the contribution to  $\det ([v^{0}]\mathcal{Z}_{J})$ in Definition \ref{dfncJ} from those permutations that are derrangments of order $|J|+1$.
It  can be made more concrete  as follows.
Let $\mathbold{J}=(j_r,\cdots, j_1)$ be a non-recurring sequence of length $|\mathbold{J}|=r, 1\leq r\leq n$, with the convention that
$j_{r+1}=j_1$.
Extending the definition of $\chi$ such that $\chi_{0\bar{j}}=0,\chi_{i\bar{0}}=1$, then by Cramer's rule one collects
the terms $([v^{0}]\mathcal{Z}_{J})_{j_2,\bar{0}}([v^{0}]\mathcal{Z}_{J})_{0,\bar{j_1}}\prod_{a=2}^{r} ([v^{0}]\mathcal{Z}_{J})_{j_{a+1},\bar{j_a}}$.
This leads to
\begin{equation}\label{eqncJintermsofsumoversequences}
c_{J}=\sum_{\mathbold{J}}(-1)^{|\mathbold{J}|-1}
\prod_{a=2}^{r}
{1\over 2}(-1)^{\chi_{j_{a+1}\overline{j_{a}}}}=(-{1\over 2})^{|J|-1}\sum_{\mathbold{J}} (-1)^{\chi_{\mathbold{J}}-\chi_{j_2 \overline{j_1}}}\,.
\end{equation}
where the summation is over all sequences $\mathbold{J}$
whose underlying set is $J$.
Here we have used the convention that
 $\prod_{a=2}^{r}
{1\over 2}(-1)^{\chi_{j_{a+1}\overline{j_{a}}}}=1$ if $r=1$.

 \begin{lem}
 \label{lemv0termofFay}
Let the notation be as above and $n\geq 2$.
Let
$Z_{[n]\setminus J}$ be the matrix obtained from $Z_{[n]}$ by deleting the
rows and columns from $J$.
Then in any region $\Omega$,
one has
\begin{equation} \label{eqnv0termofFay}
T_{n}^{\Omega}=
[v^0]\left(
\theta(\prod_{k} u_{k})
\cdot
\Theta_{2n}\right)
=\sum_{J:\, \emptyset\subsetneq J\subseteq [n]} c_{J}\cdot
 [v^0]( \det Z_{[n]\setminus J})
\,,
\end{equation}
where  the convention that  $\det Z_{\emptyset}=1$ is understood.
 \end{lem}

\begin{proof}
%%%%%
\iffalse
Using \eqref{eqnFrobenius} in Lemma  \ref{lemuseFayforAcycleintegral} and \eqref{eqnanalyticcontinutationofZ}, we have
\begin{eqnarray*}
[v^0]\left(
\theta(\prod_{k} u_{k})
\cdot
\Theta_{2n}\right)
&=&\sum_{j=1}^{n}\sum_{\mathbold{J}:\,j_1=j} (-1)^{|\mathbold{J}|-1}
\prod_{a=2}^{r}
{1\over 2}(-1)^{\chi_{j_{a+1}\overline{j_{a}}}}\cdot
 [v^0](\det Z_{[n]\setminus \mathbold{J}})\\
 &=&\sum_{\mathbold{J}:\, \emptyset\subsetneq J\subseteq [n]} (-1)^{|\mathbold{J}|-1}
\prod_{a=2}^{r}
{1\over 2}(-1)^{\chi_{j_{a+1}\overline{j_{a}}}}\cdot
[v^0]( \det Z_{[n]\setminus \mathbold{J}})
\,.
\end{eqnarray*}
Here the  sum  is over non-recurring sequences $\mathbold{J}$.
Denote the underlying set of $\mathbold{J}$ by $J$.
The above sum becomes the a double sum: the outer sum is over the choices of $J$ and the inner is over  all possible permutations of elements within the set $J$.
Using \eqref{eqncJintermsofsumoversequences}, the inner sum becomes $c_{J}$
and the
conclusion follows.
\fi
One computes the determinant in \eqref{eqnFrobeniussymmetricform}
by definition, in which one
 organizes the
summands according 
to the cycle that contains $0$ in the standard cycle decomposition for an element in $\mathfrak{S}([n]\cup \{0\})$.
Taking the $v^{0}$-term of each summand forces taking the  $v^{0}$ term
for each cycle part. 
The desired  claim then follows from
\eqref{eqndfncJ}.
\end{proof}

\begin{ex}\label{exTnn=12cases}
We now compute the  cases with $n=1,2$
using \eqref{eqnFrobenius} and \eqref{eqnv0termofFay}.
When $n=1$, we have
\[
T_{1}^{\Omega}= [v_{1}^0 c^{0}]  \theta_{c}(0) S_{c}(P_1-Q_1)
 =[v_{1}^0 c^{0}] {\theta_{c}({u_{1}})  \theta'(0)\over \theta(u_1)}=1\,.
 \]
When $n=2$, from \eqref{eqnFrobenius} we have
\[
\theta(\prod_{k} u_{k})
\cdot
\Theta_{2n}
=\det\,
\begin{pmatrix}
{\theta'\over \theta}({u_{1}v_{1}\over v_1}) & 1\\
{\theta'\over \theta}({u_{2}v_{2}\over v_1}) & 1
\end{pmatrix}
+
\det\,
\begin{pmatrix}
1 & {\theta'\over \theta}({u_{1}v_{1}\over v_2}) \\
1& {\theta'\over \theta}({u_{2}v_{2}\over v_2})
\end{pmatrix}\,.
\]
It follows that
\[
[v_{1}^0 v_{2}^0]\,
\left(
\theta(\prod_{k} u_{k})
\cdot\Theta_{2n}
\right)
={\theta'\over \theta}(u_2)+(-1)^{1+2}\cdot {1\over 2}(-1)^{\chi_{2\bar{1}}}
+{\theta'\over \theta}(u_1)+(-1)^{1+2}\cdot {1\over 2}(-1)^{\chi_{1\bar{2}}}\,.
\]
For the orderings $\Omega_{\mathrm{Wick}},\Omega_{\mathrm{GW}}$, one has
\[
T_{n}^{\Omega_{\mathrm{Wick}}}= {\theta'\over \theta}(u_1)-{1\over 2}+ {\theta'\over \theta}(u_2)-{1\over 2}\,,\quad
T_{n}^{\Omega_{\mathrm{GW}}}
= {\theta'\over \theta}(u_1)+{1\over 2}+ {\theta'\over \theta}(u_2)-{1\over 2}\,.
\]
\xxqed
\end{ex}

Note that by the parity of $\theta$ and \eqref{eqnTheta2n} we have
\[\left(
\theta(\prod_{k} u_{k})
\cdot
\Theta_{2n}\right)|_{\prod_{k}u_{k}=1}=0\,.
\]
However, it can be the case that
\[\left([v^{0}](\theta(\prod_{k} u_{k})
\cdot
\Theta_{2n}) \right)|_{\prod_{k}u_{k}=1}\neq
0\,,
\]
since the $[v^{0}]$ operation (which depends on the ordering) does not commute
with evaluation at the locus $\prod_{k}u_{k}=1$.
The $n=2$ case in Example \ref{exTnn=12cases} provides a simplest example of such phenomenon.
See Section \ref{secTnTgw} for further discussions.

\subsubsection{Ordered $A$-cycle integrals of $C_{2n}$ in terms of sums over partitions}

Both the coefficients $c_{J}$ and the determinants in \eqref{eqnv0termofFay} of
Lemma
 \ref{lemv0termofFay}
depend on the ordering in $\Omega$.
We now study properties of these quantities.
We start by recalling some standard notions.\\

 For each $\sigma\in \mathfrak{S}_{n}$, let its standard form of the cycle decomposition be
\begin{equation}\label{eqnstandardcycledecomposition}
 \sigma=(a_{1}\cdots a_{\lambda_1})(a_{\lambda_1+1}\cdots a_{\lambda_1+\lambda_{2}})
 \cdots
 (a_{\lambda_1+\cdots+\lambda_{\ell-1}+1}\cdots a_{\lambda_{1}+\cdots+\lambda_{\ell}}):=\sigma_1\cdots\sigma_{\ell}\,,
\end{equation}
 with the property that
 $a_{1},a_{\lambda_1+1},\cdots,a_{\lambda_1+\cdots+\lambda_{\ell-1}+1}$
 are the largest within the elements in the corresponding cycles and
\[
a_{1}<a_{\lambda_1+1}<\cdots <a_{\lambda_1+\cdots+\lambda_{\ell-1}+1}\,.
\]
%By abuse of notation,  we denote the underlying set of the cycle $\sigma_{k}$ by $\sigma_{k}$, whose cardinality is $|\sigma_{k}|=\lambda_{k}$.

\begin{dfn}\label{dfnstatisticsinpermutations}
	For a permutation $\sigma\in\mathfrak{S}_{n}$, 
let $\mathrm{des}(\sigma),\mathrm{cdes}(\sigma)$ be the cardinalities of the
descent set $\mathrm{Des}(\sigma)$ and cyclic descent set $\mathrm{cDes}(\sigma)$, respectively
\begin{eqnarray*}
\mathrm{Des}(\sigma)&:=&\{1\leq i\leq n-1|~\sigma(i)>\sigma(i+1)\}\,,\\
\mathrm{cDes}(\sigma)&:=&\{ 1 \leq i\leq n~|~\sigma(i)>\sigma([i+1]^{\mathrm{c}})\}\,.
%\mathrm{Exc}(\sigma)&=&\{i~|~i>\sigma(i)\}\,,
\end{eqnarray*}
Here $[i+1]^{\mathrm{c}}:=i+1$ if $i\leq n-1$ and $[n+1]^{\mathrm{c}}:=1$.

\end{dfn}

  Following the notation in \cite{Bloch:2000}, we denote the set of all partitions of $[n]$ by
$\Pi_{[n]}$. An element $\pi\in \Pi_{[n]}$ is an unordered collection of disjoint nonempty subsets of $[n]$ called parts
\begin{equation}\label{eqnpiexpression}
\pi=\{\pi_{1},\cdots, \pi_{\ell}\}\,.
\end{equation}
Similarly, we denote by $\Gamma_{[n]}$ the set of all compositions/lists of $[n]$, namely, ordered collections of subsets. Each composition is of the form
$\gamma=(\gamma_1,\cdots, \gamma_{\ell})$ whose length is $\ell(\gamma)=\ell$.
We denote
\[
(-1)^{\pi}=(-1)^{\gamma}:=(-1)^{n-\ell}\,.
\]
These are the same as
\[
(-1)^{\mathrm{sgn}(\sigma)}=(-1)^{\sum_{k}(\lambda_{k}-1)}=
(-1)^{n-\ell}\,.\]

 Let $\mathbf{S}(n,\ell)=\Pi_{[n],\ell}$ be the set of partitions of $[n]$ into $\ell$ disjoint nonempty subsets, and $\mathbf{s}(n,\ell)$
 the set of permutations in $\mathfrak{S}_{n}$ with $\ell$ cycles.
 Their cardinalities are $S(n,\ell), (-1)^{n-\ell}s(n,\ell)$, where $S(n,\ell),s(n,\ell)$
 are the 2nd and 1st Stirling numbers respectively.
It follows that
 \[
 |\mathbf{s}(n,\ell)|=\sum_{\pi\in \mathbf{S}(n,\ell)}\prod_{k=1}^{\ell} (-1)^{|\pi_{k}|-1}(|\pi_k|-1)!\,.
 \]
Furthermore,
 \[
\Pi_{n}=\bigsqcup_{\ell=1}^{n}\mathbf{S}(n,\ell)\,,\quad
\Gamma_{n}=\bigsqcup_{\ell=1}^{n}\mathbf{S}(n,\ell)\rtimes \mathfrak{S}(\ell)\,,\quad
\mathfrak{S}_{n}=\bigsqcup_{\ell=1}^{n}\mathbf{s}(n,\ell)\,.
 \]

We now define some functions and series that shall enter the evaluation 
of ordered $A$-cycle integrals of $\theta(\prod_{k} u_{k})
\Theta_{2n}$.

\begin{dfn}
	\label{dfnringofquasiellipticfunctions}
	Let $\mathrm{QE}=\mathbb{C}[\mathbb{G}_{2},\mathbb{G}_{4},\mathbb{G}_{6}][\{{\theta^{(k)}\over \theta}\}_{k\geq 1}]$ be the graded ring,
	with the grading of ${\theta^{(k)}\over \theta}$ being $k$ and that of  $\mathbb{G}_{a}$ being $a$.
	An element $f\in \mathrm{QE} $ of degree $m$ is called a quasi-elliptic function of mixed weight with leading weight $m$; $f$ is said to be of pure weight if it is homogeneous of degree $m$.	
\end{dfn}
\begin{rem}
The ring $\mathrm{QE}$ is related to the ring of 
quasi-Jacobi forms of index zero, see \cite{Eichler:1985, Libgober:2009} for details.
\end{rem}
\begin{dfn}\label{dfngbafunction}
For non-negative integers $a,b$ satisfying $0\leq a\leq b$, define the formal series
\begin{equation}\label{eqngbafunction}
  g_{b, a}(u):=\left((-1)^{a}
  ({1\over 2})^{b}+(-1)^{b}
  \sum_{k> 0} {q^{(b-a)k}u^{k}\over (1-q^{k})^{b}}+(-1)^{b}
  \sum_{k< 0} {q^{(b-a)k}u^{k}\over (1-q^{k})^{b}}\right)\,.
%  G_{b,a}(u)=(-1)^{a}({1\over 2})^{b}+g_{b,a}(u) \,,\quad
  % 1<|u|<|q|^{-1}  \,.
  \end{equation}
%  where the right hand side is understood as the analytic continuation in % the specified region of the underlying meromorphic function.
  We use the convention that $g_{0,0}=1$.
  \end{dfn}

\begin{thm}\label{thmquasiellipticitygba}
For each $0\leq a\leq b$, the following statements hold.
\begin{enumerate}[i).]
	\item The series $g_{b,a}(u)$ is locally absolutely convergent in the region $|q|^{a}<|u|<|q|^{a-b}$, and thus defines a holomorphic function there.
It analytically continues to a meromorphic function on $\mathbb{C}^{*}$ with only possible poles at $u= q^{m},m\in \mathbb{Z}$.
	\item 
	The analytic continuation satisfies 
		\[
	g_{b,a}
	\in \mathbb{C}\Big[\{{\theta^{(k)}\over \theta}\}_{0\leq k\leq b+1}\Big]\,.
	\]
	In particular, $g_{b,a}$ is a quasi-elliptic function 
	of mixed weight, with leading weight $b+1$.
	Furthermore, for fixed $b$, the leading part of $g_{b,a}$ is independent of $a$.
\end{enumerate}
\end{thm}
\begin{proof}
  See  Appendix \ref{appendixEulerianquasielliptic}.
\end{proof}
\begin{ex}\label{eqnexampleg10g11}
	The simplest examples are (see \eqref{eqnanalyticcontinutationofZ}) 
	\[
	g_{1,0}=g_{1,1}={\theta'\over \theta}\,.
	\]
\end{ex}

In what follows, when we refer to $g_{b,a}(u)$ we mean its analytic continuation in the region $1<|u|<|q|^{-1}$ unless stated otherwise.

 \begin{lem}
 \label{lemv0sumoverpartition}
 Let the notation be as above and $n\geq 1$.
Then one has
\begin{eqnarray*}
   [v^{0}] \,\det \, Z_{n}
  % =\sum_{\sigma \in \mathfrak{S}_{n}}\omega(\sigma)
  &=&
   \sum_{\sigma \in \mathfrak{S}_{n}}(-1)^{\mathrm{sgn}(\sigma)}
\prod_{k}
g_{|\sigma_{k}|,\, \chi_{\sigma_{k}}}(\prod_{i\in \sigma_{k}} u_{i})\,\\
 &=& \sum_{\pi \in \Pi_{[n]}}(-1)^{\pi}
\prod_{k} {1\over |\pi_{k}|}
\sum_{\tau_{k}\in \mathfrak{S}(\pi_{k})}
g_{|\pi_{k}|,\, \chi_{\tau_{k}}}(\prod_{i\in \pi_{k}} u_{i})\,,
\end{eqnarray*}
   where $ \mathfrak{S}(\pi_{k})$ is the permutation group of the set $\pi_{k}$.
In particular,
one has
\begin{eqnarray*}
   [v^{0}] \,\det \, Z_{n}  &=&
   \sum_{\pi \in \Pi_{[n]}}(-1)^{\pi}
\prod_{k} {1\over |\pi_{k}|}
\sum_{\tau_{k}\in \mathfrak{S}(\pi_{k})}
g_{|\pi_{k}|,\, 0}(\prod_{i\in \pi_{k}} u_{i})\,\quad \mathrm{in}\quad \Omega_{\mathrm{Wick}}\,,\\
 %  [v_1^{0} \cdots v_n^{0}] \,\det \, D_{n} 
  &=&
   \sum_{\pi \in \Pi_{[n]}}(-1)^{\pi}
\prod_{k} {1\over |\pi_{k}|}
\sum_{\tau_{k}\in \mathfrak{S}(\pi_{k})}
g_{|\pi_{k}|,\, \mathrm{cdes}(\tau_k)}(\prod_{i\in \pi_{k}} u_{i})\,\quad \mathrm{in}\quad \Omega_{\mathrm{GW}}\,.
 \end{eqnarray*}

 \end{lem}

\begin{proof}
By definition, one has
\begin{eqnarray*}
 \det \, Z_{n}
 &=&\sum_{\sigma \in \mathfrak{S}_{n}}(-1)^{\mathrm{sgn}(\sigma)}  (Z_{n})_{1,\sigma(1)}\cdots (Z_{n})_{n,\sigma(n)}\,.
  \end{eqnarray*}
 We then expand ${\theta'\over \theta} ({u_{i}v_{i}\over v_{j}})$ in the region $\Omega$.
 Using \eqref{eqnanalyticcontinutationofZ},
it follows from straightforward computations that
\[
   [ v^{0}] \,\det \, Z_{n}
  % =\sum_{\sigma \in \mathfrak{S}_{n}}\omega(\sigma)
  =
   \sum_{\sigma \in \mathfrak{S}_{n}}(-1)^{\mathrm{sgn}(\sigma)}
\prod_{k}
g_{|\sigma_{k}|,\, \chi_{\sigma_{k}}}(\prod_{i\in \sigma_{k}} u_{i})\,.
\]
For the partitions $\pi=\{\pi_1,\cdots,\pi_{\ell}\}$ in $ \Pi_{n}$,
by considering the cycles obtained by permuting elements within the blocks $\pi_{k},k=1,2,\cdots,\ell$ one recovers the set $\mathfrak{S}_{n}$:
for the
 block $\pi_{k}=\{b_1,\cdots, b_{m}\}$, a permutation $\tau_k\in\mathfrak{S}(\pi_{k})$
 gives the $k$th cycle
$\sigma_{k}=(\tau_k(b_1)\cdots \tau_k(b_{m}))$.
Since the elements $b_1,\cdots ,b_{m}$ are distinct, one has $
\chi_{\sigma_{k}}=\chi_{\tau_{k}}
$.
The contribution corresponding to the block $\pi_{k}=\{b_1,\cdots, b_{m}\}$ is then
\[
{1\over |\pi_{k}|}\sum_{\tau_{k}\in \mathfrak{S}(\pi_{k})} g_{|\pi_{k}|,\, \chi_{\tau_k}}(\prod_{i\in \pi_{k}} u_{i})\,.
\]
The factor ${1/ |\pi_{k}|}$ arises from the fact that different permutations $\tau\in\mathfrak{S}(\pi_{k})$ that are related by cyclic ones give rise to the same cycle
$(\tau(b_1)\cdots \tau(b_{m}))$.

For the region $\Omega_{\mathrm{GW}}$,  $\chi^{\mathrm{GW}}_{i\bar{j}}=1,i\neq j$ if and only if $i>j$,
thus $\chi_{\tau_{k}}=\mathrm{cdes}(\tau_{k})$.

For the region $\Omega_{\mathrm{Wick}}$, $\chi^{\mathrm{Wick}}_{i\bar{j}}=0$ for any $i\neq j$,
thus $\chi_{\tau_{k}}=0$. The result then follows.
  \end{proof}

\begin{prop}\label{propGWpairing=bindpairing}
Let the notation be as above. 
\begin{enumerate}[i).]
	\item $T_{n}^{\Omega}$ depends only on the set of values $\chi_{i\bar{j}}$ for $i\neq j$.
	\item 
	For the cycle $\gamma_{n}^{\mathrm{GW}}$ in Definition \ref{dfnrelativeGWhomologyclass},
	the corresponding pairing $ \langle\gamma_{n}^{\mathrm{GW}}\,, ~ \varphi_{n}^{\mathrm{GW}} \rangle$ coincides with (the analytic continuation of)
	$ \langle \gamma_{n}^{\Omega_{\mathrm{GW}}}\,, ~ \varphi_{n}^{\mathrm{GW}} \rangle$.
	\item 
One has
 \begin{equation}\label{eqnbound=GW}
 \Big \langle ~\gamma_{n}^{\Omega_{\mathrm{bound}}}\,, ~ \varphi_{n}^{\mathrm{GW}} ~\Big\rangle=
\Big \langle ~\gamma_{n}^{\Omega_{\mathrm{GW}}}\,, ~ \varphi_{n}^{\mathrm{GW}} ~\Big\rangle\,.
 \end{equation}

\end{enumerate}
\end{prop}
\begin{proof}
	\begin{enumerate}[i).]
		\item
		By \eqref{eqncJintermsofsumoversequences}, $c_{J}$ in \eqref{eqnv0termofFay} of Lemma \ref{lemv0termofFay} only depends on the  set of values $\chi_{i\bar{j}}$ for $i\neq j$.
This is also true
for the quantities  $\chi_{\tau_k}$ and thus
 $g_{|\pi_{k}|,\, \chi_{\tau_k}}$ in Lemma \ref{lemv0sumoverpartition}
 when $|\pi_{k}|>1$.
 In the case $|\pi_{k}|=1$, say
 $\pi_{k}=\{i\}$, one has
 $g_{|\pi_{k}|,\, \chi_{\tau_k}}={\theta'\over \theta}(u_i)$.
This is also independent of $\chi_{i\bar{i}}$ in the sense of analytic continuation.
Combining these, one proves (i).

\item 
The cycle $\gamma_{n}^{\mathrm{GW}}$ satisfies $\chi_{i\bar{j}}=1,i\neq j$ if and only if $i>j$.
The disired claim then follows from  (i).

\item

The quantity $
\theta(\prod_{k} u_{k})
\cdot
\Theta_{2n}$ is invariant under the  following action induced by transposition in the symmetry group $\mathfrak{S}_{n}$
\begin{equation}\label{eqnactiononboundvariables}
(u_i v_i, v_i, u_j v_j, v_j)\mapsto 
(u_j v_j, v_j, u_i v_i, v_i)\,,
\end{equation}
and thus by the full $\mathfrak{S}_{n}$ action on the bound pairs $(i,\bar{i})$.
Up to such a permutation, one can always assume that the characteristic function $\chi$ for 
$\Omega_{\mathrm{bound}}$, which satisfies  \eqref{eqnboundordering},
satisfy the condition that  $\chi_{i\bar{j}}=1,i\neq j$ if and only if $i>j$.
The equality \eqref{eqnbound=GW} then follows from (i).
\end{enumerate}
\end{proof}

\begin{dfn}\label{dfnGpi}
Consider the extended set $\widetilde{[n]}:=[n]\cup\{0\}$ and the extended
characteristic function of $\chi_{ab}$ satisfying
\[
{\chi}_{0\bar{i}}=0\,,\quad {\chi}_{i\bar{0}}=1\,,\quad \forall\, i\in [n]\,,
\]
with the extra variables $u_0 v_0, v_0$ included.
For $\pi_{k}\subseteq [n]$, we define
\[
G_{\pi_k}=
\sum_{\tau_{k}\in \mathfrak{S}(\pi_{k})}g_{|\pi_{k}|,\, \chi_{\tau_{k}}}(\prod_{i\in \pi_{k}} u_{i})\,.
\]
For any partition $\pi=\{\pi_{1},\cdots, \pi_{\ell}\}$, denote
\[
{1\over |\pi|}G_{\pi}=\prod_{k}{1\over |\pi_{k}|}G_{\pi_{k}}\,.
\]
\end{dfn}
From  \eqref{eqnanalyticcontinutationofZ}, Definition \ref{dfngbafunction} and Example \ref{exTnn=12cases}, we have
\begin{equation}\label{eqnG1}
G_{\{1\}}^{\mathrm{Wick}}=G_{\{1\}}^{\mathrm{GW}}=g_{1,0}={\theta'\over \theta}\,,\quad
T_{1}^{\Omega}=1\,.
\end{equation}

\begin{conv}\label{conventionG0T0}
We use the following convention
\[
(-1)^{\pi_k}
{1\over |\pi_k|}
G_{\pi_k}=1\,,~
T_{\pi_k}^{\Omega}=0\,
~ \text{for}~ \pi_k=\emptyset\subseteq [n]\,,
\quad
(-1)^{S}
{1\over |S|}
G_{S}=0\,~
\mathrm{for}~ S=\{0\}\,.
\]
\end{conv}

\begin{thm}\label{thmsumoverpartitions}
Let the notation be as above and $n\geq 1$.
Then under Convention \ref{conventionG0T0}, one has
\begin{equation} \label{eqnsumoverpartitions}
T_{n}^{\Omega}=
[v^0]\left(
\theta(\prod_{k} u_{k})
\cdot
\Theta_{2n}\right)=2^{2}\cdot [u_0^0]
\left(
\sum_{{\pi}\in \Pi^{*}_{\widetilde{[n]}}}{ (-1)^{{\pi}} \over |{\pi}|} {G}_{{\pi}}\right)
\,,
\end{equation}
where $\Pi^{*}_{\widetilde{[n]}}$ represents
the set of partitions of $\widetilde{[n]}$ in which the singlet $\{0\}$
is not a block.
In particular, $T_{n}^{\Omega}$
is a quasi-elliptic function of mixed weight with leading weight $n-1$,
whose leading part is independent of the ordering.

\end{thm}

\begin{proof}
For $n\geq 2$,  combining
Lemma \ref{lemv0termofFay} and Lemma \ref{lemv0sumoverpartition}, we have
\begin{equation}\label{eqnmixedweightstructure}
[v^0]\left(
\theta(\prod_{k} u_{k})
\cdot
\Theta_{2n}\right)=\sum_{J:\,\emptyset \subsetneq J\subseteq [n]}
c_{J}
\sum_{\pi\in \Pi_{[n]\setminus J}}{ (-1)^{\pi} \over |\pi|} G_{\pi}
\,.
\end{equation}
This is also true for the $n=1$ case according to the definition of $c_{J}$ in Definition \ref{dfncJ} and Convention \ref{conventionG0T0}.
Observe that in the evaluation of $c_{J}$, the diagonal entries in
$[v^{0}]\mathcal{Z}_{J}$ are irrelevant.
Repeating the reasoning in the proof of  Lemma \ref{lemv0sumoverpartition}, we have
\begin{equation}\label{eqncJformula}
c_{J}=2^2 \cdot [u_0^0] \left({(-1)^{J}\over |J|+1} G_{J\cup\{0\}}\right)\,,\quad J\neq \emptyset\,.
\end{equation}
Combining these, we obtain  \eqref{eqnsumoverpartitions}.
The claim on quasi-ellipticity follows form Theorem \ref{thmquasiellipticitygba}.

\end{proof}

We can and shall replace the range $\Pi^{*}_{\widetilde{[n]}}$  by $\Pi_{\widetilde{[n]}}$  in Theorem \ref{thmsumoverpartitions} using Convention \ref{conventionG0T0}.
Note that within the blocks of the partition $\pi$ in \eqref{eqnsumoverpartitions}, the operation $[u_0^0]$ has effect  only
on the block that contains $0$.

\begin{rem}\label{remsumoverpartitionsassutraction}
The reasoning in deriving \eqref{eqncJformula} shows that if one
applies Cramer's rule to compute $\det Z_{[n]\cup \{0\}}$, then as in
 Lemma \ref{lemv0sumoverpartition} one has
\[
T_{n}^{\Omega}=2^2 \cdot [u_0^{0}][v^0] \left( \det Z_{[n]\cup \{0\}}-{\theta'\over \theta}(u_0) \det Z_{[n]} \right)\,.
\]
% is this analogous to the tadpole cancellation or dilaton relation in Virasoro constraints?
This explains the combinatorial structure for $\Pi^{*}_{\widetilde{[n]}}$ in \eqref{eqnsumoverpartitions}.
\end{rem}

Theorem \ref{thmsumoverpartitions}
exhibits the mixed weight structure of $T_{n}^{\Omega}$: each summand in
\eqref{eqnmixedweightstructure} has leading weight $[n]-|J|$.
We demonstrate this by considering the following example.

  \begin{ex}\label{exTn3case}
 Consider the $n=3$ case similar to Example \ref{exTnn=12cases}.
 One has from Lemma  \ref{lemuseFayforAcycleintegral} that
\begin{eqnarray*}
&&\theta(\prod_{k} u_{k})
\cdot
\Theta_{2n}\\
&=&\det\,
\begin{pmatrix}
{\theta'\over \theta}({u_{1}v_{1}\over v_1}) & {\theta'\over \theta}({u_{1}v_{1}\over v_2}) & 1\\
{\theta'\over \theta}({u_{2}v_{2}\over v_1}) &  {\theta'\over \theta}({u_{2}v_{2}\over v_2}) & 1\\
{\theta'\over \theta}({u_{3}v_{3}\over v_1}) &  {\theta'\over \theta}({u_{3}v_{3}\over v_2}) & 1
\end{pmatrix}
+
\det\,
\begin{pmatrix}
{\theta'\over \theta}({u_{1}v_{1}\over v_1})  & 1& {\theta'\over \theta}({u_{1}v_{1}\over v_3})\\
{\theta'\over \theta}({u_{2}v_{2}\over v_1}) & 1 &  {\theta'\over \theta}({u_{2}v_{2}\over v_3}) \\
{\theta'\over \theta}({u_{3}v_{3}\over v_1})  & 1&  {\theta'\over \theta}({u_{3}v_{3}\over v_3})
\end{pmatrix}
+\det\,
\begin{pmatrix}
1& {\theta'\over \theta}({u_{1}v_{1}\over v_2}) & {\theta'\over \theta}({u_{1}v_{1}\over v_3}) \\
1& {\theta'\over \theta}({u_{2}v_{2}\over v_2}) &  {\theta'\over \theta}({u_{2}v_{2}\over v_3}) \\
1& {\theta'\over \theta}({u_{3}v_{3}\over v_2}) &  {\theta'\over \theta}({u_{3}v_{3}\over v_3})
\end{pmatrix}\,.
\end{eqnarray*}
Direct computations give the result in  Lemma \ref{lemv0termofFay}
\begin{eqnarray*}
[v^0]\begin{pmatrix}
{\theta'\over \theta}({u_{1}v_{1}\over v_1}) & {\theta'\over \theta}({u_{1}v_{1}\over v_2}) & 1\\
{\theta'\over \theta}({u_{2}v_{2}\over v_1}) &  {\theta'\over \theta}({u_{2}v_{2}\over v_2}) & 1\\
{\theta'\over \theta}({u_{3}v_{3}\over v_1}) &  {\theta'\over \theta}({u_{3}v_{3}\over v_2}) & 1
\end{pmatrix}
&=&{\theta'\over \theta}(u_1)  {\theta'\over \theta}(u_2)-g_{2,\chi_{1\bar{2}}+\chi_{2\bar{1}}}(u_1 u_2)\\
&-&{1\over 2}(-1)^{\chi_{3\bar{2}}}{\theta'\over \theta}(u_1)-{1\over 2}(-1)^{\chi_{3\bar{1}}}{\theta'\over \theta}(u_2)\\
&
+&({1\over 2})^{2}(-1)^{\chi_{1\bar{2}}+\chi_{3\bar{1}}}+({1\over 2})^{2}(-1)^{\chi_{2\bar{1}}+\chi_{3\bar{2}}}\,.
\end{eqnarray*}
This leads to  (using cyclic labeling on the sub-indices)
\begin{eqnarray*}
T_{n}^{\Omega}=
[v^0]
\left(\theta(\prod_{k} u_{k})
\cdot
\Theta_{2n}\right)&=&
\sum_{i<j}({\theta'\over \theta}(u_i)  {\theta'\over \theta}(u_j)-g_{2,\chi_{i\bar{j}}+\chi_{j\bar{i}}}(u_i u_j))\\
&-&{1\over 2}\sum_{i}\sum_{\{j,k\}} \delta_{ijk}
((-1)^{\chi_{j\overline{k}}}+(-1)^{\chi_{k\overline{j}}})
{\theta'\over \theta}(u_i) \\
&+&({1\over 2})^2\sum_{\{i,j,k\}}  \delta_{ijk}(-1)^{\chi_{i\bar{j}}+\chi_{k\bar{i}}}\,,
\end{eqnarray*}
 where $\delta_{ijk}=1$ if and only if $i,j,k$ are distinct.

Using \eqref{eqnFourierofZ} and 
the results in Appendix \ref{appendixEulerianquasielliptic}, 
we have
\begin{eqnarray*}
 g_{2,1}(u)&=&-({1\over 2})^2+{B_{2}\over 2}-({\theta''\over 2\theta}(u)+\mathbb{G}_2)\,,\\
  g_{2,0}(u)&=&{1\over 2}+g_{2,1}(u)+{\theta'\over \theta}(u)-{1\over 2}\,,\\
 g_{2,2}(u)&=&{1\over 2}+g_{2,1}(u)-{\theta'\over \theta}(u)-{1\over 2}\,,
\end{eqnarray*}
 where $B_{2}=1/6$ is the 2nd Bernoulli number.
This shows that $T_{n}^{\Omega}$ is a quasi-elliptic function whose leading part is independent of the choice of ordering and is given by
\begin{equation}\label{eqnT3explicitcomputation}
T_{3}^{\mathrm{lead}}=
\sum_{i<j}({\theta'\over \theta}(u_i)  {\theta'\over \theta}(u_j)+{\theta''\over 2\theta}(u_i u_j)+\mathbb{G}_2)\,.
\end{equation}
For the ordering $\Omega_{\mathrm{GW}}$, one has
\[
\chi_{j\bar{k}}+\chi_{k\bar{j}}=1\,,\quad
(({1\over 2})^2-{B_{2}\over 2})\cdot 3
+({1\over 2})^2\sum_{\{i,j,k\}}  \delta_{ijk}(-1)^{\chi_{i\bar{j}}+\chi_{k\bar{i}}}=0\,,
\]
and thus
\[
T_{3}^{\mathrm{lead}}=T_{3}^{\Omega_{\mathrm{GW}}}\,.
\]
On the other hand, from \eqref{eqngeneratingseriesformulaintro}, we see that the $3$-point function  $T_{3}$ is given
\begin{equation}\label{eqnT3BO}
T_3=
\sum_{1\leq i\neq j\leq 3}{\theta'\over \theta}(u_i) {\theta'\over \theta}(u_i u_j)
-\sum_{1\leq i\leq 3}{\theta''\over \theta}(u_i) +\theta'''(1)\,.
\end{equation}

Observe that $\theta'/\theta$ is connected to the Weierstrass $\zeta$-function via
\[
{\theta'\over \theta}(z)=\zeta({z})-2\mathbb{G}_2 z\,,\quad 
{\theta''\over \theta}=({\theta'\over \theta})'+({\theta'\over \theta})^2\,.
\]
The following addition formulas for Weierstrass $\zeta$- and $\wp$-functions are useful
\[
{\theta'\over \theta}(z+w)={\theta'\over \theta}(z)+{\theta'\over \theta}(w)
+{1\over 2}{\wp'(z)-\wp'(w)\over \wp(z)-\wp(w)}\,,
\quad
\wp(z+w)=-\wp(z)-\wp(w)+
{1\over 4}({\wp'(z)-\wp'(w)\over \wp(z)-\wp(w)})^2\,.
\]
They lead to
\begin{equation}\label{eqnhomogeneousadditionformula}
(\zeta(z_1)+\zeta(z_2)+\zeta(z_3))^{2}=\wp(z_1)+\wp(z_2)+\wp(z_3)\,,\quad z_1+z_2+z_3=0\,.
\end{equation}
The addition formulas also give
\begin{eqnarray*}
{\theta''\over \theta}(u_i u_j)&=&-{\theta''\over \theta}(u_i)-{\theta''\over \theta}(u_j)-6\mathbb{G}_2
+({\theta'\over \theta}(u_i))^2+({\theta'\over \theta}(u_j))^2\\
&&+({\theta'\over \theta}(u_i))^2+({\theta'\over \theta}(u_j))^2
+2 {\theta'\over \theta}(u_i) {\theta'\over \theta}(u_j)
+( {\theta'\over \theta}(u_i)+ {\theta'\over \theta}(u_j)){\wp'(u_i)-\wp'(u_j)\over \wp(u_i)-\wp(u_j)}\,,\\
{\theta'\over \theta}(u_i) {\theta'\over \theta}(u_i u_j)
&=&({\theta'\over \theta}(u_i))^2+{\theta'\over \theta}(u_i) {\theta'\over \theta}(u_j)+{\theta'\over \theta}(u_i)\cdot {1\over 2}{\wp'(u_i)-\wp'(u_j)\over \wp(u_i)-\wp(u_j)}\,.
\end{eqnarray*}
It follows that
\begin{eqnarray*}
T_{3}^{\mathrm{lead}}&=& -\sum_{1\leq i\leq 3}{\theta''\over \theta}(u_i) +2\sum_{i}({\theta'\over \theta}(u_i))^2
+2\sum_{i<j}{\theta'\over \theta}(u_i)){\theta'\over \theta}(u_j)\\
&&+{1\over 2}\sum_{i<j} ({\theta'\over \theta}(u_i)+ {\theta'\over \theta}(u_j)){\wp'(u_i)-\wp'(u_j)\over \wp(u_i)-\wp(u_j)}-6\mathbb{G}_2\,,\\
T_{3}
&=&
2\sum_{i} ({\theta'\over \theta}(u_i))^2+
2\sum_{i< j}{\theta'\over \theta}(u_i){\theta'\over \theta}(u_j)
+{1\over 2}\sum_{i<j} ({\theta'\over \theta}(u_i)+ {\theta'\over \theta}(u_j)){\wp'(u_i)-\wp'(u_j)\over \wp(u_i)-\wp(u_j)}\\
&&-\sum_{1\leq i\leq 3}{\theta''\over \theta}(u_i) +
\theta'''(1)\,.
\end{eqnarray*}
From \eqref{eqnthetalogexpansion} one sees that $\theta'''(1)=-6\mathbb{G}_{2}$, hence we obtain
\[
T_{3}^{\mathrm{lead}}=T_{3}\,.
\]

Furthermore, from \eqref{eqnhomogeneousadditionformula}
one can check that $T_{3}|_{u_1 u_2 u_3=1}=0$ as follows.
Using parity we see that on this divisor
\[
T_{3}|_{u_1 u_2 u_3=1}=-\sum_{i\neq j}{ \theta'\over \theta}(u_i ){ \theta'\over \theta}(u_j )-\sum_{i}{ \theta''\over \theta}(u_i )+\theta'''(1)\,.
\]
By the addition formula \eqref{eqnhomogeneousadditionformula} above, this gets simplified into
\begin{eqnarray*}
T_{3}|_{u_1 u_2 u_3=1}&=&\sum_{i}({ \theta'\over \theta}(u_i ))^2-\sum_{i}\wp(u_i)-\sum_{i}{ \theta''\over \theta}(u_i )+\theta'''(1)\\
&=&\sum_{i}({ \theta'\over \theta}(u_i ))^2-\sum_{i}(- ({ \theta'\over \theta}(u_i ))'-2\mathbb{G}_{2})-\sum_{i}{ \theta''\over \theta}(u_i )+\theta'''(1)\\
&=&\theta'''(1)+6\mathbb{G}_{2}=0\,.
\end{eqnarray*}
\xxqed
  \end{ex}

 \begin{rem}
 Comparing to the ``bosonic'' computations \cite{Boehm:2014tropical, Goujard:2016counting, Li:2020regularized} of $v^{0}$-terms of differential polynomials of
 \[
 \wp(z)+2\mathbb{G}_2=\sum_{k\neq 0} {ku^{k}\over 1-q^{k}}\,,
 \]
 the ``fermionic'' evaluation of $v^{0}$-terms here
 is much simpler.
 The reason is that
  the degree in ${\theta'\over \theta}({u_{i}v_{i}\over v_{j}})$
 is at most one, and there is no other $k$-factors in the numerator of the summand of $\theta'/\theta$ given in
 \eqref{eqnanalyticcontinutationofZ}, unlike the $k$ in the numerator $k u^{k}$ above.
 \end{rem}

\subsection{Recursive structure }
\label{secrecursivestructure}

Based on \eqref{eqngeneratingseriesformulaintro},
further formulas and recursive structures for $F_{n}$ have been found and combinational and modular properties (such as holomorphic anomaly equations) are studied, see e.g., \cite{Pixton:2008,
 Zagier:2016partitions}.
For instance, one can introduce a new variable $z_0$  that eventually is set to zero.
By Cramer's rule 
the determinant formula $\det M_{n}$ in \eqref{eqngeneratingseriesformulaintro}
is equivalent  to the recursion \cite {Zagier:2016partitions}
\[
\sum_{i=0}^{n}{(-1)^{n-i}\over (n-i)!}{\theta^{(n-i)}\over \theta}(z_0+z_1+\cdots+z_i) \det M_{i}(z_0, z_1,\cdots, z_{i-1})=0\,,\quad n\geq 1\,,\quad
\det M_0()=1\,.
\]
After symmetrization, this gives 
\begin{equation}\label{eqnFnrecursion}
\sum_{I\subseteq [n]}{(-1)^{n-|I|}}\theta^{(n-|I|)}\left(\sum_{i\in I}z_i\right) \cdot F_{|I|}(\{z_{i}\}_{i\in I})=0\,,\quad n\geq 1\,,
\quad
F_{0}()=1\,.
\end{equation}
The recursion \eqref{eqnFnrecursion} 
translates directly to
\begin{equation}\label{eqnTnrecursion}
\sum_{I\subseteq [n]}{(-1)^{n-|I|}}{\theta^{(n-|I|)}\over \theta}\left(\sum_{i\in I}z_i\right) \cdot T_{|I|}(\{z_{i}\}_{i\in I})=0\,,\quad n\geq 1\,,
\quad
{T_{0}()\over \theta()}:=1\,,\theta^{(n)}():=\theta^{(n)}(0)\,.
\end{equation}
The latter is in turn equivalent to the following combinatorial expression for $T_{n}$
\cite[Equation 7.7]{Bloch:2000}
\begin{equation}\label{eqnTnintermsofsumovercompositions}
T_{n}=\sum_{\gamma\in \Gamma_{[n]}}(-1)^{\gamma}\theta^{(|\gamma_{1}|)}(0)
\prod_{k=2}^{\ell} {\theta^{(|\gamma_{k}|)}\over \theta}(\sum_{i\in \gamma_{1}\cup \gamma_{2}\cdots \cup \gamma_{k-1}}z_{i})\,.
\end{equation}
In fact, let $\gamma'\in\Gamma_{I}$ and $\gamma=(\gamma', [n]\setminus I)$.
According the identity
\[
(-1)^{\gamma'} (-1)^{n-|I|-1}=(-1)^{\gamma}\,,
\]
we see immediately the right hand side in \eqref{eqnTnintermsofsumovercompositions} satisfies \eqref{eqnTnrecursion}.
The other direction follows from the uniqueness of the solution to the recursion.

The sum-over-partitions formula for $T_{n}^{\Omega}$ in Theorem \ref{thmsumoverpartitions} is quite similar to \eqref{eqnTnintermsofsumovercompositions}.
In this case, the recursion for $T_{n}^{\Omega}$ is similar to \eqref{eqnTnrecursion} and is given by
\[
T_{n}^{\Omega}=\sum_{I:\,\emptyset\subsetneq I\subsetneq [n]}(-1)^{|I|-1}{1\over |I|}G_{I}\cdot T_{[n]-I}^{\Omega}\,.
\]
Simplifying, this is
\begin{equation}\label{eqnTnOmegarecursion}
\sum_{I:\, I\subseteq[n]}{(-1)^{|I|}\over |I|}G_{I}\cdot T_{[n]-I}^{\Omega}=0\,,
\end{equation}
where we have used  Convention \ref{conventionG0T0}.

\section{Finer structure of GW generating series}
\label{secquasielliptic}

% Thanks to the enumerative interpretation of the Eulerian numbers and results in  Appendix
% \ref{appendixEulerianquasielliptic},
We have related the combinatorial quantities $G_{\pi}$ in Theorem \ref{thmsumoverpartitions} to Jacobi theta functions.
In this section we study the finer structure of them based on
their more concrete expressions that we shall derive.
Our main focus in this section is $T_{n}^{\Omega_{\mathrm{GW}}}$, which will be shown to coincide with the
GW generating series $T_{n}$ in \eqref{eqndefTn}.

\subsection{Pure weight structure of $T_{n}^{\Omega_{\mathrm{GW}}}$}

\begin{dfn}
Define
\begin{equation}\label{eqndfnHn}
H_{n}(z):=-(n-1)!\cdot [c^{n-1} ]
\left(
 {1\over 2}{1-e^{-c}\over 1+e^{-c}} +
  {1\over 2} {e^{c}+1\over e^{c}-1}-S(z,c)
  \right)\,,
\end{equation}
where (this is essentially the Szeg\"o kernel \eqref{eqndefofSzegokernel})
\begin{equation}\label{eqndfnSzegoc}
S(z,c)={\theta'(0)\theta(z+c)\over \theta(z)\theta(c)}\,.
\end{equation}
\end{dfn}

We have the following result.

\begin{lem}\label{lembuildingblock}
For $n\geq 1$,
 the following statements hold.
\begin{enumerate}[i).]
\item
One has
\begin{eqnarray*}
  H_{n}(z) =-(1-2^{n}){B_{n}\over n}-{B_{n}\over n}+(n-1)!\cdot [c^{n-1}]S(z,c)\,,
  \end{eqnarray*}
  where $B_{n}$ are the Bernoulli numbers: $B_{1}=-{1\over 2},B_{2}={1\over 6},\cdots$.
\item
In the region $\Omega_{\mathrm{GW}}$, one has
\begin{equation}\label{eqnbuildingblock}
{(-1)^{n-1}\over n}G_{n}^{\mathrm{GW}}
=H_{n}
 \,.
\end{equation}
\end{enumerate}
\end{lem}
\begin{proof}
\begin{enumerate}[i).]
\item
Recall the generating series of Bernoulli numbers
\begin{equation}
\sum_{m\geq 0} {t^{m}\over m!}B_{m}={t \over e^{t}-1}
\,,\quad
 {1\over 2}{e^{t}+1\over e^{t}-1}
 = {1\over t}+ \sum_{n\geq 2} {B_{n}\over n}{ t^{n-1}\over (n-1)!}\,.
\end{equation}
Consider also the generating series of Eulerian polynomials $A_{m}(x),m\geq 0$ (see e.g., \cite{Bona:2012})
\[
\sum_{m\geq 0} {t^{m}\over m!}A_{m}(x)={x-1\over x-e^{(x-1)t}}\,.\]
Substituting  $x$ by $-1$, we see that
\[
\sum_{m\geq 0} A_{m}(-1){t^{m}\over m!}
={2\over 1+e^{-2 t}}={2\over e^{-2t}-1}-{4\over e^{-4t}-1}\,.
\]
This gives
\begin{equation}
 {1\over 2}{1-e^{-t}\over 1+e^{-t}}
 =  \sum_{n\geq 2} ({1\over 2})^{n}A_{n-1}(-1){ t^{n-1}\over (n-1)!}\,,\quad
A_{n-1}(-1)=(-1)^{n-1}(2^{n}-4^{n}) {B_{n}\over n}\,,\quad n\geq 1\,.
\end{equation}
From
\eqref{eqndfnHn}, we obtain
\[
  H_{n}(z)=-(n-1)!\cdot [c^{n-1} ] \left(
\sum_{n\geq 2} (-1)^{n-1}({1\over 2})^{n}A_{n-1}(-1){ c^{n-1}\over (n-1)!}
+
\sum_{n\geq 2} {B_{n}\over n}{ c^{n-1}\over (n-1)!}
-S(z,c)
\right)\,.
\]
Combining these,
one finishes the proof.

\item

For the case $n=1$, we have from \eqref{eqnG1} and \eqref{eqnFourierofZ} that
\[
G_{1}^{\mathrm{GW}}(z)={\theta'\over \theta}=[c^{0}] (S(z,c))=H_{1}(z)\,.
\]
Hence we only need to consider the case $n\geq 2$.
We start by simplifying $G_{n}^{\mathrm{GW}}(z)$.
Recall that the generating series $A_{n}^{(\mathrm{c})}(x)$ of cyclic descent numbers is related to the Eulerian polynomial $A_{n-1}$ (as generating series of descent numbers) %(and of excedant numbers) 
by \cite{Bona:2012} %\cite{Petersen:2008}
 \[
A_{n}^{(\mathrm{c})}(x)
=nxA_{n-1}(x)\,,\quad 	n\geq 2\,.
\]
Therefore,
we see that for $n\geq 2$
\begin{eqnarray*}
G_{n}^{\mathrm{GW}}(u)&=&
\sum_{\sigma\in \mathfrak{S}_{n}}\left((-1)^{\mathrm{cdes}(\sigma)}
({1\over 2})^{n}
+(-1)^{n}\sum_{k\neq 0} {q^{(n- \mathrm{cdes}(\sigma))k}\over (1-q^{k})^{n}}u^{k}
\right)\\
&=&
\sum_{\sigma\in \mathfrak{S}_{n}}\left((-1)^{\mathrm{cdes}(\sigma)}
({1\over 2})^{n}
+\sum_{k\neq  0} {q^{-k\cdot \mathrm{cdes}(\sigma)}\over (1-q^{-k})^{n}}u^{k}
\right)\\
&=&-n({1\over 2})^{n} A_{n-1}(-1)+
\sum_{k\neq 0}  {nq^{-k}A_{n-1}(q^{-k})\over (1-q^{-k})^{n}}u^{k}
\,.
\end{eqnarray*}
By Proposition \ref{propAm},  one then has
\[
G_{n}^{\mathrm{GW}}(u)
% &=&n\left(-({1\over 2})^{n} A_{n-1}(-1)+(-1)^{n}\mathcal{A}_{n}(u)\right)\\
=(-1)^{n}n \left(({1\over 2})^{n}(-1)^{n-1} A_{n-1}(-1)+{B_{n}\over n} - (n-1)!\,\cdot [w^{n-1}] S(z,w) \right)\,.
\]
The claim now follows from the expressions for $H_{n}$ given in (i).
\end{enumerate}
\end{proof}

From \eqref{eqnwLaurentcoefficientofS2}, we have for $n\geq 1$
\begin{eqnarray}\label{eqnBnquasiellipticformula}
 (n-1)!\cdot [c^{n-1}] S(z,c)
 &=&{1\over n}\mathbold{B}_{n}(\mathcal{E}_{1}^{*},\mathcal{E}_{2}^{*},\cdots,
\mathcal{E}_{n}^{*})\nonumber\\
&=&{1\over n}\sum_{a+2b=n}{n\choose a}
\mathbold{B}_{a}((\ln\theta)',\cdots ,(\ln\theta)^{(a)})\cdot 2\mathbb{G}_{2b}\nonumber\\
&=&{1\over n}\sum_{a+2b=n}{n\choose a}
{\theta^{(a)}\over \theta}\cdot 2\mathbb{G}_{2b}
\,,
\end{eqnarray}
where $\mathbold{B}_{n}$ is the complete Bell polynomial.
By Definition 	\eqref{dfnringofquasiellipticfunctions},
this is a quasi-elliptic function of pure weight $n\geq 1$.
We have mentioned earlier in Theorem \ref{thmsumoverpartitions} that
$T_{n}^{\Omega}$ is a quasi-elliptic function of mixed weight with leading weight $n-1$.
For the region  $\Omega_{\mathrm{GW}}$,  the resulting integral $T_{n}^{\Omega_{\mathrm{GW}}}$ is particularly interesting and we have the following result on pure-weight structure.

\begin{thm}\label{thmTngwintermsofB}
Let $\mathbold{B}_{J}, J\subseteq [n]$ be the complete Bell polynomial in the variables
\begin{equation}\label{eqndfgEm}
\mathcal{E}_{a}^{*}=
(\ln \theta)^{(a)}(\prod_{j\in J}u_{j})+2\mathbb{G}_{a}\,,\quad a\geq 1\,,
\end{equation}
where for positive odd integers $a$, $\mathbb{G}_a$ is set to zero and for positive even integers
it is the following Eisenstein summation 
\begin{equation}\label{eqndfnG}
	\mathbb{G}_{a}={1\over 2}\cdot (a-1)!\cdot \sum_{\lambda\in \Lambda_{\tau}-\{(0,0)\}}^{~\quad e}{1\over \lambda^{a}}\,.
\end{equation}
Then
\begin{equation}\label{eqnTngwpureweightstructure}
T_{n}^{\Omega_{\mathrm{GW}}}=
\sum_{j\in [n]}\sum_{\pi\in \Pi_{[n]\setminus \{j\}}}
{\mathbold{B}_{{\pi}}\over |\pi|}\,,
\quad
{\mathbold{B}_{{\pi}}\over |\pi|}:=\prod_{k} {\mathbold{B}_{\pi_{k}}\over |\pi_{k}|}\,,
\end{equation}
where we use the convention
\[
{1\over |\pi_k|}
\mathbold{B}_{\pi_k}=1\,\quad
\mathrm{for}\,\quad \pi_k=\emptyset\,.
\]
In particular, $T_{n}^{\Omega}$ is  a quasi-elliptic function of pure weight $n-1$.
\end{thm}

\begin{proof}

Using \eqref{eqnwLaurentcoefficientofS2}, we can write Lemma \ref{lembuildingblock} (i)
as
\[
H_{\pi_{k}}=
{\mathbold{B}_{\pi_{k}}\over |\pi_{k}|}-{B_{|\pi_{k}|}\over |\pi_{k}|}+(2^{|\pi_{k}|}-1){B_{|\pi_{k}|}\over |\pi_{k}|}\,,
\]
with $\mathbold{B}_{\pi_{k}}$
being a quasi-elliptic function of pure weight $|\pi_{k}|$.
Combining
Lemma \ref{lemv0sumoverpartition}, Lemma \ref{lembuildingblock} and
 Theorem
\ref{thmsumoverpartitions}, we have
\begin{equation}\label{eqnTngwintermsofH}
T_{n}^{\Omega_{\mathrm{GW}}}=2^{2}\cdot [u_0^0]
%\left(
\sum_{{\pi}\in \Pi^{*}_{\widetilde{[n]}}} H_{{\pi}}
%\right)
\,,\quad H_{\pi}:=\prod_{k}H_{\pi_{k}}\,.
\end{equation}
By the computations in the proof of Lemma \ref{lembuildingblock} (ii) or \eqref{eqnwLaurentcoefficientofS}, we see that in \eqref{eqncJformula}
\[
c_{\pi_k}=4 [u_{0}^{0}] H_{\pi_{k}}\,,\quad
[u_{0}^{0}] H_{\pi_{k}}=(2^{|\pi_{k}|}-1){B_{|\pi_{k}|}\over |\pi_{k}|}\,\quad~\mathrm{if}\quad 0\in \pi_{k}\,.
\]

Continuing with \eqref{eqnTngwintermsofH}.
For any partition $\pi=\{\pi_{1},\cdots,\pi_{\ell}\}\in \Pi_{[n]}$, we consider
the set of its compositions/lists $\Gamma_{\pi}$
consisting of ordered collections of the form $\gamma=(\pi', \pi'', \pi''')$,
where the $\pi'$ is a singlet (consisting of only one part of $\pi$) and the other two can be empty in which case the contribution below is defined to be $1$.
Then we can  rewrite \eqref{eqnTngwintermsofH} as
\begin{equation}\label{eqnTngwintermsofB}
T_{n}^{\Omega_{\mathrm{GW}}}=2^{2}
\sum_{\pi\in\Pi_{[n]}} \sum_{\gamma\in \Gamma_{\pi}}
\left(\prod_{k:\, \pi_{k}\in \pi'}(2^{|\widetilde{\pi_{k}}|}-1)  {B_{|\widetilde{\pi_{k}}|}\over |\widetilde{\pi_{k}}|}
\cdot\prod_{k:\, \pi_{k}\in \pi''}(2^{|\pi_{k}|}-2)  {B_{|\pi_{k}|}\over |\pi_{k}|}
\cdot \prod_{k:\, \pi_{k}\in \pi'''}{\mathbold{B}_{\pi_{k}}\over |\pi_{k}|}
\right)\,,
\end{equation}
where $\widetilde{\pi_{k}}:=\pi_{k}\cup\{0\}$.
Since the quasi-elliptic function ${\mathbold{B}_{n}/ n}$ has pure weight $n$,
each summand above has pure weight
\[
w(\pi'''):=
\sum_{k:\, \pi_{k}\in \pi'''}|\pi_{k}|\,.
\]
Fix $\pi'''$, consider the coefficient of $\prod_{k:\, \pi_{k}\in \pi'''}{\mathbold{B}_{\pi_{k}}\over |\pi_{k}|}$ in \eqref{eqnTngwintermsofB}.
It is given by
\[\kappa_{\pi'''}:=
2^2
\sum_{\rho\in \Pi_{[n]\setminus (\bigcup_{k:\, \pi_{k}\in \pi'''} \pi_{k})}}
\sum_{\gamma\in\Gamma_{\rho}}\left(\prod_{k:\, \rho_{k}\in \rho'}(2^{|\widetilde{\rho_{k}}|}-1)  {B_{|\widetilde{\rho_{k}}|}\over |\widetilde{\rho_{k}}|}
\cdot\prod_{k:\, \rho_{k}\in \rho''}(2^{|\rho_{k}|}-2)  {B_{|\rho_{k}|}\over |\rho_{k}|}\right)\,,
\]
where $\gamma=(\rho',\rho'')\in\Gamma_{\rho}$ is a composition and $\rho'$ is a singlet.
Again the product over $\rho''$ is $1$ if $\rho''$ is empty.
Observe that this quantity only depends on
the cardinalities of $\rho_{k}\in \rho$ and is possibly nonzero only if
$
|\rho_{k}|
$ is odd for the one in $\rho'$ and even  for the ones in $\rho''$.
It follows that
\begin{equation}\label{eqnkappasetpartitionformula}
	\kappa_{\pi'''}=
2^2
\sum'_{\rho\in \Pi_{[n]\setminus (\bigcup_{k:\, \pi_{k}\in \pi'''} \pi_{k})}}
\left(\prod_{k:\, \rho_{k}\in \rho'}(2^{|\widetilde{\rho_{k}}|}-1)  {B_{|\widetilde{\rho_{k}}|}\over |\widetilde{\rho_{k}}|}
\cdot\prod_{k:\, \rho_{k}\in \rho''}(2^{|\rho_{k}|}-2)  {B_{|\rho_{k}|}\over |\rho_{k}|}\right)\,,
\end{equation}
where $\sum'$ is over the set of partitions where all blocks except one, say $\rho_{*}$,
have even cardinalities, and $\rho'=\{\rho_{*}\}$.

We now simplify the coefficient $\kappa_{\pi'''}$.
Let
\[
m=
n+1-w(\pi''')=|\widetilde{[n]}\setminus (\cup_{k:\, \pi_{k}\in \pi'''} \pi_{k})|\,.
\]
Then by construction $m$ is an even positive integer.
Denote
\[
x_{i}=(2^{i}-2) {B_{i}\over i}\,,\quad y_{i}=(2^{i}-1) {B_{i}}\,,\quad i\geq 2\,.
\]
Then rewritting the set partition in  \eqref{eqnkappasetpartitionformula} in terms of integer partition and using the sum-over-partitions formula for the Bell polynomials, we have
\begin{eqnarray*}
{1\over 4}\kappa_{\pi'''}&=&\sum_{\sum_{k=1}^{\ell}\lambda_k=n-w} {\ell\over \ell !}{n-w\choose \lambda_1,\cdots,\lambda_{\ell}}{y_{\lambda_1+1}\over \lambda_1+1} \prod_{k\geq 2}x_{\lambda_k}\\
&=&\sum_{\substack{\sum_{k=1}^{\ell}\lambda_k=n+1-w\\
\lambda_{k}\geq 2}} {\ell\over \ell !} {1\over n+1-w}{n+1-w\choose \lambda_1,\cdots,\lambda_{\ell}} y_{\lambda_1} \prod_{k\geq 2}x_{\lambda_k}\\
&=&\sum_{\substack{\sum_{k=1}^{\ell}\lambda_k=n+1-w\\
\lambda_{k}\geq 2}} {\ell \over \ell !\,\ell } {1\over n+1-w}{n+1-w\choose \lambda_1,\cdots,\lambda_{\ell}} [\varepsilon^1 ] \prod_{k\geq 1} (x_{\lambda_k}+\varepsilon y_{\lambda_k})\\
&=& {1\over n+1-w}\sum_{\ell} [\varepsilon^1]\mathbold{B}_{n+1-w,\ell} (x+\varepsilon y)\\
&=& {1\over n+1-w}[\varepsilon^1]\mathbold{B}_{n+1-w} (x+\varepsilon y)\,.
\end{eqnarray*}
Define
\[
f(t;\varepsilon)=\sum_{i\geq 2} (  x_{i}+ \varepsilon y_{i}) {t^{i}\over i!}= f_{0}(t)+ \varepsilon f_1(t)\,.
\]
By the set partition version of the Fa\`a di Bruno formula, we have for $m\geq 2$
\[
e^{- f}(e^{ f})^{(m)}|_{t=0}
%=\sum_{\ell=0}^{m}  \mathbold{B}_{m,\ell}(x+\varepsilon y)
= \mathbold{B}_{m}(x+\varepsilon y)\,.
\]
Observe that $f_{0}|_{t=0}=f_{1}|_{t=0}=0$, we therefore have 
\begin{eqnarray*}
\kappa_{\pi'''}={4\over m}
[\varepsilon^{1}]
\mathbold{B}_{m}(x+\varepsilon y)
&=&{4\over m}[\varepsilon^{1} ]\left(e^{-  f}(e^{f})^{(m)}|_{t=0}\right)\\
&=&{4\over m}\left(- f_{1} e^{- f_0} (e^{ f_0})^{(m)}|_{t=0}
+ e^{-f_0} ( f_1 e^{ f_0})^{(m)}|_{t=0}\right)\,\\
&=&{4\over m} ( f_1 e^{ f_0})^{(m)}|_{t=0}\,.
% &=& {4\over m}\sum_{a=2}^{m}{m\choose a} (2^{a}-1)B_{a}B_{m-a}\,.
\end{eqnarray*}
Direct computations show that
\[
f_0=\ln {t(e^{t}+1)\over 2 (e^{t}-1)}=\ln (
 {t\over 2}+{t\over e^{t}-1})\,,\quad
 f_{1}={t\over 2}-{t\over e^{t}+1}\,,
\]
and thus
\[
e^{f_0} f_1={t^2\over 4}\,.
\]
This tells that
\[
\kappa_{\pi^{'''}}
=\begin{cases}
1\,,\quad w(\pi^{'''})=n-1\,,\\
0 \,, \quad w(\pi^{'''})\neq n-1\,.
\end{cases}
\]
The proof is now completed.

\end{proof}
%%%%
\iffalse
The formula \eqref{eqnTngwpureweightstructure} above can also be written as
\begin{equation}\label{eqnTngwsumoverpermutations}
T_{n}^{\Omega_{\mathrm{GW}}}=
\sum_{j\in [n]}\sum_{\pi\in \Pi_{[n]\setminus \{j\}}} {\mathbold{B}_{\pi}\over |\pi|}
=
\sum_{j\in [n]}\sum_{\sigma\in \mathfrak{S}_{[n]\setminus \{j\}}} {\mathbold{B}_{\sigma}\over |\sigma|!}
\,,\quad
{\mathbold{B}_{\sigma}\over |\sigma|!}:=\prod_{k}{ \mathbold{B}_{\sigma_{k}}\over |\sigma_{k}|!}\,.
\end{equation}
\fi
%%%%
The complete Bell polynomials $\mathbold{B}_{m}$ themselves have combinatorial interpretations as summations over
partitions of $[m]$.
The partition-of-partition structure in \eqref{eqnTngwpureweightstructure}
enjoy a few
interesting combinatorial properties as we shall see later in e.g., the elliptic transformation property of $T_{n}^{\Omega_{\mathrm{GW}}}$.

\subsection{Quasi-ellipticity and quasi-modularity of $T_{n}^{\Omega_{\mathrm{GW}}}$}\label{secelliptictransformation}

A quasi-elliptic function $f\in \mathrm{QE}$ is called elliptic if it is invariant under the elliptic transformations $z\mapsto z+\lambda,\lambda\in \Lambda_{\tau}$.
It is called modular if each of its weight-$m$ component $f_{m}$ satisfies 
\[
f_{m}({z\over c\tau+d},{a\tau+b\over c\tau+d})=(c\tau+d)^{m}f_{m}(z,\tau)\,,\quad
\forall 
\begin{pmatrix}
	a & b\\
	c & d
	\end{pmatrix}\in \mathrm{SL}_{2}(\mathbb{Z})\,.
\]
In general, under the modular transformation by $\mathrm{SL}_{2}(\mathbb{Z})$,
 $f\in \mathrm{QE}$ is only quasi-modular in the sense of \cite{Kaneko:1995}.
 For example, both $\mathbb{G}_{4},\mathbb{G}_{6}$ are modular but $\mathbb{G}_{2}$ is quasi-modular.
 
Using the relation between the Jacobi theta function $\theta$ and Weierstrass elliptic functions, it is direct to check that
\begin{equation}\label{eqnEm*2formula}
	\mathcal{E}_{2}^{*}=({\theta'\over \theta})'+2\mathbb{G}_2=-\wp(z)\,.
\end{equation}
It follows that all of the functions $\mathcal{E}_{m}^{*},m\geq 2$ given in
\eqref{eqndfgEm}
are elliptic and modular.
However, $\mathcal{E}_{1}^{*}$ is not elliptic nor modular: 
\begin{equation}\label{eqnEm*1formula}
	\mathcal{E}_{1}^{*}={\theta'\over \theta}
	=
	\zeta({z})
	-2z\mathbb{G}_2\,,\quad \mathcal{E}_{1}^{*}(qu)=\mathcal{E}_{1}^{*}(u)-1\,.
\end{equation}

The behavior of $T_{n}^{\Omega_{\mathrm{GW}}}$ under elliptic and modular transformations
can be described in a very precise way using  \eqref{eqnTngwpureweightstructure}
in Theorem \ref{thmTngwintermsofB}.
We now derive some properties of it using the  following binomial type relation
satisfied by the complete Bell polynomial
\begin{equation}\label{eqnbinomialBell}
\mathbold{B}_{n}(x+y)=\sum_{a=0}^{n}{n\choose a} \mathbold{B}_{a}(x)\mathbold{B}_{b}(y)\,.
\end{equation}
Recall that  $T_{n}^{\Omega_{\mathrm{GW}}}$ is invariant under the $\mathfrak{S}_{n}$ action on $u_{i},i=1,2,\cdots, n$  from Proposition \ref{propGWpairing=bindpairing} or more concretely
\eqref{eqnTngwpureweightstructure}.
%: this is true since it is symmetric under any transposition.
Therefore, it is enough to discuss the behavior of $T_{n}^{\Omega_{\mathrm{GW}}}$ under the elliptic transformation $u_1\mapsto q u_1$.

 \begin{cor}\label{corTngwdifferenceequation}
 Let the notation be as above. For $n\geq 1$, one has
 \[
 T_{n}^{\Omega_{\mathrm{GW}}}(qu_1,\cdots, u_n)=T_{n}^{\Omega_{\mathrm{GW}}}(u_1,\cdots, u_n)\\
+ \sum_{I:\, \emptyset \subsetneq I\subseteq [n]\setminus \{1\}}(-1)^{I}T_{n-|I|}(u_1u_{I},\cdots, \check{u_{i}},\cdots)\,,
\]
where the notation $\check{u}_{i}$ means the argument $u_{i}, i\in I$ is omitted.
 \end{cor}

\begin{proof}
Using \eqref{eqnbinomialBell}, we see that for $n\geq 1$ one has
\begin{eqnarray*}
{\mathbold{B}_{n}(\mathcal{E}_{1}^{*}, \cdots, \mathcal{E}_{n}^{*})\over n}(qu)
&=&{\mathbold{B}_{n}(\mathcal{E}_{1}^{*}-1, \cdots, \mathcal{E}_{n}^{*})\over n}(u)\\
&=&{1\over n}\sum_{a=0}^{n}{n\choose a}
\mathbold{B}_{a}(-1, \cdots, 0)(u)\cdot\mathbold{B}_{n-a}(\mathcal{E}_{1}^{*}, \cdots, \mathcal{E}_{n}^{*})(u)\\
&=&{1\over n}\sum_{a=0}^{n}{n\choose a}(-1)^{a}\cdot  \mathbold{B}_{n-a}(\mathcal{E}_{1}^{*}, \cdots, \mathcal{E}_{n}^{*})(u)\,.
\end{eqnarray*}
Simplifying, this gives
\begin{equation}\label{eqnBnnquasiellipticity}
{\mathbold{B}_{n}\over n}(qu)={\mathbold{B}_{n}\over n}(u)+\sum_{a=1}^{n-1}{n-1\choose a}(-1)^{a}{\mathbold{B}_{n-a}\over n-a}(u)+{(-1)^{n}\over n}  \,,\quad n\geq 1\,.
\end{equation}
From \eqref{eqnTngwpureweightstructure}, we have
\begin{eqnarray*}
&&
T_{n}^{\Omega_{\mathrm{GW}}}(qu_1,\cdots, u_n)\\
&=& \sum_{j=1}\sum_{\pi\in \Pi_{[n]\setminus \{j\}}} {\mathbold{B}_{\pi}\over |\pi|}(qu_1,\cdots, u_n)
+ \sum_{j\neq 1}\sum_{\pi\in \Pi_{[n]\setminus \{j\}}} {\mathbold{B}_{\pi}\over |\pi|}(qu_1,\cdots, u_n)
\\
&=& \sum_{j=1}\sum_{\pi\in \Pi_{[n]\setminus \{j\}}} {\mathbold{B}_{\pi}\over |\pi|}(u_1,\cdots, u_n)
+ \sum_{j\neq 1}\sum_{\pi\in \Pi_{[n]\setminus \{j\}}} \prod_{k:\,k\neq *}{\mathbold{B}_{\pi_{k}}\over |\pi_{k}|}(u_1,\cdots, u_n)\cdot
{\mathbold{B}_{\pi_{*}}\over |\pi_{*}|}(qu_1,\cdots, u_n)\,,
\end{eqnarray*}
where $\pi_{*}$ is the block in the partition $\pi$ that contains $1$.
Recall that the variable of ${\mathbold{B}_{\pi_{k}}\over |\pi_{k}|}$ is $u_{\pi_{k}}=\prod_{i\in \pi_{k}}u_{i}$.
Plugging \eqref{eqnBnnquasiellipticity} into the above expression, we obtain
\begin{eqnarray*}
&&
T_{n}^{\Omega_{\mathrm{GW}}}(qu_1,\cdots, u_n)-T_{n}^{\Omega_{\mathrm{GW}}}(u_1,\cdots, u_n)\\
&=&
\sum_{j\neq 1}\sum_{\pi\in \Pi_{[n]\setminus \{j\}}} \prod_{k:\,k\neq *}{\mathbold{B}_{\pi_{k}}\over |\pi_{k}|}(u_1,\cdots, u_n)\cdot
\left(
\sum_{a=1}^{|\pi_{*}|-1}{|\pi_{*}|-1\choose a}(-1)^{a}{\mathbold{B}_{|\pi_{*}|-a}\over |\pi_{*}|-a}(u_{\pi_{*}})+{(-1)^{|\pi_{*}|}\over |\pi_{*}|}
\right)\,\\
&=& \sum_{I:\, \emptyset \subsetneq I\subseteq [n]\setminus \{1\}}(-1)^{I}T_{n-|I|}(u_1u_{I},\cdots, \check{u_{i}},\cdots)\,,
\end{eqnarray*}
Here  the second equality follows from the combinatorial explanation
in which the set with cardinality $a$ corresponds to the set $I$.
% intuitive I\cup \{1\} is condensed into $\{1^{*}\}$

\end{proof}

\begin{rem}\label{remSreg}
The result \eqref{eqnBnnquasiellipticity} on the quasi-ellipticity of $\mathbold{B}_{n}$
can be alternatively derived based on the quasi-ellipticity of the Szeg\"o kernel $S(z,c)$.
In fact, from \eqref{eqnautomorphyfactoroftheta} and  \eqref{eqndfnSzegoc} we have
\[
S(z+\pii \tau,c)=e^{-c}S(z,c)\,.
\]
This gives
\begin{eqnarray*}
(m-1)!\cdot [c^{m-1} ] S(z+\pii \tau,c)
%&=&(m-1)!\cdot [c^{m} ] \left( cS(z+\pii \tau,c)\right)\\
&=&(m-1)!\cdot [c^{m} ] \left(e^{-c} \cdot cS(z,c)\right)\\
&=&(m-1)! \left(\sum_{a+b=m,\,a\geq 1} [c^{a}] (cS(z,c)) {(-1)^{b}\over b!}+{(-1)^{m}\over m!}\right)\,.
% &=&(m-1)! \left(\sum_{a+b=m,\,a\geq 1} [c^{a-1}] (S(z,c)) {(-1)^{b}\over b!}+{(-1)^{m}\overm!}\right)\,.
\end{eqnarray*}
From the relation ${\mathbold{B}_{m}\over m}=(m-1)!\cdot [c^{m-1}]S(z,c), m\geq  1$ given in \eqref{eqnwLaurentcoefficientofS2},
we obtain \eqref{eqnBnnquasiellipticity}.
Note that since only those with $m\geq  1$
appear, we can replace $S(z,c)$ by the following regular function in $c$
\[
S^{\mathrm{reg}}(z,c)=S(z,c)-{1\over c}\,.
\]
This then gives
\begin{equation}\label{eqnBnSreg}
{\mathbold{B}_{m}\over m}=\partial_{c}^{m-1}S^{\mathrm{reg}}(z,0)\,.
\end{equation}
% Note the similarity between \eqref{eqnTngwpureweightstructure} and the Hurwitz type formula in e.g., \cite{Pitman:2002}, \cite[Theorem 1]{Hackl:2020}.

\end{rem}

One can also track the dependence in the quasi-modular form $\mathbb{G}_2$
through their occurrences in $\mathcal{E}_{m}^{*}$.
We apply the following rule that measures the quasi-modularity exhibited in  \eqref{eqnEm*1formula}
\begin{equation}\label{eqnpartialEm*}
\partial_{\mathbb{G}_2} \mathcal{E}_{1}^{*}=-2z \,,\quad
\partial_{\mathbb{G}_2} \mathcal{E}_{m}^{*}=0\,,~m\geq 2\,.
\end{equation}
This leads to another proof of the following modular anomaly structure proved\footnote{This also follows from the holomorphic limit of \cite[Theorem 3.1]{Li:2022regularized}, with Theorem \ref{thmTn=Acycleintegralintro} guaranteed.} in \cite[Lemma 4.2.2]{Pixton:2008} using the determinant formula
 \eqref{eqngeneratingseriesformulaintro}.
\begin{cor}
For any $n\geq 1$, one has
\[
\partial_{\mathbb{G}_2}T_{n}^{\Omega_{\mathrm{GW}}}=-2\sum_{i<j}(z_i+z_j)T_{n-1}^{\mathrm{GW}}(z_{i}+z_{j},\cdots, \check{z}_{i},\cdots \check{z}_{j},\cdots)\,.
\]
\end{cor}
\begin{proof}
From \eqref{eqnpartialEm*} and Theorem \ref{thmSzegokernel} (iii) or the combinatorial interpretations  \eqref{eqnwLaurentcoefficientofS2} of the complete Bell polynomials, we see that
\[
\partial_{\mathbb{G}_2}\mathbold{B}_{m}=
-2z \cdot m\cdot \mathbold{B}_{m-1}\,,\quad m\geq 1\,.
\]
Hence, when $|\pi_{k}|\geq 2$ one has
\begin{eqnarray*}
\partial_{\mathbb{G}_2}{\mathbold{B}_{|\pi_{k}|}\over |\pi_{k}|}(\sum_{i\in \pi_{k}}z_{i})
&=&
-2 \sum_{i\in \pi_{k}}z_{i}\cdot    (|\pi_{k}|-1) {\mathbold{B}_{|\pi_{k}|-1}\over |\pi_{k}|-1}(\sum_{i\in \pi_{k}}z_{i})\\
&=&
-2 \sum_{i\in \pi_{k}}z_{i}\cdot \sum_{i'\in\pi_{k}:\,i'\neq i}
{\mathbold{B}_{|\pi_{k}|-1}\over |\pi_{k}|-1}(\sum_{i\in \pi_{k}}z_{i})\,.
\end{eqnarray*}
Each term $\sum_{i'\in\pi_{k}:\,i'\neq i}
{\mathbold{B}_{|\pi_{k}|-1}\over |\pi_{k}|-1}(\sum_{i\in \pi_{k}}z_{i})$
above corresponds to the contribution that arises from first combining the two elements $i,i'$
in $\pi_{k}$ into a single element then summing over partitions of the resulting new set.
It follows that for $|\pi_{k}|\geq 2$
\[
\partial_{\mathbb{G}_2}{\mathbold{B}_{|\pi_{k}|}\over |\pi_{k}|}(\cdots, z_{i},\cdots, z_{i'},\cdots )
=
-2 \sum_{\{i,i'\}\subseteq  \pi_{k}}(z_{i}+z_{i'})
{\mathbold{B}_{|\pi_{k}|-1}\over |\pi_{k}|-1}(z_{i}+z_{i'},\cdots, \check{z}_{i},\cdots, \check{z}_{i'},\cdots )\,.
\]
When $|\pi_{k}|=1$, the contribution is the $-2z_{i}$ multiple of the contribution $1$ assigned to the set
$\{z_{i}+z_{j}\}$ where $j\in [n]\setminus \cup_{k}\pi_{k}$ in \eqref{eqnTngwpureweightstructure}.
Combining these, we obtain the desired claim.

\end{proof}

\subsubsection{Relation between $T_{n}$ and $T_{n}^{\Omega_{\mathrm{GW}}}$: difference equation and singular behavior}\label{secTnTgw}

The GW generating series $T_{n}$ in \eqref{eqndefTn} satisfy a few properties as follows.
One has $T_1=1$ and
\begin{itemize}
\item difference equation  \cite[Theorem 8.1, Corollary 10.2]{Bloch:2000}
\begin{eqnarray}\label{eqnTndifferenceequation}
T_{n}(qu_1,\cdots, u_n)
&=&
T_{n}(u_1,\cdots, u_n)\nonumber\\
&+&\sum_{I=\{i_1,\cdots, i_s\}\subseteq [n]\setminus \{1\}} (-1)^{|I|} T_{n-|I|}(u_1 u_{I}, \cdots,\check{u}_{i_1},\cdots, \check{u}_{i_s},\cdots)\,,
\end{eqnarray}
where a notation such as $\check{u}_{i_1}$ means the variable $u_{i_1}$ is omitted.
In addition, $T_{n}$ is symmetric under  the $\mathfrak{S}_{n}$ action on $u_{i},i=1,2,\cdots, n$.
\item
singular behavior \cite[Theorem 11.1, Equation 11.2,  Equation 11.3]{Bloch:2000}
\begin{equation}\label{eqnTnsimplepole}
T_{n}(u_1,\cdots, u_n)={T_{n-1}(u_2,\cdots, u_n)\over u_1-1}+\mathrm{regular ~terms}\,,\quad u_1\rightarrow 1\,,
\end{equation}
and $T_{n}(u_1,\cdots, u_n)$ is regular on the divisor $u_1 u_2\cdots u_{k}=1, k\geq 2$.
\end{itemize}
See also  \cite[Proposition 1, Proposition 2, Theorem 5, Theorem 6]{Okounkov:2001}
for representation-theoretic proofs of these properties.\\

From the viewpoint of recursively solving the system, the difference equation and singular behavior   determine the function $T_{n}$ up to constant, which is in turn fixed by the vanishing condition
 \cite[Theorem 11.2, Theorem 11.3, Equation 11.3]{Bloch:2000}
 \begin{equation}\label{eqnTnvanishingcondition}
T_{n}|_{u_1\cdots u_{n}=1}=0\,,\quad n\geq 2\,.
\end{equation}

We now provide a proof that $T_{n}=T_{n}^{\Omega_{\mathrm{GW}}}$
as follows.
\begin{prop}\label{propTn=Tngw}
Let the notation be as above. We have
\[
T_{n}=T_{n}^{\Omega_{\mathrm{GW}}}\,.
\]
\end{prop}
\begin{proof}
 Observe that besides the vanishing condition \eqref{eqnTnvanishingcondition}, the sequence of constants can also be recursively fixed by imposing the condition that $T_{n},n\geq 1$ are quasi-elliptic functions of pure weight. This condition has been established in Theorem \ref{thmTngwintermsofB}. Therefore to check $T_{n}=T_{n}^{\Omega_{\mathrm{GW}}}$ it is enough to prove \eqref{eqnTndifferenceequation}, \eqref{eqnTnsimplepole} are satisfied by 
 $T_{n}^{\Omega_{\mathrm{GW}}}$.

The difference equation \eqref{eqnTndifferenceequation} is nothing but the quasi-ellipticity of $T_{n}^{\Omega_{\mathrm{GW}}}$ given in
 Corollary
\ref{corTngwdifferenceequation}.
We now prove the condition
 \eqref{eqnTnsimplepole}. By \eqref{eqnwLaurentcoefficientofS2}, $\mathbold{B}_{m}(u)$ is a polynomial in ${\theta^{(a)}\over \theta}$, it follows from \eqref{eqnTngwpureweightstructure} that the possible singularities of $T_{n}$ are at worst simple poles along the divisors
   $u_{I}:=\prod_{i\in I}u_{i}=1\,, I\subsetneq [n]$.
   The formula \eqref{eqnTngwpureweightstructure}  also gives the desired simple pole structure \eqref{eqnTnsimplepole} of $T_{n}$ on $u_i=1$.
  For any divisor  $u_{I}=1\,, I\subsetneq [n]$ with $|I|>1$,
  the terms in \eqref{eqnTngwpureweightstructure}  which are possibly singular along this divisor
  arise from the term $B_{I}/|I|$.
  That it is regular along $u_{I}=1$ follows from \eqref{eqnwLaurentcoefficientofS}.

\end{proof}

\begin{rem}\label{remsecondcheck}
   As shown in  \cite[Section 10, Theorem 11.2]{Bloch:2000} using the recursive structure
in the expression \eqref{eqngeneratingseriesformulaintro} for $F_{n}$ (\cite[Equation 7.7]{Bloch:2000}), the vanishing condition
 \eqref{eqnTnvanishingcondition} actually implies
\eqref{eqnTnsimplepole} once \eqref{eqnTndifferenceequation} is granted. Furthermore, \eqref{eqnTnvanishingcondition}
can be weakened to
\begin{equation}\label{eqnTnvanishingconditionweakened}
T_{n}|_{u_1\cdots u_{n}=1}~\mathrm{is~regular}\,, \quad\mathrm{or}\quad
T_{n}|_{u_1\cdots u_{n}=1}\rightarrow 0~\mathrm{as}~u_1\rightarrow 1\,.
\end{equation}
Applying the same reasoning there to the formula \eqref{eqnTngwpureweightstructure} for $T_{n}^{\Omega_{\mathrm{GW}}}$, to check
$T_{n}=T_{n}^{\Omega_{\mathrm{GW}}}$ it is enough to prove  \eqref{eqnTndifferenceequation}
and any of \eqref{eqnTnvanishingcondition}, \eqref{eqnTnvanishingconditionweakened}.
It should be possible to prove
%%%%  \eqref{eqnTnrecursion} and 
%%%% 
\eqref{eqnTnvanishingcondition} or
 \eqref{eqnTnvanishingconditionweakened}
based on the sum-over-partitions formula \eqref{eqnTngwpureweightstructure} in a way
similar to the reasoning of  \cite[Theorem 11.3]{Bloch:2000}, as we did in Example \ref{exTn3case}.
\end{rem}

We can finally prove the main result Theorem \ref{thmTn=Acycleintegralintro}, which is recalled below for convenience.

\begin{thm}\label{thmTn=Acycleintegral}
	Following the notation as above,
	let the linear coordinates on $\mathsf{J}_{[n]}$ be $z_1,\cdots, z_{n}$ respectively, and $q=e^{\pii\tau}$.
Then one has
	\begin{equation}\label{eqnTnintermsofintergral}
	T_{n}(z_1,\cdots, z_{n};q)
	=\Big \langle ~\gamma_{n}^{\mathrm{GW}}\,, ~ \varphi_{n}^{\mathrm{GW}} ~\Big\rangle \,.
	\end{equation}
\end{thm}
\begin{proof}
This follows from Lemma \ref{propGWpairing=bindpairing}, Theorem \ref{thmTngwintermsofB}, and  Proposition \ref{propTn=Tngw}.
\end{proof}

\subsubsection{Relation between $T_{n}$ and $T_{n}^{\Omega_{\mathrm{Wick}}}$}

Quite a few results, such as the derivation of the determinant formula \eqref{eqnFaymulti-secant} for $C_{2n}$, are
more conveniently formulated under
the ordering $\Omega_{\mathrm{Wick}}$ instead of $\Omega_{\mathrm{GW}}$, see e.g., \cite[Theorem 1]{Okounkov:2001}.
For this reason, we now relate
$T_{n}^{\Omega_{\mathrm{Wick}}}$ to $T_{n}^{\Omega_{\mathrm{GW}}}=T_{n}$.

In contrast to \eqref{eqnFnaszeromode} in which
\[
T_{n}=\langle \mathbold{T}_{n} \rangle
=[v^{0}] \langle \psi(u_{1}v_{1})\psi^{*}(v_{1})\cdots \psi(u_{n}v_{n})\psi^{*}(v_{n}) \rangle
\,,\quad \,\quad  (u,v)\in \Omega_{\mathrm{GW}}\,,
\]
the
 correlation functions $T_{n}^{\Omega_{\mathrm{Wick}}}$ are defined to be
$\langle \mathbold{T}^{\mathrm{Wick}}\rangle$, where
\[
\mathbold{T}_{n}^{\mathrm{Wick}}=[v^0]\left(
\psi(u_1 v_1)\cdots \psi(u_{n} v_n)\psi^{*}(v_{n})\cdots \psi^{*}(v_{1})\right)\,,\quad (u,v)\in \Omega_{\mathrm{Wick}}\,.
\]
Note that $[\mathbold{T}(u_{i}), \mathbold{T}(u_{j})]=0$, in consistency with
the $\mathfrak{S}_{n}$ symmetry of $T_{n}$ mentioned in Proposition \ref{propGWpairing=bindpairing}.

Similar to \cite[Section 3.2.2]{Okounkov:2001}, %\cite{Okounkov:Virasoro}[Equation 3.12]
we can relate
$T_{n}$ and $T_{n}^{\Omega_{\mathrm{Wick}}}$ by showing the following relation on the level of the
corresponding operators $\mathbold{T}_{n}, \mathbold{T}_{n}^{\mathrm{Wick}}$.
\begin{prop}\label{proprelationbetweenoperatorsTnTnA}
Let the notation be as above.
Then one has
\[
\mathbold{T}_{[n]}^{\mathrm{Wick}}=\sum_{\pi\in \Pi_{[n]}} \prod_{k=1}^{\ell}(-1)^{\pi_{k}}\mathbold{T}_{\pi_{k}}(\prod_{i\in \pi_{k}}u_{i})\,.
\]
\end{prop}
\begin{proof}
We prove this by induction.
It is easy to check by direct computations that
\begin{equation}\label{eqncommutationTpsi}
[\mathbold{T}(u),\psi^{*}(v)]=-\psi^{*}({v\over u})\,,\quad
[v^{0}] \psi(u_{j}v_{j})\psi^{*}({v_{j}\over u_{i}})=\mathbold{T}(u_{i}u_{j})\,.
\end{equation}
Based on these one can show the statement for the $n=2$ case, namely
\[
\mathbold{T}_{2}^{\mathrm{Wick}} =\mathbold{T}(u_1)\mathbold{T}(u_2)-\mathbold{T}(u_1 u_2)\,.
\]
Now assume the statement is true for the $\leq n$ case. We now prove it for the $n+1$ case.

By definition, one has
\begin{eqnarray*}
\mathbold{T}_{[n+1]}^{\mathrm{Wick}}=[v^{0}] \left(\psi(u_{1}v_{1}) \mathbold{T}_{[n+1]\setminus\{1\}}^{(A)}
 \psi^{*}(v_{1})\right)\,.
\end{eqnarray*}
By the induction hypothesis, it follows that
\begin{eqnarray*}
\mathbold{T}_{[n+1]}^{\mathrm{Wick}}=[v^{0}]
\left(\psi(u_{1}v_{1})\sum_{\pi\in \Pi_{[n+1]\setminus\{1\}}} \prod_{k=1}^{\ell}(-1)^{\pi_{k}}\mathbold{T}_{\pi_{k}} \psi^{*}(v_{1})
\right)\,.
\end{eqnarray*}
By the first one of the commutation relations \eqref{eqncommutationTpsi} above (see \cite[Equation 3.11]{Okounkov:2001}), one has
\[
\mathbold{T}(x_{m})\cdots \mathbold{T}(x_{1})\psi_{s}^{*}
=
\sum_{I\subseteq [m]} (-1)^{|I|}\prod_{i\in I}x_{i}^{s}\cdot \psi_{s}^{*}\prod_{i\notin I}\mathbold{T}(x_{i})\,.
\]
Plugging this, one has
\begin{eqnarray*}
\mathbold{T}_{[n+1]}^{\mathrm{Wick}}&=&[v^{0}]
\sum_{r,s}\psi_{r}\psi_{s}^{*}(u_{1}v_{1})^{-r}v_{1}^{-s}
\sum_{\pi\in \Pi_{[n+1]\setminus\{1\}}}(-1)^{\pi}\prod_{K\subseteq [\ell]} (-1)^{|K|} (\prod_{k\in K}u_{\pi_{k}})^{s}\prod_{k\notin K} \mathbold{T}_{\pi_{k}}\\
&=& \sum_{\pi\in \Pi_{[n+1]\setminus\{1\}}}(-1)^{\pi}\prod_{K\subseteq [\ell]} (-1)^{|K|} \mathbold{T}(u_{1}\prod_{k\in K}u_{\pi_{k}}) \prod_{k\notin K} \mathbold{T}_{\pi_{k}}\\
&=& \sum_{\pi'\in \Pi_{[n+1]}}(-1)^{\pi'} \mathbold{T}_{\pi'}\,.
\end{eqnarray*}
This finishes the inductive proof.
\end{proof}

The inverse relation of the one in Proposition \ref{proprelationbetweenoperatorsTnTnA} can be obtained as follows.
Define
\[
\mu: \Pi_{[n]}\rightarrow \mathbb{R}\,,\quad
\pi\mapsto (-1)^{\pi}\prod_{k}(|\pi_{k}|-1)!\,.
\]
Then by the M\"obius inversion formula on the poset $\Pi_{[n]}$ ordered by the refinement relation, we have
\[
(-1)^{n-1}\mathbold{T}_{n}=\sum_{\pi\in \Pi_{[n]}} \mu(\pi)\mathbold{T}_{\pi}^{\mathrm{Wick}}\,.
\]

\subsection{Elliptic and modular completion of $T_{n}$}
\label{secellipticcompletion}

We now discuss the elliptic completion of $T_{n}=T_{n}^{\Omega_{\mathrm{GW}}}$
following the discussions in Section
\ref{secelliptictransformation}.
Standard facts in elliptic function theory (see e.g., \cite{Eichler:1985}) tells that
the elliptic and modular  completion of
$\mathcal{E}_{1}^{*}$ is given by\footnote{
See \cite[Remark 2.4]{Li:2022regularized} for studies about this function and its relevance in
performing regularized integrals, as well as Remark \ref{remoriginofanomaly} in Appendix \ref{appendixorderings} for related discussions on its appearance in correlation functions.}
\[
\widehat{\mathcal{E}_{1}^{*}}=
{\widehat{\theta}'\over \widehat{\theta}}=
{\theta'\over \theta}+\mathbold{A}\,,\quad
\widehat{\theta}=e^{-2\pi {(\mathrm{im}(\, z/\pii))^2\over \mathrm{im}\,\tau}}\theta\,,\quad
\mathbold{A}(z)={\mathrm{im}\,(z/\pii)\over \mathrm{im}\,\tau}\,.
\]

Using the set partition meaning of
\[
{\mathbold{B}}_{m}(\widehat{\mathcal{E}_{1}^{*}},\mathcal{E}_{2}^{*},\cdots ,\mathcal{E}_{m}^{*} )
=\sum_{a=0}^{m}{m\choose a} {\mathbold{B}}_{n-a}(0,\mathcal{E}_{2}^{*},\cdots ,\mathcal{E}_{m}^{*} )\cdot
 {\mathbold{B}}_{a}(\widehat{\mathcal{E}_{1}^{*}},0,\cdots ,0)\,,
 \]
 we see that
 the elliptic and modular  completion
$\widehat{\mathbold{B}}_{m}$ is obtained from
$\mathbold{B}_{m}$
by replacing the contribution $\mathcal{E}_{1}^{*}$
of the 1-blocks in the partitions of the set $[m]$ by $\widehat{\mathcal{E}}_{1}^{*}$.
On the level of their generating series, from \eqref{eqnwLaurentcoefficientofS2} we see that the elliptic and modular  completion corresponds to throughout replacing $\theta(z)$ by $\widehat{\theta}(z)$ and thus $S(z,c)$
by
\begin{equation}\label{eqnShat}
\widehat{S}(z,c)=e^{\mathbold{A}c}S(z,c)\,.
\end{equation}

\subsubsection{Completion based on sum-over-partitions formula}

We have the following relation analogous to \eqref{eqnBnnquasiellipticity}  in Corollary
\ref{corTngwdifferenceequation}.

\begin{cor}\label{corTngwhat}
For $n\geq 1$,
the elliptic and modular  completion $\widehat{T}_{n}$ is given by
\begin{equation}\label{eqnTngwhat}
 \widehat{T}_{n}=
\sum_{j\in [n]}\sum_{\pi\in \Pi_{[n]\setminus \{j\}}}
\prod_{k}{1\over |\pi_{k}|}\sum_{\pi_{k}=\pi_{k}'\sqcup \pi_{k}''}\mathbold{B}_{\pi'_{k}}(u_{\pi_{k}})
\mathbold{A}^{|\pi''_{k}|}(u_{\pi_{k}})\,,
\end{equation}
where the notation $\pi_{k}=\pi_{k}'\sqcup \pi_{k}''$ means the decomposition of $\pi_{k}$
into the union of two disjoint possibly empty subsets $\pi_{k}', \pi_{k}''$.
\end{cor}

\begin{proof}
By the binomial identity \eqref{eqnbinomialBell},
we have
\[
\mathbold{B}_{\pi_{k}}(u_{\pi_{k}})=\sum_{\gamma=(\pi_{k}',\pi_{k}'')\in \Gamma_{\pi}}
\mathbold{B}_{\pi'_{k}}(u_{\pi_{k}})
\mathbold{A}^{|\pi''_{k}|}(u_{\pi_{k}})\,,
\]
where $\gamma$ is a composition of $\pi$ with two possibly empty blocks.
Plugging this into \eqref{eqnTngwpureweightstructure} in Theorem \ref{thmTngwintermsofB},
we obtain the desired claim.

\end{proof}

\subsubsection{Regularized integrals of $C_{2n}$}

We have shown in Theorem \ref{thmTn=Acycleintegral}
that the enumerative generating series $T_{n}$ coincides with the ordered $A$-cycle integral 
\[
T_{n}^{\Omega_{\mathrm{GW}}}=\Big \langle ~\gamma_{n}^{\mathrm{GW}}\,, ~ \varphi_{n}^{\mathrm{GW}} ~\Big\rangle =\int_{A_{[n]}\subseteq \Omega_{\mathrm{GW}}}\varpi
%\theta(u_{[n]})\cdot \Theta_{2n}
\,,\quad
\varpi=\theta(u_{[n]})\cdot \Theta_{2n}\,.
\]
Hereafter for notational simplicity, in performing integrations we suppress notations for the trivialization $\bigwedge_{k=1}^{n}dz_{k}$
of the canonical bundle and for the volume form.
 
The notion of regularized integrals
 introduced in \cite{Li:2020regularized, Li:2022regularized}, which has shown to be closely related 
 to ordered $A$-cycle integrals,
can be adapted to the current case.
In fact, let (cf. \cite[Section 3.3.1]{Eskin:2006})
\begin{equation}\label{eqndivisorDelta}
D_{ij}=\Big\{ (w_1,\cdots, w_n)|~w_{i}-w_{j}+\sum_{k\in K} z_{k}-\sum_{\ell\in L} z_{\ell}\in \Lambda_{\tau}~\text{for some}~ K\subseteq [n], L\subseteq [n]\Big\}\,,
\end{equation}
and $D=\bigcup_{i,j}D_{ij}$.
Then
similar to the discussion on hyperplane arrangement divisors in \cite[Lemma 2.35]{Li:2020regularized}, one can find a coherent way of performing reduction of pole order
and define the regularized integral, while  \cite[Proposition 2.27]{Li:2020regularized}
 generalizes directly.
The construction in \cite[Lemma 3.27]{Li:2020regularized} of the ($u$-dependent) integration domain on which no poles lie in the boundary also carries
to the current case.
We also enlarge the rings $\mathcal{J}_{n},\widetilde{\mathcal{J}}_{n},
\widehat{\mathcal{J}}_{n}$
of elliptic, quasi-elliptic, and almost-elliptic functions, by adjoining to $\mathcal{J}_{n}$ rational functions of theta functions
that are meromorphic and have at worst poles along $D$.\\

In what follows, we assume $1<|u_{i}| \ll
|q|^{-1},1\leq i\leq n$ and consider the following region (see \cite[Section 3.3.1]{Eskin:2006}) that is compatible with $\Omega_{\mathrm{GW}}$
\begin{equation}\label{eqndomainoforderedAcycleintegralsimplified}
	\Big\{(u,v)~|~|q|<|\prod_{k\in K}u_{k}^{\pm 1}\cdot  {v_{i}\over v_{j}}|<1\,~\text{for}~ i>j\Big\}\,,
\end{equation}
where $K$ is any subset of $[n]$.
%By the triviality of automorphy in $w_{k},k=1,2,\cdots, n$, $C_{2n}$ can be regarded
%as an elliptic function in these variables.
Due to the simple-pole structure for $\varpi$, we have the following result 
that  relates
the ordered $A$-cycle integrals with the regularized integrals of $\varpi$.
This is
analogous to
\cite[Lemma 3.26, Lemma 3.37]{Li:2020regularized} and \cite[Lemma A.5]{Li:2022regularized}.

\begin{lem}\label{lemregularizedintegralC2n}
	One has
	\begin{equation}\label{eqnregularizedintegralC2n}
		\dashint_{E_{[n]}} 
		\varpi
		=\int_{A_{[n]}} \varpi+\sum_{ J:
			\,J\neq\emptyset} (\pii)^{|J|}  \sum_{\mathfrak{r}_{J}} \kappa_{J,\mathfrak{r}_{J}} \cdot \left(\int_{A_{[n]-J}}
		R^{J}_{\mathfrak{r}_{J}}\right)^{\prec}\varpi\,,
	\end{equation}
	where
	\begin{itemize}
		\item
		$\prec$ is the reversed ordering by magnitude, as
		in $\int_{A_{[n]}}=\int_{A_{n}}\cdots\int_{A_2}\int_{A_1}$.
		\item
		$J=(j_1,\cdots, j_\ell)$ is an non-recurring sequence: $j_1\prec j_2\prec\cdots \prec j_\ell$ with cardinality $|J|=\ell$.
		\item
		$\mathfrak{r}_{J}=(\mathfrak{r}_{j_1},\cdots, \mathfrak{r}_{j_\ell})$ satisfies
		%\begin{equation}\label{eqnconstraint1appendix}
		$j_a\prec \mathfrak{r}_{j_{a}}\,, \mathfrak{r}_{j_{a}}\in [n]$ for
		$a=1,2,\cdots,\ell$.
		%\end{equation}
		\item
		%$\kappa_{J,\mathfrak{r}_{J}}=\int_{0}^{1}dx_{\mathfrak{r}_{j_{\ell}}}\int_{0}^{x_{\mathfrak{r}_{j_{\ell}}}}dx_{j_{\ell}}\cdots \int_{0}^{x_{\mathfrak{r}_{j_{1}}}}dx_{j_{1}}$.
		$\kappa_{J,\mathfrak{r}_{J}}$ is a degree-$|J|$ polynomial in the generators $\mathbold{A}(z_{i}),i\in J$.
		\item 
		Denote by $R^{(i)}_{j}$ the summation of residues at
		$w_{i}\in D_{ij}$, with $D_{ij}$ defined in \eqref{eqndivisorDelta}. Then
		$
		R^{J}_{\mathfrak{r}_{J}}=R^{(j_\ell)}_{\mathfrak{r}_{j_\ell}}\circ\cdots \circ R^{(j_1)}_{\mathfrak{r}_{j_1}}$
		and the notation
		$ \left(\int_{A_{[n]-J}}
		R^{J}_{\mathfrak{r}_{J}}\right)^{\prec}$ stands for the operator arranged according the ordering $\prec$.
	\end{itemize}
\end{lem}
\begin{proof}
	The proof is similar to the one for  \cite[Lemma A. 5]{Li:2022regularized}.
	We use the linear coordinate $w$ on $\mathbb{C}$ for which the volume form is
	$\mathrm{vol}={\mathbf{i}\over 2\,\mathrm{im}\,\tau}{dw\wedge d\overline{w}}$.
	For any almost-elliptic function $\Psi=\sum_{k} \Psi_{k} ({\overline{w}_i-{w}_i\over \overline{\tau}-\tau})^k\,,
	$
	from \cite[Lemma 3.26]{Li:2020regularized}  one has
	\begin{eqnarray*}
		&&\dashint_{E_{i}}\Psi\\
		&=&\sum_{k} {1\over k+1}\int_{A} \Psi_{k}
		+\pii \sum_{k}{1\over k+1}\sum_{r} R^{(i)}_{r}  \left(({\overline{w}_i-{w}_i\over \overline{\tau}-\tau})^{k+1} \Psi_{k}\right)\\
		&=& \sum_{k} {1\over k+1}\int_{A} \Psi_{k}
		+\pii \sum_{k}{1\over k+1}\sum_{r} R^{(i)}_{r}  \left(({\overline{w}_i-\overline{w}_r+\overline{w}_r-w_{r}+w_{r}-w_{i}\over \overline{\tau}-\tau})^{k+1} \Psi_{k}\right)\nonumber\\
		&=&
		\sum_{k} {1\over k+1}\int_{A} \Psi_{k}
		+\pii \sum_{k}{1\over k+1}\sum_{\ell:\, 0\leq \ell\leq k+1}{k+1\choose \ell}({1\over \overline{\tau}-\tau})^{\ell}\sum_{r} ({\overline{w}_r-w_{r}\over \overline{\tau}-\tau})^{k+1-\ell}  R^{(i)}_{r}  \left(
		(w_{r}-w_{i})^{\ell}
		\Psi_{k}\right)\nonumber\,.
	\end{eqnarray*}
	Here in the last equality we have used $\mathrm{Res}_{w=0}(\overline{w} f)=0$ for any meromorphic function $f$.
	Assume now $\Psi$ has at most simple poles, then one has
	\begin{eqnarray*}
		\dashint_{E_{i}}\Psi
		&=&
		\sum_{k} {1\over k+1}\int_{A} \Psi_{k}
		+\pii \sum_{k}{1\over k+1}  \sum_{r}
		({\overline{w}_r-w_{r}\over \overline{\tau}-\tau})^{k+1}
		R^{(i)}_{r}(\Psi_{k})\,.
	\end{eqnarray*}

	By the results in \cite{Li:2020regularized},
	iterated regularized integrals of $\varpi$, which can be easily checked to have at most simple poles,  are almost-elliptic
	and
	have at most simple poles.
	One then expands iterated regularized integrals of
	$\varpi$
	in terms of iterated $A$-cycle integrations and iterated residue operations.
	The structure of the divisor $D$ in
	\eqref{eqndivisorDelta} tells that
	$\kappa_{J,\mathfrak{r}_{J}}$ has the desired polynomial structure.
	
\end{proof}

We have seen in Proposition \ref{propGWpairing=bindpairing} that the
$\int_{A_{[n]}\subseteq \Omega_{\mathrm{bound}}}\varpi$ is invariant under permutations on $(z_{1},\cdots, z_{n})$.
This is consistent with its formulation as iterated $A$-cycle integrations: the invariance follows by a change of bound variables  as in \eqref{eqnactiononboundvariables} in the integration and the invariance of $\varpi$ under this change.
By \cite{Li:2020regularized}, for such orderings the average
  of them gives a quasi-elliptic function of pure weight $n-1$, hence each individual ordered $A$-cycle integral
itself is so. See Appendix \ref{secregularizedintegralT3} for a sample of computations.
Furthermore,  its elliptic and modular completion
is given by the regularized integral $\dashint_{E_{[n]}}\varpi$.
The structure for $\kappa_{J,\mathfrak{r}_{J}}$ agrees
with the discussions in Section \ref{secellipticcompletion} that for a bound ordering
both the elliptic and modular anomaly
enter through $\theta'/\theta$
and are cured by adding terms in $\mathbold{A}$ (no higher order terms in $\mathbold{Y}=1/\mathrm{im} \tau$).

\begin{rem}\label{remellipticompletion}
	According to its automorphy behavior, the elliptic and modular completion $ \widehat{T}_{n}$ in
	Corollary \ref{corTngwhat}
descends 	from \eqref{eqnchomologicalpairing} to
	a smooth section of 
	${\mathcal{P}^{\circ}_{[n]}}({\mathsf{J}^{\circ}_{[n]}}\times \Theta)|_{{\mathsf{J}^{\circ}_{[n]}}\times \Theta}$.
	Furthermore, the singularity structure \eqref{eqnTnsimplepole}
	tells that it only has first order pole along
$
	\sum_{j=1}^{n} 0\times \mathsf{J}_{[n]\setminus \{j\}}
	$, here recall that $0$ stands for the divisor given by the origin on the Jacobian $J$.
	Summarizing, one has
	\begin{eqnarray*}
	&&\text{elliptic completion of}~\Big\langle~ \gamma_{n}^{\mathrm{GW}},\varphi_{n}^{\mathrm{GW}}~\Big\rangle\\
	&=&\dashint_{E_{[n]}}\Phi^{\mathrm{GW}}_{n}
	\quad \in~ \Gamma\left(\mathsf{J}_{[n]}, \mathcal{C}^{\infty}\otimes  \left(\mathcal{P}_{[n]}(\mathsf{J}_{[n]}\times \Theta)|_{\mathsf{J}_{[n]}\times \Theta}\right) (\sum_{j=1}^{n} 0\times \mathsf{J}_{[n]\setminus \{j\}})\right)\,.
	\end{eqnarray*}
Similar to the ordered $A$-cycle integral $\langle \gamma_{n}^{\mathrm{GW}},\varphi_{n}^{\mathrm{GW}}\rangle$, 
the regularized integral  $\dashint_{E_{[n]}}\Phi^{\mathrm{GW}}_{n}$ (and the notion of regularized integrals \cite{Li:2020regularized} in general) also admits a purely cohomological formulation. 
The details will appear in an incoming work.
%See \cite{Zhou:2023cohomologicalpairings} for details.
\end{rem}

\subsection*{Speculations on possible connection to stratification}
The sum-over-partitions structure in Theorem \ref{thmTngwintermsofB} gives
\begin{equation}\label{eqnTngwsumoverpermutationsstratification}
	(q)_{\infty}F_{n}= 	\sum_{j\in [n]}\left( {1\over \theta(z_{[n]})}\cdot 
\sum_{\pi\in \Pi_{[n]\setminus \{j\}}} {\mathbold{B}_{\pi}\over |\pi|}\right)\,.	
\end{equation}
This structure seems to beg for more geometric explanations.
Some experiments on the $n=3$ case (see Appendix \ref{secregularizedintegralT3})
indicate that
this structure is different from the natural stratification provided by 
ordered $A$-cycle integrals and regularized integrals as
given in  \cite{Li:2020regularized, Li:2022regularized}.
At this stage we are not able to provide a nice geometric explanation for \eqref{eqnTngwsumoverpermutationsstratification}.

In this final part, we provide some speculations on the geometric interpretation of the sum-over-partitions structure in Theorem \ref{thmTngwintermsofB} in term of the stratification of configurations spaces,
motivated by the discussions in \cite{Totaro:1996, Getzler:1999}.

We first introduce some constructions.
For each nonempty labeling set $S$ and
$\pi\in \Pi_{S}$, define 
\[
\mathrm{Conf}_{\pi}(E)=
\Big\{
(p_{a})_{a\in S}~|~p_{a}=p_{b}~\mathrm{if ~and~ only~ if}~\exists\, \pi_{k}\in \pi\,~\mathrm{such~that}~ a,b\in \pi_{k}
\Big\}\,.
\]
Let also $\delta_{S}$ be the small diagonal in $E_{S}=\prod_{a\in S}E_a$:
\[\delta_{S}
=\Big\{(p_{a})_{a\in S}\in \prod_{a\in S}E_a~|~
p_{a}=p_{b}~\text{for}~a,b\in S
\Big\}\,.
\]
Then the closure of $\mathrm{Conf}_{\pi}(E)$ satisfies
$\overline{\mathrm{Conf}_{\pi}(E)}\cong \prod_{\pi_{k}\in \pi}\delta_{\pi_{k}}$.
Furthermore, one has the stratifications
\begin{equation}\label{eqnconfigurationspacestratification}
E_{S}=\bigsqcup_{\pi\in \Pi_{S}}\mathrm{Conf}_{\pi}(E)\,,\quad 
E_{S}=\bigcup_{\pi\in \Pi_{S}}\overline{\mathrm{Conf}_{\pi}(E)}\,.
\end{equation}
Now for $n\geq 1$ and $j\in [n]$, we define the %``gauge fixing" 
morphism
\begin{eqnarray}\label{eqngaugefixingmap}
	f_{j}:
E_{[n]}&\rightarrow&    E_{[n]\setminus \{j\}} \,,\nonumber\\
	(p_1,\cdots, p_{n})&\mapsto& (p_{1},\cdots,p_{j-1}, \check{p}_{j},p_{j+1},\cdots ,p_{n})\,.
\end{eqnarray}
Consider the following rule that assigns multi-variable quasi-elliptic functions to  the small diagonals 
\begin{equation}\label{eqnhyporule}
	\mathfrak{R}:\quad
\delta_{I}\subseteq E_{I}\mapsto 
\partial_{c}^{(|I|-1)}|_{c=0}S^{\mathrm{reg}}(z_{I},c)={\mathbold{B}_{|I|}\over |I|}(z_{I})\,~\text{for}~ |I|\geq 1\,.
\end{equation}
Extend $\mathfrak{R}$ such that it maps the operations $\bigcup, \prod$ 
on small diagonals
to the operations $+, \cdot$ on the ring of quasi-elliptic functions.
Then from   \eqref{eqngaugefixingmap}, the second stratification in \eqref{eqnconfigurationspacestratification},  \eqref{eqnhyporule}
and
\eqref{eqnTngwsumoverpermutationsstratification}, we see that
\begin{equation}
	\theta(z_{[n]})(q)_{\infty} F_{n}=
	\mathfrak{R}\left(\sum_{j\in [n]} f_{j}( E_{[n]}) \right)\,.
\end{equation}
We hope to further investigate this  structure using the language of mixed Hodge structures along the lines in \cite{Kriz:1994, Totaro:1996,  Getzler:1999, Brown:2011}
in a future work, borrowing the inputs from the GW/Hurwitz correspondence in 
\cite{Okounkov:2006}.

\begin{appendices}
\section{Quasi-ellipticity of ordered $A$-cycle integrals}
\label{appendixellipticfunctions}

In this part we shall first review a few basic properties of the Jacobi theta functions, then give the proof of
Theorem \ref{thmquasiellipticitygba}.

\subsection{Basics on elliptic and quasi-elliptic functions}
\label{appendixbasicsonellipticfunctions}

We recall a few facts for Eisenstein series and Jacobi forms, mainly following \cite{Eichler:1985}.
For a positive integer $k$, define the Eisenstein series
\begin{equation}\label{eqndfnGmappendix}
\mathbb{G}_{k}={1\over 2}\cdot (k-1)!\cdot \sum_{\lambda\in \Lambda_{\tau}-\{(0,0)\}}^{~\quad e}{1\over \lambda^{k}}\,.
\end{equation}
When $k$ is odd, it is defined to be zero by convention.
When $k=2$, one uses the Eisenstein summation prescription $\sum^{e}$ (summing over the $\pii$ direction first then $\pii \tau$ direction in $\Lambda_{\tau}$) to deal with the non-absolute convergence issue.

The Jacobi theta function $\theta$ satisfies
\begin{equation}\label{eqnthetalogexpansion}
\theta(z)
=z\exp(\sum_{k\geq 1}-2\mathbb{G}_{2k}{z^{2k}\over (2k)!})
\,.
\end{equation}
In particular, one has the Laurent expansion
\begin{equation}\label{eqnLaurentexpansionthetaprime}
{\theta'(z)\over\theta(z)}={1\over z}+\sum_{k\geq 1}-2\mathbb{G}_{2k}{z^{2k-1}\over (2k-1)!}
={1\over z}+\sum_{\lambda\in\Lambda_{\tau}-\{(0,0)\}}^{\quad ~e}({1\over z+\lambda}-{1\over \lambda})
\,.
\end{equation}
On the other hand, using the Jacobi triple identity for $\theta$, one has the Fourier expansion ($u=e^{z}$)
\begin{eqnarray}\label{eqnthetaFourier}
{\theta'(z)\over\theta(z)}&=&-{1\over 2}-\sum_{k\geq 0}\sum_{n\geq 1} {q^{nk} u^{n}}
+\sum_{k\geq 1}\sum_{n\geq 1} {q^{nk} u^{-n}}\nonumber\\
&=&-{1\over 2}-\sum_{n\geq 1} {1\over 1-q^{n}} {u^{n}}+
\sum_{n\geq 1} {q^{n}\over 1- q^{n}} {u^{-n}}\,\nonumber\\
&=&-{1\over 2}-\sum_{n\neq 0 } {1\over 1-q^{n}} {u^{n}}\,,
\quad
|q|<|u|<1\,.
\end{eqnarray}
It follows that for $m\geq 2$
\[
(\ln\theta)^{(m)}
=-
\left(\sum_{k\geq 0}\sum_{n\geq 1}n^{m-1} {q^{nk} u^{n}}
+\sum_{k\geq 1}\sum_{n\geq 1}(-1)^{m}n^{m-1} {q^{nk} u^{-n}}
\right)
=-\sum_{n\neq 0} {n^{m-1 }u^{n}\over (1-q^{n})}\,,\quad
|q|<|u|<1\,.
\]
By the set partition version of the Fa\`a di Bruno formula, one has
\begin{equation}\label{eqnthetanovertheta}
{\theta^{(n)}\over \theta}=
\mathbold{B}_{n}((\ln \theta)',(\ln \theta)'',\cdots, (\ln \theta)^{(n)})
=\sum_{\ell=0}^{n} \mathbold{B}_{n,\ell}((\ln \theta)',(\ln \theta)'',\cdots, (\ln \theta)^{(n-\ell+1)})\,,
\end{equation}
where $\mathbold{B}_{n,\ell}$ are the Bell polynomials and $\mathbold{B}_{n}=\sum_{\ell=0}^{n}\mathbold{B}_{n,\ell}$
is the complete Bell polynomial.
Each of them is a quasi-elliptic function of pure weight $n$.
% See \cite{Eichler:1985, Libgober:2009} for details.
These quantities can be expressed in terms of the Eisenstein-Kronecker series \cite{Weil:1976} defined by
\begin{equation}\label{eqndfnKroneckerEisenstein}
\mathcal{E}_{m}(z)=\sum^{~\quad e}_{\lambda\in\Lambda_{\tau}} {1\over  ( z+\lambda)^{m}}\,.
\end{equation}
% where $\sum^{e}$ stands for the Eisenstein summation.
It follows that
\begin{equation}\label{eqnEm}
\partial_{z}\mathcal{E}_{m}=-m\mathcal{E}_{m+1}\,,\quad 
\mathcal{E}_{m}={(-1)^{m-1}\over (m-1)!}(\ln\theta)^{(m)}\,,\quad
m\geq 1\,.
\end{equation}

We now collect a few useful properties about the Szeg\"o kernel from \cite{Zagier:1991}
(see also \cite{Brown:2011, Goujard:2016counting}).
\begin{thm}[\cite{Zagier:1991}]\label{thmSzegokernel}
Let
\begin{equation}\label{eqndfnSezgoappendix}
S(z,w)={\theta'(0)\theta(z+w)\over \theta(z)\theta(w)}\,.
\end{equation}
Then it satisfies the following properties.
\begin{enumerate}[i).]
\item
It is the analytic continuation of the theta series over the indefinite lattice $A_{1,1}$:
\[
S(z,w)={u\over u-1}+{1\over v-1}-\sum_{m,n\geq 1} (u^{m}v^{n}-u^{-m}v^{-n})q^{mn}
\,,\quad
|uq|<1\,, |vq^{-1}|>1\,,
\]
where $u=e^{z}, v=e^{w}$.
Alternatively,
\[
-S(z,w)=\sum_{n\in \mathbb{Z}}{u^n\over 1-v q^{n}}=\sum_{n\in \mathbb{Z}}{v^n\over 1-u q^{n}}\,,\quad
|q|<|u|,|v|<1\,.
\]

\item
Its Fourier development satisfies
\[
S(z,w)={1\over 2}{e^{z}+1\over e^{z}-1}+
{1\over 2}{e^{w}+1\over e^{w}-1}-\sum_{n=1}^{+\infty}\sum_{d|n} (e^{ (dz+{n\over d}w)}-e^{- (dz+{n\over d}w)})q^{n}\,.
\]
\item
One has
\[
S(z,w)={1\over w}
\exp
\left(\sum_{m\geq 1}{w^{m}\over m!}
\,
((-1)^{m-1}(m-1)! \,
\mathcal{E}_{m}(z)
+2\mathbb{G}_{m})
\right)\,.
\]

\end{enumerate}

\end{thm}
Theorem \ref{thmSzegokernel} (iii) is particularly useful in obtaining nice Laurent series expansions. It can be derived as follows.
One has
\[
\mathcal{E}_{1}(z+w)
=
\sum_{m\geq 0}
\partial_{z}^{m}\mathcal{E}_{1}(z) {w^{m}\over m!}
=\sum_{m\geq 0}(-1)^{m}m! \,
\mathcal{E}_{m+1}(z) {w^{m}\over m!}\,.
\]
It follows fom \eqref{eqnEm} that
\[
{\theta(z+w)\over \theta(z)}
=\exp
\left(\sum_{m\geq 0}
(-1)^{m}m! \,
\mathcal{E}_{m+1}(z){w^{m+1}\over (m+1)!}
\right)\,.
\]
Combining with \eqref{eqnthetalogexpansion},
one thus obtains Theorem \ref{thmSzegokernel} (iii).
This identity should be compared with the $m$th Laurent coefficients of
 $S(z,w)$ in $w$ given in Theorem \ref{thmSzegokernel} (ii)
 \begin{equation}\label{eqnwLaurentcoefficientofS}
 m!\cdot [w^{m}] S(z,w)
 =-\delta_{m,0}\sum_{k\geq 1}u^{k}+{B_{m+1}\over m+1} -
  \sum_{k\geq 1}\sum_{n\geq 1}n^{m}q^{nk}(u^{k}+(-1)^{m+1}u^{-k})\,,
  \quad m\geq 0\,.
 \end{equation}

 The above Laurent coefficients are quasi-elliptic functions in $z$.
More concretely, let
\begin{equation}\label{eqnEm*}
\mathcal{E}_{m}^{*}(z)=
(-1)^{m-1}(m-1)! \,
\mathcal{E}_{m}(z)
+2\mathbb{G}_{m}=(\ln \theta)^{(m)}+2\mathbb{G}_{m}\,,\quad m\geq 1\,.
\end{equation}
Then from \eqref{eqndfnGmappendix} and \eqref{eqndfnKroneckerEisenstein} one has
\begin{equation}
\mathcal{E}_{m}^{*}:=
(-1)^{m-1}(m-1)!
\left({1\over z^m}+\sum_{\lambda\in \Lambda_{\tau}-\{(0,0)\}}( {1\over (z+\lambda)^{m}}+{(-1)^{m-1}\over (m-1)!}{1\over \lambda^{m}})
\right)
\,,\quad m\geq 1\,.
\end{equation}
Furthermore,  by  Theorem \ref{thmSzegokernel} (iii), the  binomial type relation
satisfied by the complete Bell polynomial  \eqref{eqnbinomialBell}, and \eqref{eqnthetanovertheta},
 one has
 \begin{eqnarray}\label{eqnwLaurentcoefficientofS2}
 m!\cdot [w^{m}] S(z,w)
 &=&{1\over m+1}\mathbold{B}_{m+1}(\mathcal{E}_{1}^{*},\mathcal{E}_{2}^{*},\cdots,
\mathcal{E}_{m+1}^{*})\nonumber\\
&=&{1\over m+1}\sum_{a+2b=m+1}{m+1\choose a}
\mathbold{B}_{a}((\ln\theta)',\cdots ,(\ln\theta)^{(a)})\cdot 2\mathbb{G}_{2b}\nonumber\\
&=&{1\over m+1}\sum_{a+2b=m+1}{m+1\choose a}
{\theta^{(a)}\over \theta}\cdot 2\mathbb{G}_{2b}
\,,
\end{eqnarray}
where the convention $2\mathbb{G}_0=1$ is used.

 \begin{ex}
 Consider the $m=0,1$ terms, one obtains
\begin{equation}\label{eqnFourierofZ}
\mathcal{E}_{1}^{*}={\theta'\over \theta}=
 -({1\over 2}+\sum_{k\neq 0} {u^{k}\over 1-q^{k}})\,,\quad
 {1\over 2}\mathcal{E}_{2}^{*}+ {1\over 2}(\mathcal{E}_{1}^{*})^{2}=
{\theta''\over 2\theta}+\mathbb{G}_{2} ={1\over 12}-\sum_{k\neq 0} {q^{k}u^{k}\over (1-q^{k})^2}\,.
 \end{equation}
 \end{ex}

\subsection{Proof of Theorem \ref{thmquasiellipticitygba}}\label{appendixEulerianquasielliptic}

Based on the results reviewed in Appendix \ref{appendixbasicsonellipticfunctions}, we now 
establish a few  results 
that lead to the proof of Theorem \ref{thmquasiellipticitygba}
which seem to be of independent interest.

\begin{prop}\label{propAm}
Define
\begin{equation}\label{eqnAm}
\mathcal{A}_{m+1}(u)=\sum_{k\neq 0}  {q^{k}A_{m}(q^{k})\over (1-q^{k})^{m+1}}u^{k}\,,
\end{equation}
where $A_{m}$ are the Eulerian polynomials: $A_{0}=A_{1}=1,A_{2}=x+1,\cdots$.
Then for $m\geq 1$, $\mathcal{A}_{m+1}(u)$ is absolutely convergent in the region $ |q|<|u|<1$.
Furthermore,
one has
\begin{equation}\label{eqnAmquasi-elliptic}
\mathcal{A}_{m+1}(u)= {B_{m+1}\over m+1} - m!\cdot [w^{m}] S(z,w) \in \mathbb{C}\Big[\{{\theta^{(k)}\over \theta}\}_{0\leq k\leq m+1}\Big]
\,.
\end{equation}
\end{prop}
 \begin{proof}
By definition, one has
\begin{equation}
\mathcal{A}_{m+1}(u)=
\sum_{k\geq 1}  {q^{k}A_{m}(q^{k})\over (1-q^{k})^{m+1}}u^{k}
+(-1)^{m+1}\sum_{k\geq 1}  {q^{(m+1)k-k}A_{m}(q^{-k})\over (1-q^{k})^{m+1}}u^{-k}\,.
\end{equation}
When $m\geq 1$, one has $\deg A_{m}=m-1$.
A comparison to the geometric series shows its absolute convergence in the region $|q|<|u|<|q^{-1}|$, in particular in the region $|q|<|u|<1$.
Using
the generating series of $A_{m}(x)$ given by \cite{Bona:2012}
\[
\sum_{m\geq 0} {t^{m}\over m!}A_{m}(x)={x-1\over x-e^{(x-1)t}}\,,\]
one sees that
\[
A_{m}({1\over x})=x^{1-m}A_{m}(x)\,,\quad m\geq 1\,.
\]
It follows that
\begin{eqnarray*}
\mathcal{A}_{m}(u)=
\sum_{k\geq 1}  {q^{k}A_{m}(q^{k})\over (1-q^{k})^{m+1}}u^{k}
+(-1)^{m+1}\sum_{k\geq 1}  {q^{k}A_{m}(q^{k})\over (1-q^{k})^{m+1}}u^{-k}
\,,\quad m\geq 1\,.
\end{eqnarray*}
Applying the classical identity
\[
\sum_{\ell\geq 1}\ell^{m}x^{\ell}
  ={xA_{m}(x)\over (1-x)^{m+1}}\,,
  \]
one obtains
\begin{eqnarray*}
\mathcal{A}_{m+1}(u)=
\sum_{k\geq 1}\sum_{\ell\geq 1}  {\ell^{m}q^{k\ell}}u^{k}
+(-1)^{m+1}
\sum_{k\geq 1}\sum_{\ell\geq 1}   {\ell^{m}q^{k\ell}}
 u^{-k}
\,,\quad m\geq 1\,.
\end{eqnarray*}
Using
 the relations \eqref{eqnwLaurentcoefficientofS}, \eqref{eqnwLaurentcoefficientofS2} in Appendix \ref{appendixbasicsonellipticfunctions}, one has
\[
 m!\cdot [w^{m}] S(z,w)
 =-\delta_{m,0}\sum_{k\geq 1}u^{k}+ {B_{m+1}\over m+1} -
  \sum_{k\geq 1}\sum_{\ell\geq 1}\ell^{m}q^{k\ell}(u^{k}+(-1)^{m+1}u^{-k})
  \in \mathbb{C}\Big[\{{\theta^{(k)}\over \theta}\}_{0\leq k\leq m+1}\Big]\,.
\]
Combining these, we arrive at \eqref{eqnAmquasi-elliptic}.

\end{proof}

\begin{ex}\label{exanalyticalcontinutationA0}
When $m=0$, one has
\[
\mathcal{A}_{1}(u)=\sum_{k\neq 0}{q^{k}u^{k}\over 1-q^{k}}=\sum_{k\geq 1}{q^{k}u^{k}\over 1-q^{k}}
-\sum_{k\geq 1}{u^{-k}\over 1-q^{k}}
\]
which converges in the region $1<|u|<|q|^{-1}$.
By
 \eqref{eqnthetaFourier}, one has
\[
\mathcal{A}_{1}(u)=-{1\over 2}-{\theta'\over \theta}(qu)\,,\quad 1<|u|<|q|^{-1}\,.
\]
The right hand side can be analytically continued  from the region $1<|u|<|q|^{-1}$ to the
region $|q|<|u|<|q|^{-1}, u\neq 1$ using its quasi-ellipticity.
Indeed, one has from the automorphy \eqref{eqnautomorphyfactoroftheta} of $\theta$
\begin{equation}\label{eqnthetaderivativeautomorphy}
{\theta'\over \theta}(u)={\theta'\over \theta}(qu)+1
\,,\quad |q|<|u|<1\,,
\end{equation}
and thus the analytic continuation $\widetilde{\mathcal{A}}_{1}(u)$ of $\mathcal{A}_{1}(u)$ satisfies
\begin{equation}\label{eqnA1uspecialcase}
\widetilde{\mathcal{A}}_{1}(u)={1\over 2}-{\theta'\over \theta}(u)=1+\sum_{k\neq 0}{u^{k}\over 1-q^{k}}\,,\quad |q|<|u|<1\,.
\end{equation}
Note that here one can also use the parity instead of quasi-ellipticity
\[
\widetilde{\mathcal{A}}_{1}(u)=-{1\over 2} + {\theta'\over \theta}(u^{-1})\,,\quad |q|<|u|<1
\]
to obtain the same result.
\xxqed
\end{ex}

\begin{proof}[Proof of Theorem \ref{thmquasiellipticitygba}]
	Recall from Definition \ref{eqngbafunction} that one has
	\[g_{b,a}(u)= (-1)^{a}
	({1\over 2})^{b}+\sum_{k\geq 1} {q^{ka}u^{k}\over (1-q^{k})^{b}}
	+\sum_{k\geq 1}(-1)^{b} {q^{k(b-a)}u^{-k}\over (1-q^{k})^{b}}\,.
	\]
Consider the case with $b=1$, by
\eqref{eqnthetaFourier} one has
\[
g_{1,0}(u)=-{\theta'\over \theta}(u)-{1\over 2}\,,\quad |q|<|u|<1\,.
%g_{1,1}(u)&=&\mathcal{A}_{0}(u) =-{1\over 2}-{\theta'\over \theta}(qu)\,,\quad 1<|u|<|q^{-1}|\,.
\]
One has $g_{1,1}(u)=\mathcal{A}_{1}(u)$ which has been addressed in 
Example \ref{exanalyticalcontinutationA0}.
Hence we only need to prove the statement for $b\geq 2$.

\begin{enumerate}[i).]
	\item 
	This follows from the comparison 
to the geometric series. In particular, each series $g_{b,a}$ is absolutely convergent
 in the region $|q|<|u|<1$.

\item
Ignoring the constant term in $g_{b,a}$
does not change the statement. 
For convenience we do that and focus on 
 the rest part denoted by $f_{b,a}$.
Recall Frobenius's formula \cite{Bona:2012}
\[
A_{n}(x)=\sum_{\ell=1}^{n} \ell ! \, S(n,\ell) (x-1)^{n-\ell}\,, \quad n\geq 1\,.
\]
We see that
\[
{q^{k}A_{n}(q^{k})\over (1-q^{k})^{n+1}}=
\sum_{\ell=1}^{n}  S(n,\ell) (-1)^{n-\ell}
{q^{k}\over (1-q^{k})^{\ell+1}}\,,\quad n\geq 1\,.
\]
By the Stirling inversion formula, we then have
\[(-1)^{n}
{q^{k}\over (1-q^{k})^{n+1}}
 =\sum_{\ell=1}^{n}s(n,\ell) (-1)^{\ell}{q^{k}A_{\ell}(q^{k})\over (1-q^{k})^{\ell+1}}\,,\quad n\geq 1\,.
\]
Therefore, one has
\begin{equation}\label{eqnfn1}
f_{n,1}(u)=
\sum_{k\neq 0}
{q^{k}\over (1-q^{k})^{n+1}}u^{k}
 =(-1)^{n}\sum_{\ell=1}^{n}s(n,\ell) (-1)^{\ell} \mathcal{A}_{\ell+1}\,,\quad n\geq 1\,.
\end{equation}
By Proposition \ref{propAm}, we see for any $n\geq 2$
\begin{equation}\label{eqnfn1quasielliptic}
f_{n,1}(u)\in \mathbb{C}\Big[\{{\theta^{(k)}\over \theta}(u)\}_{0\leq k\leq n+1}\Big]\,,\quad |q|<|u|<1\,.
\end{equation}
Using quasi-ellipticity, we obtain the analytic continuation $\widetilde{f}_{n,1}$ of $f_{n,1}$
in any annulus $|q^{m+1}|<|u|<|q^{m}|,m\in\mathbb{Z}$.

Now consider $f_{b,a}$ with $0\leq a<b$. If  $u$ satisfies $|q|<|u|<1$, then
  $q^{a-1}u$ satisfies
 \[
|q^{a}|<|q^{a-1}u|<|q^{a-1}|\,.
 \]
Using \eqref{eqnfn1quasielliptic} and \eqref{eqnthetaderivativeautomorphy}, we obtain
\[
f_{b,a}(u)=\widetilde{f}_{b,1}(q^{a-1}u)\in  \mathbb{C}\Big[\{{\theta^{(k)}\over \theta}(q^{a-1}u)\}_{0\leq k\leq b+1}\Big]
= \mathbb{C}\Big[\{{\theta^{(k)}\over \theta}(u)\}_{0\leq k\leq b+1}\Big]\,,\quad |q|<|u|<1\,.
\]
By \eqref{eqnthetanovertheta} and  \eqref{eqnthetaderivativeautomorphy}, we see that
 ${\theta^{(k)}\over \theta}(u)$ is a quasi-elliptic function of pure weight $k$
and  in general ${\theta^{(k)}\over \theta}(q^{a-1}u)$ is quasi-elliptic of mixed weight with leading weight $k$.
Furthermore, it
shares the same leading term with ${\theta^{(k)}\over \theta}(u)$.

Consider the last case with $2\leq a=b$
in which the function $f_{b,a}$ is now absolutely convergent in the region $1<|u|<|q|^{-1}$.
The same reasoning in Example \ref{exanalyticalcontinutationA0} applies.
This finishes the proof.

\end{enumerate}

 \end{proof}

\section{Evaluation of $n$-point functions}
\label{appendixevaluations}

In this part we supply some computations on $n$-point functions both in 2d chiral conformal field 
theories and in GW theory.
Many steps rely heavily on non-trivial identities among quasi-elliptic functions as in Example \ref{exTn3case}.

\subsection{2-point functions for free boson and free fermions}
\label{appendixorderings}

We compute $2$-point correlation functions of operators in
radial/time ordering and in
normal ordering, in terms of quasi-elliptic functions.
We  follow closely the exposition in \cite{Kawamoto:1988, DJM00}.

\subsubsection*{Free chiral boson}
Consider the Heisenberg algebra generated by $a_{n},n\in \mathbb{Z}$, with
\[
[a_{m},a_{n}]:=a_{m}a_{n}-a_{n}a_{m}=m\delta_{m+n,0}\,.
\]
Introduce also an operators $x_0$ such that
$[ a_0,x_0]=\mathbf{1}$.
The chiral boson field is defined as the following generating series 
\[
\phi(u)=x_0+a_0 \ln u+\sum_{n\neq 0}a_{n}{u^{{-n}}\over -n}\,.
\]
The current field is defined as
\[
J(u)=u\partial_{u}\phi=a_0 +\sum_{n\neq 0}{a_{n}}u^{-n}\,.
\]

Introduce the normal product of operators 
\[
:a_{m}a_{n}:=\begin{cases}
	a_{m}a_{n}\,,\quad n\geq m\\
	a_{n}a_{m}\,,\quad n<m\,.
\end{cases}
\]
The Hamiltonian $H$ is then defined as the normal product
\[
H={1\over 2}\sum_{k\neq 0}: a_{-k}a_{k}:
={1\over 2}a_0^2+{1\over 2}\sum_{k\geq 1} a_{-k}a_{k}+{1\over 2}\sum_{k\leq -1} a_{k}a_{-k}\,.
\]
The operator $H-{1\over 24}$ can be regarded as the zeta-regularization of the naively defined product
\[
{1\over 2}\sum_{k\neq 0} a_{-k}a_{k}
={1\over 2}\sum_{k\neq 0} :a_{-k}a_{k}:
+{1\over 2}\sum_{k\leq -1} [ a_{-k}, a_{k}]
={1\over 2}\sum_{k\neq 0} :a_{-k}a_{k}:
+{1\over 2}\sum_{k\leq -1}  -k
={1\over 2}\sum_{k\neq 0} :a_{-k}a_{k}:
+{1\over 2}\zeta(-1)\,.
\]
That is, formally  ${1\over 2}\zeta(-1)$ is the vacuum energy
\[
{1\over 2}\zeta(-1)=^{\mathrm{formally}}{1\over 2}\lim_{{v\over u}\rightarrow 1}\sum_{k\geq 1} k({v\over u})^{k}
=^{\mathrm{formally}}{1\over 2} \lim_{{v\over u}\rightarrow 1}{{v\over u}\over (1-{v\over u})^2}\,.
\]
More geometrically,
the extra term $-{1\over 24}$ in $H-{1\over 24}$ above arises
from the change of the projective connection arising from the map $z\mapsto u=e^{z}$.
The Hamiltonian $H$ satisfies
\begin{equation}\label{eqnHancommutator}
[a_{n},H]=n a_{n}\,.
\end{equation}

The $2$-point correlators
are defined as the following  trace over the usual bosonic Fock space
\[
\langle  a_{\ell}a_{k}\rangle:=\mathrm{Tr}\,(q^{H-{1\over24}}a_{\ell}a_{k})\,.
\]
They can be determined using \eqref{eqnHancommutator} as follows.
One has
\[
\mathrm{Tr}(q^{H-{1\over24}}a_{\ell}a_{k})
=
\mathrm{Tr}(a_{k} q^{H-{1\over24}} a_{\ell})=
\mathrm{Tr}( q^{H-{1\over24}+k} a_{k}a_{\ell})
\,.
\]
On the other hand,
\[
\mathrm{Tr}(q^{H-{1\over24}}a_{\ell}a_{k})-\mathrm{Tr}(q^{H-{1\over24}}a_{k}a_{\ell})
=
\mathrm{Tr}(q^{H-{1\over24}}\cdot \ell \delta_{\ell+k})
\,.
\]
Combining these two, we obtain
\begin{equation}\label{eqn2pointcorrelators}
\mathrm{Tr}(q^{H-{1\over24}}a_{\ell}a_{k})=
\begin{cases}
{\ell \over 1-q^{-k}}\delta_{k+\ell}\mathrm{Tr}(q^{H-{1\over24}})\quad \ell\neq 0\,,\\
0\,,\quad \ell=0\,,
\end{cases}
\end{equation}
where 
the partition function $\mathrm{Tr}(q^{H-{1\over24}})$ can be directly computed to be
\[\langle \mathbold{1}\rangle=
\mathrm{Tr}(q^{H-{1\over24}})=\left(q^{1\over 24}\prod_{n=1}^{\infty}(1-q^{n})\right)^{-1}\,.
\]

\begin{rem}\label{remoriginofanomaly}
The spectrum of $a_0$, called momentum, is $\mathbb{R}$.
The contribution of this spectrum to the correlation function/character  is
\[
\int_{\mathbb{R}}e^{{1\over 2}\pii\tau p^2}e^{{1\over 2}\overline{\pii\tau} p^2}e^{p z}e^{p\bar{z}} dp\,.
\]
This gives the extra factor
$(\mathrm{im}\, \tau)^{-{1\over 2}}$ to $\langle \mathbold{1}\rangle$ and
$e^{-2\pi  {(\mathrm{im}\, z)^2 \over \mathrm{im}\,\tau}}$
to ${\theta}$. This contribution of spectrum is the origin of modular and elliptic anomaly in $n$-point functions.

\end{rem}

It follows from \eqref{eqn2pointcorrelators} that the $2$-point function of the radially/time ordered product\footnote{One has $u=e^{\pii(\sigma-\mathrm{i}t)}$, where $t,\sigma$
are the coordinates of the time and spatial directions of the spacetime. One can also use $\tilde{u}=e^{\pii(\sigma+\mathrm{i}t)}$ as the coordinate, while changing
$a_{n}u^{-n}$ to $a_{n}\tilde{u}^{-n}$ so that the time ordering is given by $|\tilde{u}|<|\tilde{v}|$. } $J(u)J(v), |u|>|v|$ is given by
\begin{equation}\label{eqn2pointfunctiontimeordered}
\langle  J(u)J(v)\rangle=\sum_{k,\ell}
 u^{-\ell}v^{-k} \langle  a_{\ell}a_{k}\rangle
 =\sum_{k\neq 0} {kq^{k} \over 1-q^{k}} ({v\over u})^{-k}= \wp({v\over u})+2\mathbb{G}_2\,.
\end{equation}
As a comparison, the $2$-point function of the normally ordered product is 
\begin{equation}\label{eqn2pointfunctionnormalordered}
\langle  :J(u)J(v):\rangle=\sum_{k,\ell}u^{-\ell}v^{-k}
\langle  :a_{\ell}a_{k}:\rangle=
\sum_{k\geq 0}
{kq^{k}\over 1-q^{k}} ({v\over u})^{-k}
+
\sum_{k< 0}{-kq^{-k}\over 1-q^{-k}} ({v\over u})^{-k}
\,.
\end{equation}
The difference is then the vacuum expectation value
\begin{equation}\label{eqn2pointfunctiontime-normalordered}
\langle  J(u)J(v)\rangle-
\langle  :J(u)J(v):\rangle=-\sum_{k<0} k ({v\over u})^{-k}=\sum_{k\geq 1} k({v\over u})^{k}
={{v\over u}\over (1-{v\over u})^2}\,.
\end{equation}
In the region $|q|<|{v\over u}|<1$,
the singularity of $\wp({v\over u})+2\mathbb{G}_2$ at ${v\over u}=1$ is exhibited by (recall \eqref{eqnthetaFourier})
\begin{eqnarray*}
\wp+2\mathbb{G}_2&=&\sum_{k\geq 1} {k({v\over u})^{k}\over 1-q^{k}}-\sum_{k\geq 1} {k({v\over u})^{-k}\over 1-q^{-k}}\\
&=&\sum_{k\geq 1} k({v\over u})^{k}+\sum_{k\geq 1} {k({v\over u})^{k}\over 1-q^{k}}+\sum_{k\geq 1}
{kq^{k}({v\over u})^{-k}\over 1-q^{k}}\,,
\end{eqnarray*}
where the sum of the last two terms are regular near ${v\over u}=1$.
Therefore, the difference \eqref{eqn2pointfunctiontime-normalordered} is given by the singular part of the correlation function \eqref{eqn2pointfunctiontimeordered}  of the radial ordering, rendering the
normal ordering one \eqref{eqn2pointfunctionnormalordered} regular but not elliptic anymore.
In terms of the additive coordinate $z$, one has
\[
{{v\over u}\over (1-{v\over u})^2}=-1+{1\over 1-e^{w-z}}=-1-\sum_{n\geq	 0} {B_{n}\over n!} (w-z)^{n-1}\,.
\]

\subsubsection*{Free chiral fermions}

One can apply the same analysis to the fermions $\psi(u), \psi^{*}(v)$
whose modes gives the Clifford algebra.
In this case,
the correlation function of the radial ordering product $\psi(u) \psi^{*}(v), |u|>|v|$  gives the Szeg\"o kernel $S_{c}(u,v)$.

The details are as follows.
The modes satisfy
\[
\{\psi_{r},\psi^{*}_{s}\}=\psi_{r}\psi^{*}_{s}+\psi^{*}_{s}\psi_{r}=\delta_{r+s}\,.
\]
The fermion fields are
\[
\psi(u)=\sum_{r} \psi_{r}u^{-r}\,,\quad
\psi^{*}(v)=\sum_{s} \psi^{*}_{s}v^{-s}\,.
\]
The shifted Hamiltonian
\[
H-{1\over 24}=\sum_{r} r :\psi_{-r}\psi^{*}_r:-{1\over 24}
\]
is the zeta-regularized version of
\[
\sum_{r} r \psi_{-r}\psi^{*}_r=\sum_{r} r :\psi_{-r}\psi^{*}_r:+\sum_{r<0} r=\sum_{r} r :\psi_{-r}\psi^{*}_r:+{1\over 2}\zeta(-1)\,,
\]
when $r,s\in \mathbb{Z}+{1\over 2}$.
One has
\begin{equation}\label{eqncommutatorHpsifermion}
[H,\psi_{r}]=-r\psi_{r}\,,\quad  [H,\psi^{*}_{r}]=-r\psi_{r}^{*}\,.
\end{equation}
By definition,
\[
\psi(u)\psi^*(v)=\sum_{r,s}\psi_{r}\psi^{*}_{s}u^{-r}v^{-s}\,,\quad
:\psi(u)\psi^*(v):=\sum_{r,s:\, s>0}\psi_{r}\psi^{*}_{s}u^{-r}v^{-s}
-\sum_{r,s:\, s<0}\psi^{*}_{s}\psi_{r}u^{-r}v^{-s}\,.
\]
It follows that
\begin{eqnarray*}
	:\psi^{*}(v)\psi(u):&=&-:\psi(u)\psi^{*}(v):\\
	\psi(u)\psi^{*}(v)&=&:\psi(u)\psi^{*}(v):+\sum_{s<0} ({v\over u})^{-s}\mathbf{1}\,,\quad |{v\over u}|<1\,,\\
	\psi^{*}(v)\psi(u)&=&:\psi^{*}(v)\psi(u):+\sum_{s>0} ({v\over u})^{-s}\mathbf{1}\,,\quad |{u\over v}|<1\,.
\end{eqnarray*}
The differences are formal distributions\footnote{In \cite[Section 3.3.2]{Okounkov:2001} and \eqref{dfngbafunction}, it is assumed that
	after analytic continuation $\delta_{u,v}$ is set to zero as an operator.
	This is also the case in the Fourier expansion of the functional equation $\theta(u^{-1})=-\theta(u)$.
	However, in computations such as in deriving \eqref{eqncommutationTpsi} of 
	Proposition \ref{proprelationbetweenoperatorsTnTnA}, when multiplied with other
	formal distributions this distribution is not treated as zero.}
\begin{eqnarray*}
\psi(u)\psi^*(v)-:\psi(u)\psi^*(v):&=&\sum_{r,s:\, s<0}\delta_{r+s}u^{-r}v^{-s}
=\sum_{s<0}({v\over u})^{-s}\mathbf{1}\,,\\
	\psi(u)\psi^{*}(v)-	\psi^{*}(v)\psi(u)&=&
\delta_{u,v}:=\sum_{s<0} ({v\over u})^{-s}\mathbf{1}+\sum_{s>0} ({v\over u})^{-s}\mathbf{1}\,.
\end{eqnarray*}

Similar to the reasoning in the  free chiral boson case, using \eqref{eqncommutatorHpsifermion} one has
\begin{equation}\label{eqn2pointcorrelatorsfermion}
\langle \psi_{r}\psi_{s}^{*}\rangle={q^{s}\over 1+q^{s}}\delta_{r+s}\,,\quad
\langle \psi_{s}^{*} \psi_{r}\rangle={q^{r}\over 1+q^{r}}\delta_{r+s}\,,
\end{equation}
up to normalization by the partition function $\langle \mathbold{1}\rangle$.
It follows that
\begin{eqnarray}\label{eqn2pointcorrelationfunctionfermion}
\langle \psi(u)\psi^{*}(v)\rangle
&=& \sum_{r,s} {q^{s}\over 1+q^{s}} \delta_{r+s} u^{-r}v^{-s}
=\sum_{s} {q^{s}\over 1+q^{s}} ({v\over u})^{-s}
\,,\\
\langle :\psi(u)\psi^{*}(v):\rangle
&=&\sum_{s>0}{q^{s}\over 1+q^{s}} \delta_{r+s} u^{-r}v^{-s}
-\sum_{s<0} {q^{-s}\over 1+q^{-s}} \delta_{r+s} u^{-r}v^{-s}\nonumber\\
&=&-\sum_{s>0}{q^{s}\over 1+q^{s}} ({v\over u})^{s}
+\sum_{s>0} {q^{s}\over 1+q^{s}} ({v\over u})^{-s}\,.
\end{eqnarray}
The difference is then the ``singular part" of $\langle \psi(u)\psi^{*}(v)\rangle$, which is also the vacuum expectation value,
\[
\langle \psi(u)\psi^{*}(v)\rangle-
\langle :\psi(u)\psi^{*}(v):\rangle
=\sum_{s>0}  ({v\over u})^{s}\,,
\]
leaving the correlation functions $\langle :\psi(u)\psi^{*}(v):\ \rangle=-\langle :\psi^{*}(v)\psi(u): \rangle$ nonsingular everywhere.
It is  the radially ordered product in which Wick's theorem is conveniently formulated 
\[C_{2n}=
\langle
\psi(u_1)\cdots \psi(u_{n})\psi^{*}(v_{n})\cdots \psi^{*}(v_{1})\rangle=\det( \langle \psi(u_i) \psi^{*}(v_j)\rangle)\,.\]

Let $s=n+\delta, \delta\in (0,1)$.
Then from \eqref{eqn2pointcorrelationfunctionfermion} the Fourier series for $\langle \psi(u)\psi^{*}(v)\rangle$ is convergent in the region $|q|<|{v\over u}|<1$
\[\langle \psi(u)\psi^{*}(v)\rangle=
\sum_{s} {q^{s}\over 1+q^{s}} ({v\over u})^{-s}
=\sum_{s> 0} {1\over 1+q^{s}} ({v\over u})^{s}+\sum_{s>0} {q^{s}\over 1+q^{s}} ({v\over u})^{-s}\,.
\]
By Theorem \ref{thmSzegokernel} (i), this gives the Szeg\"o kernel
\begin{equation}
\langle \psi(u)\psi^{*}(v)\rangle=
({v\over u})^{1-\delta}
\sum_{n} {({v\over u})^{n}\over 1-q^{n}e^{c}}=
- ({v\over u})^{1-\delta} S_{c}(w-z)\,,
\end{equation}
 where
\[
e^{c}=-q^{1-\delta}=e^{\pii ({1\over 2}+(1-\delta) \tau)}\,,\quad c\neq 0\,.
\]
When $\delta={1\over 2}$, this simplifies into
\[
\langle \psi(u)\psi^{*}(v)\rangle
=- ({v\over u})^{1\over 2}
{\theta_{c}({v\over u})\over \theta_{c}(0)}{\theta'(0)\over \theta({v\over u})}
=- ({v\over u})^{1\over 2}
{\vartheta_{({1\over 2}, {1\over 2})}(c+w-z)\over \vartheta_{({1\over 2},
{1\over 2})}(c)}{\theta'(0)\over \theta({v\over u})}
=-{\vartheta_{(0,0)}({v\over u})\over \vartheta_{(0,0)}(0)}
{\theta'(0)\over \theta({v\over u})}
\,.
\]
In particular, using the parity of $\vartheta_{(0,0)}, \vartheta_{({1\over 2},{1\over 2})}$, one has the desired relation (see \cite[Equation 3.9]{Okounkov:2001})
\begin{equation}
	[v^0]
\langle \psi(uv)\psi^{*}(v)\rangle=-{\vartheta_{(0,0)}({1\over u})\over \vartheta_{(0,0)}(0)}
{\theta'(0)\over \theta({1\over u})}
={\vartheta_{(0,0)}(u)\over \vartheta_{(0,0)}(0)}
{\theta'(0)\over \theta(u)}={\vartheta_{(0,0)}(u)\over \theta(u)}\,.
\end{equation}

\subsection{Ordered $A$-cycle integral and regularized integral for the $n=3$ case}
\label{secregularizedintegralT3}

In this part, we evaluate the ordered $A$-cycle integrals and regularized integral of
\[
\theta(\prod_{k}u_{k}) \Theta_{2n}={\theta(\prod_{k}u_{k}) \over \prod_{k} \theta(u_{k})}\Phi\,,\quad
\Phi:= {\prod_{i<j} \theta({u_{i}v_{i}\over u_{j}v_{j}} \theta({v_{i}\over v_{j}})\over   \prod_{i<j}
\theta({u_{i}v_{i}\over v_{j}} \theta({v_{i}\over u_{j}v_{j}}  )}
\]
 for the $n=3$ case, using the techniques developed in \cite{Li:2020regularized, Li:2022regularized}.
Again we use the  additive coordinate $z$ and multiplicative coordinate $u=e^{z}$
interchangably.

Before proceeding, the following residues, which can be checked directly, will be used repeatedly
\begin{eqnarray}\label{eqnresiduesofblock}
		\mathrm{Res}_{v_{i}=u_{i}^{-1}v_{j}}\,
		\left(
		{\theta({u_{i}v_{i}\over u_{j}v_{j}})\theta({v_{i}\over v_{j}})\over
			\theta({u_{i}v_{i}\over v_{j}})
			\theta({v_{i}\over u_{j}v_{j}})}	
\right)
&=&-{1\over S_{ij}}\nonumber\\
\mathrm{Res}_{v_{i}=u_{j}v_{j}}\,
\left(
{\theta({u_{i}v_{i}\over u_{j}v_{j}})\theta({v_{i}\over v_{j}})\over
	\theta({u_{i}v_{i}\over v_{j}})
	\theta({v_{i}\over u_{j}v_{j}})}
\right) &=&{1\over S_{ij}}
\,,\quad \quad 
S_{ij}:=
\left({\theta'(1) \theta(u_i u_j)
	\over \theta(u_i)\theta(u_j)}\right)^{-1}\,.
\end{eqnarray}
Define the quasi-elliptic function
\begin{equation}\label{eqndfnZappendix}
	Z(z)=
{\theta'(z)\over \theta(z)}
=\zeta({z})-2\mathbb{G}_2 z\,.
\end{equation}
Its elliptic and modular completion is $\widehat{Z}(z)=Z(z)+\mathbold{A}$, where $\mathbold{A}(z)={\mathrm{im}\,z\over \mathrm{im}\,\tau}$.
%Both of them are odd functions in $z$.

Denote by $R^{i}_{j}$ the summation of residues at
$w_{i}\in D_{ij}$, with $D_{ij}$ defined in \eqref{eqndivisorDelta}.
Following the residue formulas shown in \cite{Li:2022regularized}, we have
\[
\int_{E_{[n]}}\Phi=
R^{2}_{3}R^{1}_{2}  \Phi\left( \widehat{Z}({v_{1}\over v_{3}})\widehat{Z}({v_{2}\over v_{3}})
-{1\over 2}\widehat{Z}({v_{2}\over v_{3}})^2\right)
+R^{2}_{3}R^{1}_{3}  \Phi\left( \widehat{Z}({v_{1}\over v_{3}})\widehat{Z}({v_{2}\over v_{3}})\right)\,.
\]
and
\[
\int_{A_{[n]}}\Phi=
R^{2}_{3}R^{1}_{2} \Phi \left( Z({v_{1}\over v_{3}})Z({v_{2}\over v_{3}})
-{1\over 2}Z({v_{2}\over v_{3}})^2-{1\over 2}Z({v_{2}\over v_{3}})\right)
+R^{2}_{3}R^{1}_{3}  \Phi\left( Z({v_{1}\over v_{3}})Z({v_{2}\over v_{3}})\right)\,.
\]
Straightforward computations using  \eqref{eqnresiduesofblock},  \eqref{eqnLaurentexpansionthetaprime} and \eqref{eqndfnZappendix} give

\begin{eqnarray*}
R^{2}_{3}R^{1}_{2}  \Phi\left( \widehat{Z}({v_{1}\over v_{3}})\widehat{Z}({v_{2}\over v_{3}})\right)
&=&{\prod \theta(u_i)\over \theta(\prod u_{i})}
(-\widehat{Z}({u_3 })\widehat{Z}({u_1 u_3 })+\widehat{Z}({u_2 })\widehat{Z}({u_1 u_2 })\\
&&\quad\quad\quad \quad-\widehat{Z}({u_1 })\widehat{Z}({u_1 u_2 })+\widehat{Z}({u_3 })\widehat{Z}({u_2 u_3 }))\\
&&-{1\over S_{12}}
  {\theta({1\over u_1 })\theta({ u_2\over   u_{3}})\over
\theta({1\over u_{1}u_{3}})\theta(u_2)} \widehat{Z}({1\over u_1})+{1\over S_{12}}
 {\theta({u_{1}\over u_{3}})\theta({1\over u_{2}})\over
\theta(u_{1})
\theta({1\over u_{2}u_{3}})}\widehat{Z}({1\over  u_2 })\,,\\
R^{2}_{3}R^{1}_{2} \Phi\left(-{1\over 2}\widehat{Z}({v_{2}\over v_{3}})^2\right)
&=& {\prod \theta(u_i)\over \theta(\prod u_{i})}
({1\over 2} \widehat{Z}(u_1 u_3)^2-{1\over 2} \widehat{Z}({1\over u_2})^2+{1\over 2} \widehat{Z}({1\over u_1 u_2})^2-{1\over 2} \widehat{Z}( u_3)^2)\\
&&-{1\over S_{12}} (-{1\over 2} )
  {\theta({1\over u_1 })\theta({ u_2\over   u_{3}})\over
\theta({1\over u_{1}u_{3}})\theta(u_2)} (Z({1\over u_1})+Z({ u_2\over   u_{3}})-Z({1\over u_{1}u_{3}})-Z({u_2}))\\
&&+{1\over S_{12}} (-{1\over 2})
 {\theta({u_{1}u_2\over u_{3}}) \over
\theta(u_{1}u_2)
\theta({1\over u_{3}})}\,,\\
R^{2}_{3}R^{1}_{2} \Phi\left(-{1\over 2}\widehat{Z}({v_{2}\over v_{3}})\right)
&=& {\prod \theta(u_i)\over \theta(\prod u_{i})}
({1\over 2}Z(u_1 u_3)-{1\over 2} Z({1\over u_2})+{1\over 2}Z({1\over u_1 u_2})-{1\over 2} Z( u_3))\\
&&-{1\over S_{12}} (-{1\over 2} )
  {\theta({1\over u_1 })\theta({ u_2\over   u_{3}})\over
\theta({1\over u_{1}u_{3}})\theta(u_2)}\,,\\
R^{2}_{3}R^{1}_{3}  \Phi\left( \widehat{Z}({v_{1}\over v_{3}})\widehat{Z}({v_{2}\over v_{3}})\right)
&=&{\prod \theta(u_i)\over \theta(\prod u_{i})}
(\widehat{Z}({u_1 })\widehat{Z}({u_1 u_2 })+\widehat{Z}({u_1 })\widehat{Z}(u_3)
+\widehat{Z}({u_3 })\widehat{Z}({u_1 u_3 })+\widehat{Z}({u_2 })\widehat{Z}({ u_3 }))\\
&&-{1\over S_{13}}
  {\theta({ u_1 })\theta({ u_2\over   u_{3}})\over
\theta({u_{1}u_{2}})\theta({1\over u_3})} \widehat{Z}({1\over u_1})\,.
\end{eqnarray*}
The corresponding terms appearing in the ordered $A$-cycle integrals are obtained by the same expressions above, with $\widehat{Z}$ replaced by $Z$.
Note that  in each integration, there is an ambiguity in choosing the integration constants that are elliptic functions in the remaining variables.
These integration constants can change the individual terms above, but result in no difference for the overall sum by the global residue theorem.

Computing the $c^{1}$ of the Fay's identity  \eqref{eqnFaymulti-secantsimplified} for $n=2$ gives
\begin{eqnarray*}
Z({1\over u_1})+Z({ u_2\over   u_{3}})-Z({1\over u_{1}u_{3}})-Z({u_2})
&=& \theta({u_{2}\over u_1 u_3}) {1\over \theta({1\over u_1})\theta({u_{2}\over u_{3}})}
{\theta(u_3) \theta(u_1 u_2)\over \theta({ u_2})\theta({u_1 u_{3}})}\,,\\
{1\over 2}Z(u_1 u_3)-{1\over 2} Z({1\over u_2})+{1\over 2}Z({1\over u_1 u_2})-{1\over 2} Z( u_3)
&=&{1\over 2}{\theta(u_1 u_2 u_3) \theta({1\over u_1}) \theta({u_3\over u_2}) \over \theta(u_2) \theta(u_{3})  \theta(u_1 u_3) \theta({1\over u_1 u_2})}\,.
\end{eqnarray*}
Computing the $Q_{i}\rightarrow 0$ limit of the Fay's identity  \eqref{eqnFaymulti-secantsimplified} for $n=2$ gives
\[
\wp(z_1)-\wp(z_2)= -{\theta(z_1+z_2)\theta(z_1-z_2)\over \theta(z_1)^2\theta(z_2)^2}
\]
and thus
\begin{eqnarray*}
- {1\over 2}
{\theta(u_1)\theta(u_3) \theta({u_{2}\over u_1 u_3}) \over \theta({ u_2})\theta({u_1 u_{3}})^2}+ {1\over 2}
{\theta(u_1)\theta(u_2) \theta({u_{1}u_2\over  u_3}) \over \theta({ u_3})\theta({u_1 u_{2}})^2}
= {\prod \theta(u_i)\over \theta(\prod u_{i})}
{1\over 2}\left(
\wp(u_2)-\wp(u_1 u_3)+\wp(u_3)-\wp(u_1 u_2)
\right)\,.
\end{eqnarray*}

Simplifying using these identities, we have
$R^{2}_{3}R^{1}_{2} (\Phi (-{1\over 2}Z({v_{2}\over v_{3}}))=0$ as it should be the case since $\int_{A_{[n]}}\Phi$ is of pure weight $0$ while the term above
is the only term with
negative weight.
The other terms are
\begin{eqnarray}
&&R^{2}_{3}R^{1}_{2} (\Phi Z_1 Z_2)\nonumber\\
&=& {\prod \theta(u_i)\over \theta(\prod u_{i})}
(-Z({u_3 })Z({u_1 u_3 })+Z({u_2 })Z({u_1 u_2 })-Z({u_1 })Z({u_1 u_2 })+Z({u_3 })Z({u_2 u_3 }))\nonumber\\
&&+ {\prod \theta(u_i)\over \theta(\prod u_{i})} Z_{1}
\left(
Z_{21}-Z_{2}-Z_{13}+Z_{3}\right)
-{\prod \theta(u_i)\over \theta(\prod u_{i})} Z_{2}
\left(
Z_{12}-Z_{1}-Z_{23}+Z_{3}\right)\nonumber\\
&=& {\prod \theta(u_i)\over \theta(\prod u_{i})}
\left( -Z_{3}Z_{13}+Z_{3}Z_{23}-Z_{1}Z_{13}+Z_{1}Z_{3}+Z_{2}Z_{23}-Z_2 Z_3
\right)\,,\\
%This is anti-symmetric in $1,2$ as it should be the case.
&&R^{2}_{3}R^{1}_{2} (\Phi (-{1\over 2}Z^2_2))\nonumber\\
&=&  {\prod \theta(u_i)\over \theta(\prod u_{i})}
({1\over 2} Z(u_1 u_3)^2-{1\over 2} Z({1\over u_2})^2+{1\over 2} Z({1\over u_1 u_2})^2-{1\over 2} Z( u_3)^2)\nonumber\\
&&+{\prod \theta(u_i)\over \theta(\prod u_{i})}
{1\over 2}\left(
\wp(u_2)-\wp(u_1 u_3)+\wp(u_3)-\wp(u_1 u_2))\right)\nonumber\\
&=&^{\eqref{eqnhomogeneousadditionformula}}  {\prod \theta(u_i)\over \theta(\prod u_{i})}
({1\over 2} \mathbold{B}_2(u_1 u_3)+{1\over 2}\mathbold{B}_2({u_1 u_2})
+{1\over 2} \mathbold{B}_2({u_2 u_3})
+Z_2 Z_{3}-Z_{2} Z_{23}-Z_{3}Z_{23})\,,\\
&&R^{2}_{3}R^{1}_{3} (\Phi Z_1 Z_2)\nonumber\\
&=& {\prod \theta(u_i)\over \theta(\prod u_{i})}
(Z({u_1 })Z({u_1 u_2 })+Z({u_1 })Z(u_3)
+Z({u_3 })Z({u_1 u_3 })+Z({u_2 })Z({ u_3 }))\nonumber\\
&&-
   {\prod \theta(u_i)\over \theta(\prod u_{i})} Z_{1}
\left(
Z_{21}-Z_{2}-Z_{13}+Z_{3}\right)\nonumber\\
&=& {\prod \theta(u_i)\over \theta(\prod u_{i})}
\left(
Z({u_3 })Z({u_1 u_3 })+Z({u_2 })Z({ u_3 })+Z({u_1 })Z({ u_2 })+Z({u_1})Z({u_1 u_3 })
\right)\,.
\end{eqnarray}
Here we have used the short-hand notation that $Z_{k}=Z(u_{k}), Z_{k\ell}=Z(u_{k}u_{\ell})$.
Putting all these terms together, we have
\[
\int_{A_{[n]}}\Phi=
{\prod \theta(u_i)\over \theta(\prod u_{i})}
({1\over 2} \mathbold{B}_2(u_1 u_3)+{1\over 2} \mathbold{B}_2({u_1 u_2})
+{1\over 2}\mathbold{B}_2({u_2 u_3})
+Z_1 Z_{3} +Z_2 Z_{3}+Z_3 Z_{1})\,.
\]
This
matches the result \eqref{eqnT3explicitcomputation}
and
Theorem \ref{thmTngwintermsofB}, according to
\eqref{eqnwLaurentcoefficientofS2} and \eqref{eqnFourierofZ}.
Similarly, the regularized integral assembles a similar form as above, and is the elliptic completion of the above ordered $A$-cycle integral.

\end{appendices}

\bibliographystyle{amsalpha}

\begin{thebibliography}{BEK06}

%%%%
\iffalse
\bibitem[Arn69]{Arnol1969The}
V.~I. Arnold, \emph{The cohomology ring of the colored braid group},
  Mathematical Notes of the Academy of Ences of the Ussr \textbf{5} (1969),
  no.~2, 138--140.

\bibitem[AS93]{axelrod1993chern}
S.~Axelrod and I.~M. Singer, \emph{Chern--{S}imons {P}erturbation {T}heory
  {II}}, Journal of Differential Geometry \textbf{39} (1993),
  no.~hep-th/9304087, 173--213.

  \bibitem[AT97]{Aldrovandi:1997}
 E.~Aldrovandi and L.~Takhtajan, \emph{Generating Functional in CFT and Effective Action for Two-Dimensional Quantum Gravity on Higher Genus Riemann Surfaces}. Comm Math Phys 188, 29--67 (1997). %https://doi.org/10.1007/s002200050156
\fi
%%%%


\bibitem[BBBM17]{Boehm:2014tropical}
J.~ B{\"o}hm, K.~Bringmann, A.~Buchholz, and H.~Markwig,
  \emph{Tropical mirror symmetry for elliptic curves}, Journal f{\"u}r die
  reine und angewandte Mathematik (Crelles Journal), \textbf{732} (2017), 211--246.
%    \emph{\ \ \ \ Erratum to Tropical mirror symmetry for elliptic curves (J. reine angew. Math. 732 (2017), 211--246)}, Journal f{\"u}r die
%  reine und angewandte Mathematik (Crelles Journal), \textbf{760} (2020), 163--164.

%\bibitem[BCOV94]{Bershadsky:1993cx}
%M.~Bershadsky, S.~Cecotti, H.~Ooguri, and C.~Vafa, \emph{{Kodaira-Spencer
%  theory of gravity and exact results for quantum string amplitudes}},
%Commun. Math. Phys. \textbf{165} (1994), 311--428.





\bibitem[BL13]{Birkenhake:2013complex}
C. Birkenhake and H. Lange, \emph{Complex abelian varieties}, vol.
  \textbf{302}, Springer Science \& Business Media, 2013.

\bibitem[BL11]{Brown:2011}
F.~Brown and A.~Levin, \emph{Multiple Elliptic Polylogarithms}, arXiv:1110.6917 [math.NT].

%%%%
\iffalse
\bibitem[BEK06]{Bloch:2006motives}
S.~Bloch, H.~Esnault, and D.~Kreimer, \emph{On motives
  associated to graph polynomials}, Communications in Mathematical Physics
  \textbf{267} (2006), no.~1, 181--225.

\bibitem[BK08]{Bloch:2008mixed}
S.~Bloch and D.~Kreimer, \emph{Mixed hodge structures and
  renormalization in physics}, Communications in Number Theory and Physics
  \textbf{2} (2008), no.~4, 637--718.

\bibitem[Blo07]{Bloch:2007motives}
S.~Bloch, \emph{Motives associated to graphs}, Japanese Journal of
  Mathematics \textbf{2} (2007), no.~1, 165--196.

\bibitem[Blo08]{Bloch:2008motives}
\bysame, \emph{Motives associated to sums of graphs},
The Geometry of Algebraic Cycles--Proceedings of the Conference on Algebraic Cycles, Columbus, Ohio, 2008.
Clay Mathematics Proceedings Volume \textbf{9}, 2010, 137--143.

\bibitem[{Blo}15]{Bloch:2015}
S.~{Bloch}, \emph{{Feynman Amplitudes in Mathematics and Physics}},  	arXiv:1509.00361 [math.AG].


\bibitem[BSV20]{Benini:2020}
M.~Benini, A. ~Schenkel and B.~Vicedo, \emph{Homotopical analysis of 4d Chern-Simons theory and integrable field theories}, arXiv:2008.01829 [hep-th].


\bibitem[CK98]{Connes:1998}
A.~Connes and D.~Kreimer, \emph{{Hopf Algebras, Renormalization and Noncommutative Geometry.}}, Comm Math Phys 199, 203--242 (1998).


\bibitem[CK00]{Connes:2000renormalization}
A.~Connes and D.~Kreimer, \emph{Renormalization in quantum field theory
  and the Riemann--Hilbert problem i: The Hopf algebra structure of graphs and
  the main theorem}, Communications in Mathematical Physics \textbf{210}
  (2000), no.~1, 249--273.

\bibitem[C11]{Costello:2011book}
K.~J. Costello, \emph{{Renormalization and effective field theory}}, Mathematical Surveys and Monographs, Volume 170,
American Mathematical Society (2011).




\bibitem[CL12]{Costello:2012cy}
K.~J. Costello and S.~Li, \emph{{Quantum BCOV theory on Calabi-Yau manifolds
  and the higher genus B-model}}, arXiv:1201.4501[math.QA].

\bibitem[CMSZ20]{Chen:2020masur}
D.~Chen, M.~M{\"o}ller, A.~Sauvaget, and D.~Zagier,
  \emph{Masur--Veech volumes and intersection theory on moduli spaces of
  abelian differentials}, Inventiones mathematicae (2020).

\bibitem[CMZ18]{Chen:2018quasimodularity}
D.~Chen, M.~M{\"o}ller, and D.~Zagier, \emph{Quasimodularity and large
  genus limits of siegel-veech constants}, Journal of the American Mathematical
  Society \textbf{31} (2018), no.~4, 1059--1163.

\bibitem[Dem12]{Demailly:2012complex}
J.~P.~ Demailly, \emph{Complex analytic and differential geometry}. 2007.
\fi
%%%%%




\bibitem[BO00]{Bloch:2000}
S.~{Bloch} and A.~{Okounkov}, \emph{{The Character of the Infinite
  Wedge Representation}}, Advances in Mathematics \textbf{149} (2000), 1--60.

\bibitem[Bon12]{Bona:2012}
M.~B\'ona,
\emph{Combinatorics of Permutations, second edition}, Discrete Mathematics and Its Applications (Boca Raton), CRC Press, Boca Raton, FL, 2012. With a foreword by Richard Stanley. % MR2919720






\bibitem[Dij95]{Dijkgraaf:1995}
R.~Dijkgraaf, \emph{Mirror symmetry and elliptic curves}, The moduli space
  of curves ({T}exel {I}sland, 1994), Progr. Math., vol. \textbf{129}, Birkh\"auser
  Boston, Boston, MA, 1995, pp.~149--163.

%\bibitem[Dij97]{Dijkgraaf:1997chiral}
%\bysame, \emph{Chiral deformations of conformal field theories}, Nuclear
%  physics B \textbf{493} (1997), no.~3, 588--612.


\bibitem[DJM00]{DJM00}
T.~Miwa, E.~Date and M.~Jimbo,
\emph{Solitons: Differential Equations, Symmetries and Infinite Dimensional Algebras}, Cambridge University Press.



%\bibitem[Dou95]{Douglas:1995conformal}
%M.~Douglas, \emph{Conformal Field Theory Techniques in Large N Yang-Mills Theory}, In: Baulieu L., Dotsenko V., Kazakov V., Windey P. (eds) Quantum Field Theory and String Theory. NATO ASI Series (Series B: Physics), vol \textbf{328}. Springer, Boston, MA, 1995.


%\bibitem[DMZ12]{Dabholkar:2012}
%A.~{Dabholkar}, S.~{Murthy}, and D.~{Zagier}, \emph{{Quantum Black Holes, Wall
%  Crossing, and Mock Modular Forms}}, arXiv:1208.4074 [hep-th].

%\bibitem[EH21a]{Ekeren:2021a}
%J.~van Ekeren and R.~Heluani,
%\emph{ Chiral Homology of Elliptic Curves and the Zhu Algebra}. Commun. Math. Phys. \textbf{386}, 495--550 (2021).
%https://doi.org/10.1007/s00220-021-04026-w

%\bibitem[EH21b]{Ekeren:2021b}
%J.~van Ekeren and R.~Heluani,
%\emph{The First Chiral Homology Group}. arXiv:2103.06322 [math.QA].

%\bibitem[EO01]{Eskin:2001asymptotics}
%A.~Eskin and A.~Okounkov, \emph{Asymptotics of numbers of branched
 % coverings of a torus and volumes of moduli spaces of holomorphic
 % differentials}, Inventiones mathematicae \textbf{145} (2001), no.~1, 59--103.


\bibitem[EO06]{Eskin:2006}
A.~Eskin and A.~ Okounkov, \emph{ Pillowcases and quasimodular forms}. In: Ginzburg V. (eds) Algebraic Geometry and Number Theory. Progress in Mathematics, vol \textbf{253}. Birkh\"auser Boston.





\bibitem[EZ85]{Eichler:1985}
M.~Eicher and D.~Zagier, \emph{The theory of Jacobi forms}, Progress in
  Math. \textbf{55}, Birkh\"auser Boston, Boston, MA, 1985.



%\bibitem[FM11]{Frabetti:2011}
%A.~Frabetti, D.~Manchon, \emph{Five interpretations of Fa\`a di Bruno's formula}. Dyson--Schwinger Equations and Fa\`a di Bruno Hopf Algebras in Physics and Combinatorics, Jun 2011,
%Strasbourg, France. pp. 5--65.% ffhal-00950525f




%%%%%
\iffalse
\bibitem[FMS97]{Francesco:1997}
P.~Francesco, P.~Mathieu and
D.~S\'en\'echal, \emph{Conformal Field Theory}, Graduate Texts in Contemporary Physics, Springer-Verlag New York, 1997.






\bibitem[GH14]{Griffiths:2014principles}
P.~Griffiths and J.~Harris, \emph{Principles of algebraic geometry},
  John Wiley \& Sons, 2014.

\bibitem[GJ94]{getzler1994operads}
E.~Getzler and J.~D.~S.~ Jones, \emph{Operads, homotopy algebra and iterated
  integrals for double loop spaces}, arXiv 9403055[hep-th].
\fi
%%%%%




\bibitem[GM20]{Goujard:2016counting}
E.~Goujard and M.~M{\"o}ller, \emph{Counting Feynman-like graphs:
  Quasimodularity and Siegel-Veech weight}, J. Eur. Math. Soc., \textbf{22} (2020), No 2,  pp. 365--412.



%%%%
\iffalse
\bibitem[Kac98]{Kac:1998}
V.~Kac, \emph{Vertex algebras for beginners}, University
Lecture Series 10, American Mathematical Society, Providence, RI
(1998).





\bibitem[Kat76]{Katz:1976p}
N.~M.~Katz, \emph{p-adic interpolation of real analytic Eisenstein
  series}, Annals of Mathematics (1976), 459--571.



\bibitem[KNTY88]{Kawamoto:1988}
N.~Kawamoto, Y.~Namikawa, A.~Tsuchiya and Y.~Yamada \emph{Geometric realization of conformal field theory on Riemann surfaces}, Commun. Math. Phys. 116, 247--308 (1988).




\bibitem[Kon93]{Kontsevich:1993formal}
M.~Kontsevich, \emph{Formal (Non)-Commutative Symplectic Geometry},  In: Gelfand I.M., Corwin L., Lepowsky J. (eds) The Gelfand Mathematical Seminars, 1990--1992. Birkh\"auser, Boston, MA.


\bibitem[Kon94]{Kontsevich:1994feynman}
M.~Kontsevich, \emph{Feynman diagrams and low-dimensional topology}, First
  European Congress of Mathematics Paris, July 6--10, 1992, Springer, 1994,
  pp.~97--121.



\bibitem[Kon03]{kontsevich2003deformation}
\bysame, \emph{Deformation quantization of Poisson manifolds}, Letters in
  Mathematical Physics \textbf{66} (2003), no.~3, 157--216.

\fi
%%%%%

\bibitem[Get99]{Getzler:1999}
E. Getzler, \emph{Resolving mixed Hodge modules on configuration spaces}. Duke Mathematical Journal.  \textbf{96}(1):175--203 (1999). %https://doi.org/10.1215/S0012-7094-99-09605-9




%\bibitem[GL21]{Gui-Li2021}
%Z.~Gui and S.~Li, \emph{{Elliptic Trace Map on Chiral Algebras}}, arXiv:2112.14572 [math.QA].



%\bibitem[Hac20]{Hackl:2020}
%B. Hackl, \emph{A combinatorial identity for rooted labeled forests}. Aequat. Math. 94, 253--257 (2020). % https://doi.org/10.1007/s00010-019-00662-9



\bibitem[KNTY88]{Kawamoto:1988}
N~.Kawamoto, Y.~Namikawa, A.~Tsuchiya and
Y.~Yamada,
\emph{Geometric Realization of Conformal Field Theory on Riemann Surfaces},
Commun. Math. Phys. \textbf{116},247-308 (1988).


\bibitem[Kri94]{Kriz:1994}
I.~Kriz, \emph{On the Rational Homotopy Type of Configuration Spaces}. Annals of Mathematics \textbf{139}, no. 2 (1994): 227–237. %https://doi.org/10.2307/2946581.

\bibitem[KZ95]{Kaneko:1995}
M.~Kaneko and D.~Zagier, \emph{A generalized {J}acobi theta function and
  quasimodular forms}, The moduli space of curves ({T}exel {I}sland, 1994),
  Progr. Math., vol. \textbf{129}, Birkh\"auser Boston, Boston, MA, 1995, 165--172.



%\bibitem[Lan85]{Lang:1985}
%S.~Lang, \emph{Complex Analysis}, Graduate Texts in Mathematics, vol \textbf{103}. Springer, New York, NY, 1985.



%\cite{Li:2020tex}
%\bibitem[Li20]{Li:2020}
%J.~Li, Y.~Shen and J.~Zhou,
%``Higher Genus FJRW Invariants of a Fermat Cubic,''
%arXiv:2001.00343 [math.AG].



\bibitem[Li11]{Li:2011BCOV}
S.~Li, \emph{{BCOV theory on the elliptic curve and higher genus mirror symmetry
}},
arXiv:1112.4063 [math.QA].


\bibitem[Li12]{Li:2011mi}
S.~Li, \emph{{Feynman Graph Integrals and Almost Modular Forms}},
  Commun. Number Theory Phys., \textbf{6} (2012), 129--157.

\bibitem[Li16]{Li:2016vertex}
S.~Li, \emph{Vertex algebras and quantum master equation}, J. Differential Geom. \textbf{123} (3), 2023, 461--521.% arXiv:1612.01292[math.QA].

\bibitem[LZ21]{Li:2020regularized}
S.~Li and J.~Zhou, \emph{Regularized Integrals on Riemann Surfaces and Modular Forms},
Commun. Math. Phys. \textbf{388} (2021), 1403--1474.

\bibitem[LZ23]{Li:2022regularized}
S.~Li and J.~Zhou, \emph{Regularized Integrals on Elliptic Curves and Holomorphic Anomaly Equations},
Commun. Math. Phys. \textbf{401} (2023), 613--645.
%arXiv:2205.14562 [math.DG].


\bibitem[{Lib}09]{Libgober:2009}
A.~{Libgober}, \emph{{Elliptic genera, real algebraic varieties and
  quasi-Jacobi forms}}, arXiv:0904.1026 [math.AG].


%%%%
\iffalse
\bibitem[Mar09]{Marcolli:2009feynman}
M.~Marcolli, \emph{Feynman integrals and motives}, European Congress of Mathematics, 2009, vol \textbf{313}, No. 1, pp.~293--332.




\bibitem[NT04]{Nikolov:2004}
N.~M.~Nikolov and I.~T. Todorov, \emph{Lectures on Elliptic Functions and Modular Forms in Conformal Field Theory}, arXiv:math-ph/0412039.



%%%%%

\fi


\bibitem[Mil03]{Milas:2003}
A.~Milas, \emph{Formal differential operators, vertex operator algebras and zeta values II},
Journal of Pure and Applied Algebra. \textbf{183}, 191--244 (2003).



%\bibitem[Mil22]{Milas:2022}
%A.~Milas, \emph{Generalized multiple $q$-Zeta values and characters of vertex algebras},
%arXiv:2203.15642[QA].



\bibitem[Oko01]{Okounkov:2001}
A.~Okounkov, \emph{ Infinite wedge and random partitions}. Sel. math., New ser. \textbf{7}, 57 (2001). %https://doi.org/10.1007/PL00001398


\bibitem[OP06a]{Okounkov:2006}
A.~Okounkov and R. ~Pandharipande,
{\em Gromov-Witten theory, Hurwitz theory, and completed cycles},
Ann. of Math. (\textbf{2}) 163 (2006), no. 2, 517--560.


\bibitem[OP06b]{Okounkov:2006b}
A.~Okounkov and R. ~Pandharipande,
{\em The equivariant Gromov-Witten
	theory of $\mathbb{P}^1$},
Ann. of Math. (\textbf{2}) 163 (2006), no. 2, 561--605.

%\bibitem[OP06c]{Okounkov:Virasoro}
%Okounkov, A.; Pandharipande, R.
%{\em Virasoro constraints for target curves. }
%Invent. Math. \textbf{163} (2006), no. 1, 47--108.

%\bibitem[OP18]{Oberdieck:2018}
%G.~Oberdieck and A.~Pixton, \emph{Holomorphic anomaly equations and the Igusa cusp form conjecture}, Invent. math. \textbf{213}, 507--587 (2018).
%https://doi.org/10.1007/s00222-018-0794-0



%\bibitem[Pet05]{Petersen:2008}
%T.~Petersen, \emph{Cyclic Descents and P-Partitions}. J Algebr Comb \textbf{22}, 343--375 (2005). %https://doi.org/10.1007/s10801-005-4532-5



%\bibitem[Pit05]{Pitman:2002}
%J.~Pitman, \emph{Forest Volume Decompositions and Abel-Cayley-Hurwitz Multinomial Expansions}, Journal of Combinatorial Theory, Series A,
%Volume \textbf{98}, Issue 1, 2002, Pages 175--191.


\bibitem[Pas88]{Passare:1988}
M.~Passare,
\emph{A calculus for meromorphic currents},
Journal f\"ur die reine und angewandte Mathematik , Volume \textbf{32} (1988), 37--56.

\bibitem[Pix08]{Pixton:2008}
A.~Pixton, \emph{The Gromov-Witten Theory of an Elliptic Curve and Quasimodular Forms}. Senior thesis. Princeton University. 2008.


\bibitem[Rai89]{Raina:1989}
A.~Raina,
\emph{Fay's Trisecant Identity and Conformal Field Theory},
Commun. Math. Phys. \textbf{122}, 625-641 (1989).




%\bibitem[Rud94]{rudd1994string}
%R.~Rudd, \emph{The string partition function for QCD on the torus}, arXiv: 9407176[hep-th].

\bibitem[RY09]{Roth:2009mirror}
M.~Roth and N.~Yui, \emph{Mirror symmetry for elliptic curves: the
  B-model (bosonic) counting}, preprint.

%\bibitem[RZZ19]{Ruan:2019}
%Y.~Ruan, Y.~Zhang and J.~Zhou, \emph{Genus Two Quasi-Siegel Modular Forms and Gromov-Witten Theory of Toric Calabi-Yau Threefolds}, arxiv:1911.07204 [math.AG].




%%%%
\iffalse
\bibitem[RY10]{Roth:2010mirror}
\bysame, \emph{Mirror symmetry for elliptic curves: the {A}-model (fermionic)
  counting}, Clay Math. Proc \textbf{12} (2010), 245--283.


\bibitem[Siv09]{Silverman:2009}
J.~H.~Silverman, \emph{The Arithmetic of Elliptic Curves}, Graduate Texts in Mathematics, vol \textbf{106}. Springer, New York, NY, 2009.

\bibitem[Sol68]{Solomon:1968}
L.~Solomon, \emph{On the Poincar\'e-Birkhoff-Witt theorem}, Journal of Combinatorial Theory, vol \textbf{4}, issue 6, 1968, 363--375.







  \bibitem[Tak01]{Takhtajan:2001}
 L.~A.~Takhtajan, \emph{Free Bosons and Tau-Functions for Compact Riemann Surfaces and Closed Smooth Jordan Curves. Current Correlation Functions},
 Letters in Mathematical Physics \textbf{56} (2001), 181--228.



 \bibitem[Tyu78]{Tyurin:1978}
A.~Tyurin, \emph{On periods of quadratic differentials}, Russian mathematical
surveys \textbf{33} (1978), 169--221.
\fi
%%%


%\bibitem[Sha07]{Shareshian:2007}
%J.~Shareshian and M.~Wachs,
%\emph{$q$-Eulerian polynomials: excedance number and major index}, Electron. Res. Announc. Amer. Math. Soc. \textbf{13} (2007), 33--45.
% DOI 10.1090/S1079-6762-07-00172-2. MR2300004 MR2300004


%\bibitem[Sil09]{Silverman:2009arithmetic}
%Joseph~H Silverman,
%{\em The arithmetic of elliptic curves, }
%vol. 106,  Springer Science \& Business Media, 2009.




\bibitem[Tot96]{Totaro:1996}
B.~Totaro
{\em Configuration spaces of algebraic varieties},
Topology, Volume \textbf{35}, Issue 4, October 1996, Pages 1057--1067.




%\bibitem[Urb14]{Urban:2014nearly}
%E.~Urban, \emph{Nearly overconvergent modular forms}, In: Bouganis T., Venjakob O. (eds) Iwasawa Theory 2012, 401--441. Contributions in Mathematical and Computational Sciences, vol \textbf{7}. Springer, Berlin, Heidelberg.


\bibitem[Wei76]{Weil:1976}
A.~Weil, \emph{Elliptic functions according to Eisenstein and Kronecker},
Ergebnisse der Mathematik bind \textbf{88}, Springer (1976).





%%%%
\iffalse




\bibitem[Zag16]{Zagier:2016partitions}
\bysame, \emph{Partitions, quasimodular forms, and the Bloch--Okounkov
  theorem}, The Ramanujan Journal \textbf{41} (2016), no.~1-3, 345--368.

\fi
%%%%


\bibitem[Zag91]{Zagier:1991}
D.~Zagier, \emph{Periods of modular forms and Jacobi theta functions},
  Inventiones mathematicae \textbf{104} (1991), no.~1, 449--465.


%\bibitem[Zag08]{Zagier:2008}
%\bysame, \emph{Elliptic modular forms and their applications}, The 1-2-3 of
 % modular forms, Universitext, Springer, Berlin, 2008, pp.~1--103.

\bibitem[Zag16]{Zagier:2016partitions}
\bysame, \emph{Partitions, quasimodular forms, and the Bloch--Okounkov
  theorem}, The Ramanujan Journal \textbf{41} (2016), no.~1-3, 345--368.




\bibitem[Zho23]{Zhou:2023cohomologicalpairings}
\bysame, \emph{Regularized Integrals on Configuration Spaces of Riemann Surfaces and Cohomological Pairings}, arXiv:2305.12362 [math.AG].


\end{thebibliography}

\providecommand{\bysame}{\leavevmode\hbox to3em{\hrulefill}\thinspace}
\providecommand{\MR}{\relax\ifhmode\unskip\space\fi MR }
\providecommand{\MRhref}[2]{%
  \href{http://www.ams.org/mathscinet-getitem?mr=#1}{#2}
}
\providecommand{\href}[2]{#2}

\bigskip{}

\noindent{\small Yau Mathematical Sciences Center, Tsinghua University, Beijing 100084, P. R. China}

%\medskip{}
%\noindent{\small Yau Mathematical Sciences Center, Tsinghua University, Beijing 100084, P. R. China}

\noindent{\small Email: \tt jzhou2018@mail.tsinghua.edu.cn}

\end{document}